\documentclass[12pt]{article}

\usepackage{amsmath, amsthm, amssymb}
\usepackage{enumerate}
\usepackage{pdflscape}
\usepackage{caption}

\usepackage{ifpdf}
\ifpdf
\usepackage[pdftex]{graphicx}
\else
\usepackage[dvips]{graphicx}
\fi
\usepackage{tikz}
 	 \usetikzlibrary{arrows,backgrounds}
\usepackage[all]{xy}

\usepackage{multicol}

\usepackage{tocvsec2}

\input xy
\xyoption{all}

\usepackage[pdftex,plainpages=false,hypertexnames=false,pdfpagelabels]{hyperref}
\newcommand{\arxiv}[1]{\href{http://arxiv.org/abs/#1}{\tt arXiv:\nolinkurl{#1}}}
\newcommand{\arXiv}[1]{\href{http://arxiv.org/abs/#1}{\tt arXiv:\nolinkurl{#1}}}
\newcommand{\doi}[1]{\href{http://dx.doi.org/#1}{{\tt DOI:#1}}}

\newcommand{\googlebooks}[1]{(preview at \href{http://books.google.com/books?id=#1}{google books})}

\usepackage{xcolor}
\definecolor{medium-blue}{rgb}{0,0,.8}
\hypersetup{
   colorlinks, linkcolor={purple},
   citecolor={medium-blue}, urlcolor={medium-blue}
}

\usepackage{longtable}
\usepackage{fullpage}
\theoremstyle{plain}
\newtheorem{thm}{Theorem}[section]
\newtheorem*{thm*}{Theorem}
\newtheorem{thmalpha}{Theorem}

\newtheorem*{cor*}{Corollary}
\newtheorem{conj}[thm]{Conjecture}

\newtheorem*{conj*}{Conjecture}
\newtheorem{lem}[thm]{Lemma}

\newtheorem{prop}[thm]{Proposition}
\newtheorem{quest}[thm]{Question}
\newtheorem*{quest*}{Question}
\newtheorem*{claim*}{Claim}

\theoremstyle{definition}
\newtheorem{defn}[thm]{Definition}

\newtheorem{nota}[thm]{Notation}

\newtheorem{ex}[thm]{Example}
\newtheorem{sub-ex}[thm]{Sub-Example}
\newtheorem{rem}[thm]{Remark}
\newtheorem*{rem*}{Remark}

\DeclareMathOperator{\Bim}{Bim}

\DeclareMathOperator{\End}{End}
\DeclareMathOperator{\ev}{ev}
\DeclareMathOperator{\Hom}{Hom}

\DeclareMathOperator{\Ob}{Ob}
\DeclareMathOperator{\op}{op}

\DeclareMathOperator{\id}{id}

\DeclareMathOperator{\Irr}{Irr}

\newcommand{\comment}[1]{\textcolor{red}{[stuff commented out]}}

\newcommand{\be}{\begin{enumerate}[label=(\arabic*)]}
\newcommand{\ee}{\end{enumerate}}

\newcommand{\N}{\mathbb{N}}

\newcommand{\C}{\mathbb{C}}

\newcommand{\n}{\mathfrak{n}}

\def\semicolon{;}
\def\applytolist#1{
    \expandafter\def\csname multi#1\endcsname##1{
        \def\multiack{##1}\ifx\multiack\semicolon
            \def\next{\relax}
        \else
            \csname #1\endcsname{##1}
            \def\next{\csname multi#1\endcsname}
        \fi
        \next}
    \csname multi#1\endcsname}

\def\calc#1{\expandafter\def\csname c#1\endcsname{{\mathcal #1}}}
\applytolist{calc}QWERTYUIOPLKJHGFDSAZXCVBNM;
\def\bbc#1{\expandafter\def\csname bb#1\endcsname{{\mathbb #1}}}
\applytolist{bbc}QWERTYUIOPLKJHGFDSAZXCVBNM;
\def\bfc#1{\expandafter\def\csname bf#1\endcsname{{\mathbf #1}}}
\applytolist{bfc}QWERTYUIOPLKJHGFDSAZXCVBNM;
\def\sfc#1{\expandafter\def\csname s#1\endcsname{{\sf #1}}}
\applytolist{sfc}QWERTYUIOPLKJHGFDSAZXCVBNM;

\newcommand{\Rep}{{\sf Rep}}

\renewcommand{\Vec}{{\sf Vec}}
\newcommand{\Hilb}{{\sf Hilb}}

\newcommand{\noshow}[1]{}
\newcommand{\MR}[1]{}

\usepackage[utf8]{inputenc}
\DeclareUnicodeCharacter{2693}{\anchor}
\usepackage{dingbat}

\usetikzlibrary{shapes}
\usetikzlibrary{backgrounds}
\usetikzlibrary{decorations,decorations.pathreplacing,decorations.markings}
\usetikzlibrary{fit,calc,through}
\usetikzlibrary{external}
\tikzset{
	super thick/.style={line width=3pt}
}
\tikzstyle{shaded}=[fill=red!10!blue!20!gray!30!white]
\tikzstyle{unshaded}=[fill=white]
\tikzstyle{empty box}=[circle, draw, thick, fill=white, opaque, inner sep=2mm]
\tikzstyle{annular}=[scale=.7, inner sep=1mm, baseline]
\tikzstyle{rectangular}=[scale=.75, inner sep=1mm, baseline=-.1cm]
\tikzstyle{mid>}=[decoration={markings, mark=at position 0.5 with {\arrow{>}}}, postaction={decorate}]
\tikzstyle{mid<}=[decoration={markings, mark=at position 0.5 with {\arrow{<}}}, postaction={decorate}]
\tikzstyle{over}=[double, draw=white, super thick, double=]

\newcommand{\roundNbox}[6]{
	\draw[rounded corners=5pt, very thick, #1] ($#2+(-#3,-#3)+(-#4,0)$) rectangle ($#2+(#3,#3)+(#5,0)$);
	\coordinate (ZZa) at ($#2+(-#4,0)$);
	\coordinate (ZZb) at ($#2+(#5,0)$);
	\node at ($1/2*(ZZa)+1/2*(ZZb)$) {#6};
}

\newcommand{\halfDottedEllipse}[3]{
	\draw[thick] #1 arc(-180:0:{#2} and {#3});
	\draw[thick, dotted] ($ #1 + 2*(#2,0)$) arc(0:180:{#2} and {#3});
}

  \newcommand{\tikzmath}[2][]
     {\vcenter{\hbox{\begin{tikzpicture}[#1]#2
                     \end{tikzpicture}}}
     }

\newcommand{\alphacolor}{blue}
\newcommand{\betacolor}{orange}

\begin{document}
\title{Bicommutant categories from fusion categories}
\author{Andr\'{e} Henriques and David Penneys}
\date{\today}
\maketitle
\begin{abstract}
Bicommutant categories are higher categorical analogs of von Neumann algebras that were recently introduced by the first author.
In this article, we prove that every unitary fusion category gives an example of a bicommutant category.
This theorem categorifies the well known result according to which a finite dimensional $*$-algebra that can be faithfully represented on a Hilbert space is in fact a von Neumann algebra.
\end{abstract}

%%%%%%%%%%%%%%%%%%%%%%%%%%%%%%%%%%%%%%%%%%%%%%%%%
%%%%%%%%%%%%%%%%%%%%%%%%%%%%%%%%%%%%%%%%%%%%%%%%%
%%%%%%%%%%%%%%%%%%%%%%%%%%%%%%%%%%%%%%%%%%%%%%%%%
\tableofcontents

\section{Introduction}
%\comment{

Bicommutant categories were introduced by the first author in the recent preprint \cite{1503.06254}%What Chern-Simons theory assigns to a point
, as a categorification of the notion of a von Neumann algebra.

Recall that a von Neumann algebra is a subalgebra
of the algebra of bounded operators on a Hilbert space which is equal to its bicommutant:
\[
\qquad\qquad\qquad A\subset B(H)\,\,\quad\mathrm{s.t.}\quad A=A''\qquad\qquad\mathrm{(von\,Neumann\,algebra).}
\]
Bicommutant categories are defined similarly.
They are tensor categories equipped with a tensor functor to the category $\Bim(R)$ of all separable bimodules over a hyperfinite factor, such that the natural comparison functor from the category to its bicommutant is an equivalence of categories:
\[
\qquad\qquad\qquad 
\cC\to \Bim(R)\,\quad\mathrm{s.t.}\quad
\cC\stackrel\simeq\to\cC''
\qquad\qquad\mathrm{(bicommutant\,category).}
\]

The main result of this paper is that every unitary fusion category gives an example of a bicommutant category.
The fusion categories themselves are not bicommutant categories, as they do not admit infinite direct sums:
in a fusion category, every object is a \emph{finite} direct sum of simple objects.
In other words, every object is of the form $\bigoplus_i c_i\otimes V_i$
for some finite dimensional vector spaces $V_i\in\mathsf{Vec}$ and simple objects $c_i\in\cC$.
In order to make $\cC$ into a bicommutant category, we need to allow the $V_i$ to be arbitrary separable Hilbert spaces.
The resulting category is denoted $\cC\otimes_\Vec\Hilb$ (this is an instance of balanced tensor product of linear categories \cite{MR1843311}%Tambara, A duality for modules over monoidal categories...
). Our main result is:

\begin{thmalpha}
\label{thm:Main}
If $\cC$ is a unitary fusion category, then $\cC\otimes_\Vec\Hilb$ is a bicommutant category.
\end{thmalpha}

\noindent
By a result of Popa \cite{MR1334479}, every unitary fusion category $\cC$ can be embedded in $\Bim(R)$ (see Theorem~\ref{thm: exists faithful rep}).
We prove that its bicommutant $\cC''$ is equivalent to $\cC\otimes_\Vec \Hilb$, and that the latter is a bicommutant category.

As a special case of the above theorem, if $G$ is a finite group and $\omega$ is a cocycle representing a class $[\omega]\in H^3(G,U(1))$, then the tensor category $\Hilb^\omega[G]$ of $G$-graded Hilbert spaces with associator twisted by $\omega$ is a bicommutant category.
That result was conjectured in \cite[\S6]{1503.06254} %What Chern-Simons theory assigns to a point
as part of a bigger conjecture about categories of representations of twisted loop groups.

We summarize the categorical analogy in the table below.
Going left to right is ``categorification'', and going down is passing to the infinite dimensional case:

\begin{center}
\def\arraystretch{1.2}
\begin{tabular}{|c|c|}
\hline
& \\[-3ex]
an algebra $A$ & a tensor category $\cC$
\\
a finite dimensional algebra & a fusion category
\\
the center of an algebra $Z(A)$ & the Drinfeld center $\cZ(\cC)$
\\
the commutant
(or centralizer) $Z_B(A)$ of $A$ in $B$ & the commutant $\cZ_\cD(\cC)$ of $\cC$ in $\cD$
\\[1ex]
\hline
& \\[-2.7ex]
the algebra $B(H)$ of bounded operators &
the category $\Bim(R)$ of all bimodules\\[-.7ex]
on a Hilbert space & on a hyperfinite factor $R$
\\
the commutant $A':=Z_{B(H)}(A)$ & the commutant $\cC':=\cZ_{\Bim(R)}(\cC)$
\\
a von Neumann algebra $A=A''$ & a bicommutant category $\cC\cong \cC''$
\\[.5ex]
\hline
\end{tabular}
\end{center}

We have omitted one technical point in the above discussion.
Von Neumann algebras are not just algebras; they are $*$-algebras
(all the other structures such as the norm and the various topologies can be deduced from the $*$-algebra structure, but the $*$-algebra cannot be deduced from the algebra structure).
Similarly, bicommutant categories are equipped with two involutions which mimic the involutions that are naturally present on $\Bim(R)$. One of the involutions acts at the level of morphisms (the adjoint of a linear map), and the other acts at the level of objects (the complex conjugate of a bimodule).
We call such categories bi-involutive tensor categories (see Definition \ref{def: bi-involutive category}).
Thus, we add the following line to the above table:

\begin{center}
\def\arraystretch{1.2}
\begin{tabular}{|c|c|}
\hline
& \\[-3ex]
\hspace{3.2cm} $*$-algebra $A$ \hspace{3.2cm} & \hspace{.5cm} bi-involutive tensor category $\cC$ \hspace{.5cm}
\\[.5ex]
\hline
\end{tabular}
\end{center}

%%%%%%%%%%%%%%%%%%%%%%%%%%%%%%%%%%%%%%%%%%%%%%%%%%%%
\subsection{Acknowledgements}
This project began at the 2015 Mathematisches Forschungsinstitut Oberwolfach workshop on Subfactors and conformal field theory.
The authors would like to thank the organizers and MFO for their hospitality.
Andr\'e Henriques was supported by the Leverhulme trust and the EPSRC grant ``Quantum Mathematics and Computation'' during his stay in Oxford.
David Penneys was partially supported by an AMS-Simons travel grant and NSF DMS grant 1500387.

%%%%%%%%%%%%%%%%%%%%%%%%%%%%%%%%%%%%%%%%%%%%%%%%%
%%%%%%%%%%%%%%%%%%%%%%%%%%%%%%%%%%%%%%%%%%%%%%%%%
%%%%%%%%%%%%%%%%%%%%%%%%%%%%%%%%%%%%%%%%%%%%%%%%%
\section{Preliminaries}

%%%%%%%%%%%%%%%%%%%%%%%%%%%%%%%%%%%%%%%%%%%%%%%%%
\subsection{Involutions on tensor categories}
\label{sec:invoutions}

A linear \emph{dagger category} is a linear category $\cC$ over the complex numbers, equipped with an antilinear map $\cC(x,y)\to\cC(y,x):f\mapsto f^*$ for every $x,y\in \cC$ called the \emph{adjoint} of a morphism. It satisfies $f^{**}=f$ and $(f\circ g)^*=g^*\circ f^*$, from which it follows that $\id_x^*=\id_x$.
An invertible morphism of a dagger category is called \emph{unitary} if $f^*=f^{-1}$.

A functor $F:\cC\to\cD$ between dagger categories is a \emph{dagger functor} if $F(f)^*=F(f^*)$.

\begin{defn}[{\cite[\S7]{MR2767048}}]\label{def: tensor category}
A \emph{dagger tensor category} %Selinger's survey
is a linear dagger category $\cC$ equipped with a monoidal structure whose associators $\alpha_{x,y,z}:(x\otimes y)\otimes z\to x\otimes (y\otimes z)$ and unitors $\lambda_x:1\otimes x\to x$ and $\rho_x:x\otimes 1\to x$ are unitary, and which satisfies the compatibility condition $(f\otimes g)^* =f^*\otimes g^*$.
\end{defn}

The last condition can be rephrased as saying that the monoidal product $\otimes:\cC\otimes_{\mathsf{Vec}} \cC\to \cC$ is a dagger functor.
From now on, we shall abuse notation, and omit all associators and unitors from our formulas. We trust the reader to insert them wherever needed.

\begin{defn}\label{def: tensor functor}
Let $\cC$ and $\cD$ be dagger tensor categories.
A \emph{dagger tensor functor} $F: \cC\to \cD$ is a dagger functor equipped with a unitary natural transformation $\mu_{x,y}:F(x)\otimes F(y)\to F(x\otimes y)$ and a unitary isomorphism $i:1_\cD\to F(1_\cC)$ such that the following identities hold for all $x,y,z\in\cC$:
\begin{gather*}
\mu_{x,y\otimes z}\circ(\id_{F(x)}\otimes\mu_{y,z})
= 
\mu_{x\otimes y, z}\circ (\mu_{x,y}\otimes \id_{F(z)})
\\
\mu_{1,x}\circ (i\otimes \id_{F(x)})
=\id_{F(x)}
\qquad\,\,\,
\mu_{x,1}\circ(\id_{F(x)}\otimes i)
=
\id_{F(x)}.
\end{gather*}
\end{defn}

We shall be interested in dagger tensor categories which are equipped with a second involution, this time at the level of objects (compare \cite[Def.\,1.3]{MR1749868}%
%Hayashi Tomohiro, Yamagami Shigeru, Amenable tensor categories and their realizations as AFD bimodules. 
):

\begin{defn}
A \emph{bi-involutive tensor category} is a dagger tensor category $\cC$ with a covariant anti-linear dagger functor $\overline{\,\cdot\,}:\cC\to \cC$ called the conjugate.
This functor should be involutive, meaning that for every $x\in\cC$, we are given a unitary natural isomorphisms $\varphi_x:x\to \overline{\overline{x}}$ satisfying $\varphi_{\overline{x}}=\overline{\varphi_{x}}$.
It should be anti-compatible with the tensor structure, meaning that we have unitary natural isomorphisms
\[
\nu_{x,y}:\overline{x} \otimes \overline{y} \stackrel\simeq\longrightarrow \overline{y \otimes x}
\]
and a unitary $j:1\to\overline 1$ satisfying
$\nu_{x,z\otimes y}\circ (\id_{\overline x}\otimes\nu_{y,z})=\nu_{y\otimes x, z}\circ (\nu_{x,y}\otimes\id_{\overline z})$ and $\nu_{1,x}\circ(j\otimes \id_{\overline x})=\id_{\overline x}=\nu_{x,1}\circ(\id_{\overline x}\otimes j)$.
Finally, we require the compatibility conditions $\varphi_1=\overline j\circ j$ and $\varphi_{x \otimes y}=\overline{\nu_{y,x}}\circ\nu_{\overline x,\overline y}\circ(\varphi_x\otimes\varphi_y)$
between the above pieces of data.
\end{defn}

\begin{rem}
It is interesting to note that the map $j$ can be recovered from the other data as $j=
\lambda_{\overline 1}\circ(\varphi^{-1}_1\otimes \id_{\overline 1})\circ\nu_{\overline 11}^{-1}\circ\overline{\lambda_{\overline 1}}^{-1}\circ\varphi_1$. We believe that the notion of bi-involutive category as presented above is equivalent to its variant without $j$ (and without the axioms that involve $j$). Nevertheless, we find it more pleasant to include this piece of data in the definition.
  \end{rem}

Note that in the category of Hilbert spaces,
the isomorphism $\varphi_H:H\to\overline{\overline H}$ is an identity arrow.
%In the examples that we will treat in this paper,
%the involutiveness \dave{sounds awkward} of $\overline{\,\cdot\,}:\cC\to \cC$ will always be strict, meaning that the natural isomorphisms $\varphi_x$ are identity arrows.
Whenever that is the case, we have
$\overline{j}=j^{-1}$ and $\overline{\nu_{y,x}}=\nu_{\overline x,\overline y}^{-1}$.

\begin{defn} \label{def: bi-involutive category}
Let $\cC$ and $\cD$ be bi-involutive tensor categories.
A \emph{bi-involutive tensor functor} is a dagger tensor functor $F:\cC\to\cD$, equipped with a unitary natural transformation $\upsilon_x:F(\overline x)\to \overline{F(x)}$ satisfying the three conditions
$\upsilon_{\overline x}=
\overline{\upsilon_{x}}^{-1}\circ\varphi_{F(x)}\circ F(\varphi_x)^{-1}
$,\,
$\upsilon_{1_\cC}=
\overline i\circ j_\cD\circ i^{-1}\circ F(j_\cC)^{-1}$, and
$\upsilon_{x\otimes y}=
\overline{\mu_{x,y}}\circ\nu_{F(y),F(x)}\circ (\upsilon_y\otimes\upsilon_x)\circ\mu_{\overline y,\overline x}^{-1}\circ F(\nu_{y,x})^{-1}$.
\end{defn}

%%%%%%%%%%%%%%%%%%%%%%%%%%%%%%%%%%%%%%%%%%%%%%%%%
\subsection{Unitary fusion categories}
\label{sec: Unitary fusion categories}

A tensor category $\cC$ is \emph{rigid} if for every object $x\in \cC$ there exists an object $x^\vee\in\cC$, called the dual of $x$, and maps
$\mathrm{ev}_x : x^\vee \otimes x \to 1$
and $\mathrm{coev}_x : 1 \to x \otimes x^\vee$
satisfying the zig-zag axioms
\begin{equation}\label{eq: zig-zag axioms}
(\id_x \otimes\, \mathrm{ev}_x) \circ (\mathrm{coev}_x \otimes \id_x) = \id_x
\quad\mathrm{and}\quad
(\mathrm{ev}_x \otimes \id_{x^\vee}) \circ (\id_{x^\vee} \otimes\, \mathrm{coev}_x) = \id_{x^\vee}
\end{equation}
(those equations determine $x^\vee$ up to unique isomorphism).
Moreover, for every $x \in \cC$, there should exist an object ${}^\vee x \in \cC$ such that $({}^\vee x)^\vee \cong x$.
The dual of a morphism
$f:x\to y$ is given by
%\begin{equation}\label{eq: f^vee}
\[
f^\vee:=(\mathrm{ev}_y\otimes\id_{x^\vee})\circ(\id_{y^\vee}\otimes f \otimes\id_{x^\vee})\circ(\id_{y^\vee}\otimes\,\mathrm{coev}_x):y^\vee\to x^\vee.
\]
%\end{equation}

Let $\mathsf{Vec}$ denote the category of finite dimensional vector spaces.
A category is \emph{semisimple} if it is equivalent to a direct sum of copies of $\mathsf{Vec}$, possibly infinitely many.
Equivalently, it is semisimple if it admits finite direct sums (including the zero sum), and every object is a direct sum of finitely many (possibly zero) simple objects.

\begin{defn}
A fusion category is a tensor category which is  rigid, semisimple, with simple unit, and finitely many isomorphism classes of simple objects.
\end{defn}

Let $\mathsf{Hilb}$ denote the dagger category of Hilbert spaces and bounded linear maps.
A C*-category is a dagger category $\cC$ for which there exists a faithful dagger functor $\cC\to \mathsf{Hilb}$ whose image is norm-closed at the level of hom-spaces.
Equivalently \cite[Prop.\,1.14]{MR808930}, a C*-category is a dagger category such that for every arrow $f:x\to y$ there exists an arrow $g:x\to x$ with $f^*\circ f=g^*\circ g$,\footnote{This condition is present in the original definition
\cite{MR808930} %Ghez, Lima, Roberts, $W^*$-categories
of Ghez, Lima, and Roberts, but is omitted from many other references (e.g. from
\cite{MR2091457} %Yamagami Shigeru, Frobenius duality in C*-tensor categories.
\cite{MR1749868} %Hayashi Tomohiro, Yamagami Shigeru, Amenable tensor categories and their realizations as AFD bimodules. 
\cite{MR1010160}%Doplicher Sergio, Roberts John, A new duality theory for compact groups
). It is automatic for categories that admit direct sums, but it can otherwise fail.
}
and such that
\[
\,\|f\|:=\sup \big\{|\lambda| : f^*{\circ} f-\lambda{\cdot} \id\,\,\, \mathrm{is\,\,not\,\,invertible}\big\}^{1/2}
\]
are complete norms on the hom-spaces which satisfy $\|f\circ g\|\le\|f\|\|g\|$
and $\|f^*\circ f\|=\|f\|^2$.
A C*-tensor category is a dagger tensor category 
whose underlying dagger category is a C*-category.

\begin{defn}
A unitary fusion category is a dagger tensor category whose underlying dagger category is a C*-category, and whose underlying tensor category is a fusion category.
\end{defn}

By \cite[Thm.\,4.7]{MR2091457} %Frobenius duality in C*-tensor categories
and \cite[\S4]{MR3342166}%Dualizability and index of subfactors
, every rigid C*-tensor category with simple unit (in particular, every unitary fusion category) can be equipped with a \emph{canonical bi-involutive structure}.
The conjugation $\overline{\,\cdot\,}$ is characterized at the level of objects (up to unique unitary isomorphisms) by the data of structure morphisms
$\mathrm{ev}_x : \overline{x} \otimes x \to 1$
and $\mathrm{coev}_x : 1 \to x \otimes \overline{x}$,
subject to the two zig-zag axioms (\ref{eq: zig-zag axioms}) and the balancing condition
\begin{equation*}
\qquad\quad\mathrm{coev}_x^*\circ(f\otimes \id_{\overline{x}})\circ\mathrm{coev}_x=\mathrm{ev}_x\circ(\id_{\overline{x}}\otimes f)\circ\mathrm{ev}_x^*
\qquad\quad
\forall\,\,f:x\to x.
\end{equation*}
The conjugation applied to a morphism $f:x\to y$
is given by $\overline{f}:=(f^*)^\vee:\overline{x}\to\overline{y}$.
The coherences between the conjugation and the tensor structure are given by
$j=\mathrm{coev}_1$ and
$\nu_{x,y}=(\mathrm{ev}_x\otimes\id_{\overline{y\otimes x}})\circ(\id_{\overline{x}}\otimes\mathrm{ev}_y\otimes\id_{x\otimes\overline{y\otimes x}})\circ(\id_{\overline x\otimes\overline y}\otimes \mathrm{coev}_{y\otimes x})$.
The last piece of data is provided by the isomorphisms
\[
\varphi_x:=(\id_{\overline{\overline x}}\,\otimes\, \mathrm{ev}_x)\circ
(\mathrm{ev}_{\overline x}^*\otimes \id_x):x\to \overline{\overline{x}}.
\]
Finally, the maps $\varphi_x:x\to \overline{\overline{x}}$ equip such a category with a canonical pivotal structure, which is furthermore spherical.

Note that a unitary fusion category is a fusion category with an additional structure. A fusion category could therefore, in principle, have more than one unitary structures. The question of uniqueness is best formulated in the following way
(see \cite[\S5]{MR3021796} % Galindo César, Hong Seung-Moon, Rowell Eric, Generalized and quasi-localizations of braid group representations.
for related work):
\begin{quest}
Let $F:\cC\stackrel\simeq\to \cD$ be a tensor equivalence between two unitary fusion categories. Is any such $F$ naturally equivalent to a dagger tensor functor?
\end{quest}

Given a fusion category $\cC$, we define a new category $\cC\otimes_\Vec \Hilb$ as follows.
Its objects are formal expressions
$\bigoplus_i x_i\otimes H_i$ (finite direct sums) with $x_i\in\cC$ and $H_i\in\Hilb$, and the morphisms are given by
\[
\Hom_{\,\cC\otimes_\Vec \Hilb}\Big(\bigoplus_i x_i\otimes H_i,
\bigoplus_j y_j\otimes K_j\Big)\,:=\,\,
\bigoplus_{i,j} \cC(x_i,y_j)\otimes_\C \Hilb(H_i,K_j).
\]
As we saw, if $\cC$ is a unitary fusion category then it is equipped with a canonical bi-involutive structure.
Combining it with the corresponding structure on $\Hilb$ yields a bi-involutive structure on $\cC\otimes_\Vec \Hilb$.
The adjoint of a morphism $\sum f_{ij}\otimes g_{ij}:\bigoplus x_i\otimes H_i\to \bigoplus y_j\otimes K_j$ is $\sum f_{ij}^*\otimes g_{ij}^*$, and the conjugate of an object $\bigoplus x_i\otimes H_i$ is $\bigoplus \overline{x_i}\otimes \overline{H_i}$.
The structure data $\varphi$, $\nu$, $j$ are inherited from those of $\cC$ and of $\Hilb$.

%%%%%%%%%%%%%%%%%%%%%%%%%%%%%%%%%%%%%%%%%%%%%%%%%
\subsection{The commutant of a category}

Given an algebra $B$ and a subalgebra $A\subset B$, the commutant of $A$ inside $B$, also called the centralizer, is the algebra
\[
Z_B(A):=\{b\in B\,|\,ab=ba\,\, \forall a\in A\}.
\]
In this section, we introduce higher categorical variants of the above notion, where the algebras $A$ and $B$ are replaced by tensor categories, dagger tensor categories, and finally bi-involutive tensor categories.

\begin{defn}[\cite{MR1151906}]
Let $\cC$ and $\cD$ be tensor categories, and
let $F=(F,\mu,i): \cC \to \cD$ be a tensor functor.
The commutant $\cZ_\cD(\cC)$ of $\cC$ in $\cD$ is the category
whose objects are pairs $(X, e_X)$ with $X\in\cD$ an object, and $e_X=(e_{X,y}:X\otimes F(y)\stackrel{\scriptscriptstyle \simeq}\to F(y)\otimes X)_{y\in \cC}$ a half-braiding.
The components $e_{X,y}$ of the half-braiding must satisfy the following `hexagon' axiom:
\[
\xymatrix@C=.1cm{
&& 
F(y) \otimes X \otimes F(z)\ar[rrd]^(.55){\id_{F(y)}\otimes e_{X,z}}
\\
\ar[rru]^(.45){e_{X,y}\otimes \id_{F(z)}}X\otimes F(y)\otimes F(z)\ar[rd]^(.6){\id_X\otimes \mu_{y,z}}
&&
&& 
F(y) \otimes F(z) \otimes X
\\
&X\otimes F(y\otimes z)\ar[rr]^{e_{X,y\otimes z}} 
&&
F(y\otimes z)\otimes X \ar[ru]^(.4){\mu_{y,z}^{-1}\otimes \id_X} 
& 
}
\]
Note that by setting $y=z=1_\cC$ in the above diagram, it follows that $e_{X,1_\cC}=\id_X$.

A morphism $(X,e_X)\to (Y, e_Y)$ in $\cZ_\cD(\cC)$ is a morphism $f:X\to Y$ in $\cD$ such that $(\id_{F(z)}\otimes f)\circ e_{X,z} = e_{Y,z}\circ(f\otimes \id_{F(y)})$.
The tensor product of two objects $(X,e_X)$, $(Y, e_Y)$ of $\cZ_\cD(\cC)$ is given by $(X,e_X)\otimes (Y, e_Y) = (X\otimes Y,e_{X\otimes Y})$, with
\[
e_{X\otimes Y,z}=(e_{X,z}\otimes \id_Y)\circ (\id_X\otimes e_{Y,z}),
\]
and the associators and unitors of $\cZ_\cD(\cC)$ are inherited from those of $\cD$.
\end{defn}

\begin{rem}
The Drinfeld center $\cZ(\cC)$ is the commutant of $\cC$ in itself.
\end{rem}

If $\cC$ and $\cD$ are dagger tensor categories and $F: \cC \to \cD$ is a dagger tensor functor, then we may consider the full subcategory
\[
\cZ_\cD^*(\cC)\,\subset\, \cZ_\cD(\cC)
\]
whose objects are pairs $(X, e_X)$ as above, where the maps $e_{X,y}: X \otimes F(y)\to F(y)\otimes X$ are unitary.
We call $\cZ_\cD^*(\cC)$ the \emph{unitary commutant} of $\cC$ in $\cD$ (compare \cite[Def.\,6.1]{MR1966524}).
Unlike $\cZ_\cD(\cC)$, the unitary commutant is a dagger category, and its $*$-operation is inherited from $\cD$.
\begin{rem}
The inclusion $\cZ_\cD^*(\cC)\hookrightarrow \cZ_\cD(\cC)$ is in general not an equivalence.
The easiest counterexample is given by $\cC=\mathsf{Vec}[G]$ for $G$ some infinite group, and $\cD=\mathsf{Vec}$. Then $\cZ^*_\cD(\cC)$ is the category of unitary representations of $G$, whereas $\cZ_\cD(\cC)$ is the category of all representations of $G$.
See \cite[Thm.\,6.4]{MR1966525} %Mueger,
%From Subfactors to Categories and Topology II
and \cite[Proposition 5.24]{MR3021796}
for some positive results when $\cC$ is a fusion category.
\end{rem}

If $\cC$ and $\cD$ are bi-involutive tensor categories, and $F: \cC \to \cD$ is a bi-involutive tensor functor,
then the unitary commutant $\cZ^*_\cD(\cC)$ of $\cC$ in $\cD$ is also naturally equipped with the structure of a bi-involutive tensor category.
The conjugate of $(X,e_X)\in \cZ^*_\cD(\cC)$ is the pair $(\overline X,e_{\overline X})$ consisting of the object $\overline X\in\cD$ and the half-braiding
\[
e_{\overline X,y}:
\xymatrix@C=1.2cm{
\overline X \otimes y
\ar[r]^{\id \otimes \varphi_y} &
\overline X \otimes\overline{\overline{y}}
\ar[r]^{\nu_{X,\overline{y}}} &
\overline{\overline{y} \otimes X}
\ar[r]^{\overline{e_{X,\overline{y}}}^{-1}} &
\overline {X \otimes \overline{y}}
\ar[r]^{\nu_{\overline y,X}^{-1}} &
\overline{\overline{y}}\otimes \overline X
\ar[r]^{\varphi_y^{-1}\otimes \id} &
y\otimes \overline X.
}
\]
The coherence isomorphisms $\varphi$, $j$ and $\nu$ are inherited from $\cD$.

We will be especially interested in the case when $\cD=\Bim(R)$, the tensor category of bimodules over some hyperfinite von Neumann factor $R$.
The monoidal product on that category is based on the operation of Connes fusion, which we describe next.

%%%%%%%%%%%%%%%%%%%%%%%%%%%%%%%%%%%%%%%%%%%%%%%%%
\subsection{\texorpdfstring{$L^2$}{L2}-spaces and Connes fusion}\label{sec:Connes fusion}

Let $R$ be a von Neumann algebra, with predual $R_*$ and positive part $R_*^+\subset R_*$.
The $L^2$-space of $R$ (also known as standard form of $R$), denoted $L^2R$, is the Hilbert space generated by symbols $\sqrt\phi$ for $\phi\in R_*^+$, under the inner product
\[
\langle\sqrt\phi,\sqrt\psi\rangle = \underset{t\to i/2}{\mathrm{anal.\,cont.}}\,
\phi([D\phi:D\psi]_t),
\vspace{-.1cm}
\]
where $[D\phi:D\psi]_t\in R$ is Connes' non-commutative Radon-Nikodym derivative.\footnote{The formula for the inner product makes most sense if one rewrites formally $[D\phi:D\psi]_t$ as $\phi^{it}\psi^{-it}$ and $\phi(a)$ as $\mathrm{Tr}(\phi  a)$. It then simplifies to $\mathrm{Tr}(\phi^{1+it}\psi^{-it})|_{t=i/2}=\mathrm{Tr}(\phi^{1/2}\psi^{1/2})$.
Similarly, for next formula, one may replace formally $\sigma_t^\psi(b)$ by $\psi^{it}b\psi^{-it}$. Note that these formal symbols are genuinely meaningful, and can be implemented as (unbounded) operators on some Hilbert space, see e.g. \cite{MR1203761}.}
The Hilbert space $L^2R$ is an $R$-$R$-bimodule, with the two actions of $R$ are determined by the formula
\[
\langle a\sqrt\phi\, b,\sqrt\psi\rangle = \underset{t\to i/2}{\mathrm{anal.\,cont.}}\,
\phi\big([D\phi:D\psi]_t \sigma_t^\psi(b) a\big),
\vspace{-.1cm}
\]
where $\sigma_t^\psi$ is the modular flow.
Finally, the modular conjugation $J:L^2R\to L^2R$ is given by $J(\lambda\sqrt \phi)=\overline\lambda\sqrt \phi$ for $\lambda\in \bbC$.
General references about $L^2R$ include \cite{MR0407615, %Haagerup. The standard form of von Neumann algebras
MR560633, %Haagerup. L^p-spaces associated with an arbitrary von Neumann algebra
MR2630616}. %Kosaki, Hideki
%Canonical L^p spaces associated with an arbitrary von Neumann algebra (Ph.D. thesis)

Given a right module $H$ and a left module $K$, their fusion $H\boxtimes_R K$ is the Hilbert space generated by symbols $\alpha[\xi]\beta$, for $\alpha:L^2R\to H$ a right $R$-linear map, $\xi\in L^2R$, and $\beta:L^2R\to K$ a left $R$-linear map, under the inner product
\[
\big\langle \alpha_1[\xi_1]\beta_1, \alpha_2[\xi_2]\beta_2\big\rangle \,=\, \big\langle\ell^{-1}(\alpha_2^*\circ\alpha_1)\xi_1 r^{-1}(\beta_2^*\circ\beta_1),\xi_2\big\rangle_{L^2R}.
\]
Here, $\ell$ and $r$ denote the left and right actions of $R$ on its $L^2$ space, defined by $\ell(a)(\xi)=a\xi$ and $r(a)(\xi)=\xi a$, respectively.

There exist two alternative descriptions of $H\boxtimes_R K$, as generated by symbols $\alpha[\xi$ for $\alpha:L^2R\to H$ a right $R$-linear map and $\xi\in K$ a vector, and generated by 
symbols $\xi]\beta$ for $\beta:L^2R\to K$ a left $R$-linear map and $\xi\in H$ a vector.
The isomorphisms between the above models are given by
\[
\alpha[\xi]\beta\,\,\mapsto\,\,\alpha(\xi)]\beta
\qquad\mathrm{and}\qquad
\alpha[\xi]\beta\,\,\mapsto\,\,\alpha[\beta(\xi).
\]
General references about Connes fusion include 
\cite{MR703809, %Sauvageot. Sur le produit....
correspondences} %S. Popa. Correspondences. INCREST Preprint, 56/1986, 1986.
and
\cite[Appendix B.$\delta$]{MR1303779}. %Connes. Noncommutative geometry

The two actions of $R$ on $L^2R$ are each other's commutants.
That property characterizes the bimodules which are invertible with respect to Connes fusion:

\begin{lem}[{\cite[Prop.\,3.1]{MR703809}}]
%Sauvageot. Sur le produit....
\label{lem: invertible bimodules}
Let $A$ and $B$ be von Neumann algebras, and let $H$ be an $A$-$B$-bimodules such that $A$ and $B$ are each other's commutants on $H$ (in particular they act faithfully on $H$).
Then $H$ is an invertible $A$-$B$-bimodule.
\end{lem}

Connes fusion has the following useful \emph{faithfulness} property:

\begin{lem} 
\label{lem: fusion is faithful}
Let $R$ be a von Neumann algebra and let $H$ be a faithful right module.
Then for any left modules $K_1$ and $K_2$, the map
\begin{equation}\label{eq: faithful (in lemma)}
H\boxtimes_R -:
\Hom_R(K_1,K_2)\to\Hom(H\boxtimes_R K_1,H\boxtimes_R K_2)
\end{equation}
is injective.
\end{lem}

\begin{proof}
Let $R'$ be the commutant of $R$ on $H$.
By Lemma \ref{lem: invertible bimodules}, $H$ is an invertible $R'$-$R$-bimodule.
The map (\ref{eq: faithful (in lemma)}) can then be factored as the composite of the bijection
$\Hom_R(K_1,K_2)\cong\Hom_{R'}(H\boxtimes K_1,H\boxtimes K_2)$
with the inclusion
\[
\Hom_{R'}(H\boxtimes K_1,H\boxtimes K_2)\subset\Hom(H\boxtimes K_1,H\boxtimes K_2).\qedhere
\]
\end{proof}

The operation of fusion makes the category $\Bim(R)$ of $R$-$R$-bimodules\footnote{Later on, we will restrict attention to separable von Neumann algebras (i.e., ones which admit faithful actions on separable Hilbert spaces), in which case we will take $\Bim(R)$ to be the category of $R$-$R$-bimodules whose underlying Hilbert space is separable. 
The reason for that restriction will become evident in Section~\ref{sec:Absorbing objects}.} into a tensor category, with unit object $L^2R$.
The associator is given by
\[
(H\boxtimes_R K) \boxtimes_R L \to H\boxtimes_R (K \boxtimes_R L)\,\,:\,\, (\alpha[\xi)]\beta \mapsto \alpha[(\xi]\beta),
\]
for $\alpha:L^2R\to H$ a right $R$-linear map, $\xi \in K$, and $\beta:L^2R\to L$ a left $R$-linear map,
and the two unitors are given by
\begin{equation*}%\label{eq: CF's unitors}
H\boxtimes_R L^2R\to H: \alpha[\xi \mapsto \alpha(\xi)
\quad\,\,\,\mathrm{and}\,\,\,\quad L^2R\boxtimes_R H\to H: \alpha[\xi \mapsto \ell^{-1}(\alpha)\xi.
\end{equation*}

The category $\Bim(R)$ is a dagger tensor category, with adjoints of morphisms defined at the level of the underlying Hilbert spaces.
It is even a bi-involutive tensor category.
Given a bimodule $H\in \Bim(R)$, the underlying Hilbert space of $\overline H$ is the complex conjugate of $H$ (with scalar multiplication $\lambda \overline{\xi}=\overline{\overline{\lambda}\xi}$), and the two actions of $R$ are given by $a\overline{\xi}b=\overline{b^*\xi a^*}$.
The transformation $\varphi$ is the identity.
The map $j:L^2R\to\overline{L^2R}$ is given by $j(\xi)=\overline{J(\xi)}$, with $J$ the modular conjugation (note that $j$ is linear, and $J$ is anti-linear), and the coherence
$\nu :\overline{H} \boxtimes_R \overline{K} \to \overline{K \boxtimes_R H}$ is given by
\[
\nu(\alpha[\xi]\beta)=\overline{(\overline\beta\circ j)[J(\xi)](\overline\alpha\circ j)}
\]
for $\alpha:L^2R\to \overline H$, $\xi\in L^2R$, and $\beta:L^2R\to \overline K$.
The latter is equivalently given by $\nu(\alpha[\xi)=\overline{J(\xi)](\overline\alpha\circ j)}$, or
$\nu(\xi]\beta)=\overline{(\overline\beta\circ j)[J(\xi)}$.

\begin{rem}
Let $\Bim^\circ(R)\subset \Bim(R)$ be the full subcategory of dualizable bimodules (equivalently, the bimodules with finite statistical dimension \cite[\S5 and Cor.\,7.14]{MR3342166})%
%Dualizability and index of subfactors
.
Then by \cite[Cor.\,6.12]{MR3342166}%Dualizability and index of subfactors
, the canonical conjugation on $\Bim^\circ(R)$ (described in Section~\ref{sec: Unitary fusion categories}) is the restriction of
the conjugation on $\Bim(R)$ described above.
\end{rem}

%%%%%%%%%%%%%%%%%%%%%%%%%%%%%%%%%%%%%%%%%%%%%%%%%
\subsection{Graphical calculus}

Throughout this paper, we will use the string diagram calculus familiar from tensor categories: objects are denoted by strands and morphisms are denoted by coupons \cite{MR1113284}%Joyal, Street, The geometry of tensor calculus. I.
, \cite{MR2767048}%Selinger, A survey of graphical languages for monoidal categories
. For example, the following string diagram
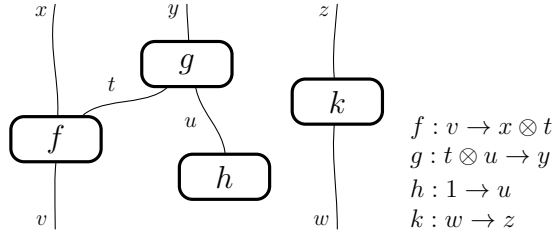
\[
\qquad
\tikzmath{
\draw (.25,0) to[out=80, in=-100] (.25,1.8) node[scale=.7, left, yshift=-4]{$x$};
\draw (.25,-1.2) node[scale=.7, left, yshift=5]{$v$} to[out=90, in=-100] (.3,0);
\draw (.45,0) to[out=80, in=-100, looseness=1.2] (1.9,1);
\draw (2,1) to[out=80, in=-90] (2,1.8) node[scale=.7, left, yshift=-4]{$y$};
\draw (4,.5) to[out=100, in=-95] (4,1.8) node[scale=.7, left, yshift=-4]{$z$};
\draw (2.5,-.5) to[out=80, in=-95, looseness=1.2] (2.1,1);
\draw (4,-1.2) node[scale=.7, left, yshift=5]{$w$} to[out=90, in=-100] (4,.5);
\node[scale=.7] at (1,.73) {$t$};
\node[scale=.7] at (2.05,.23) {$u$};
\filldraw[fill=white, very thick, rounded corners=5pt]
(.26,0) node{$f$} +(-.6,-.3) rectangle +(.6,.3)
(2,1) node{$g$} +(-.6,-.3) rectangle +(.6,.3)
(2.5,-.5) node{$h$} +(-.6,-.3) rectangle +(.6,.3)
(4,.5) node{$k$} +(-.6,-.3) rectangle +(.6,.3);
\pgftransformxshift{10}
\pgftransformyshift{10}
\node[scale=.8, right] at (4.5,-.2) {$f:v\to x\otimes t$};
\node[scale=.8, right] at (4.5,-.6) {$g: t\otimes u \to y$};
\node[scale=.8, right] at (4.5,-1.0) {$h:1\to u$};
\node[scale=.8, right] at (4.5,-1.4) {$k:w\to z$};
}
\]
represents a morphism $v\otimes w\to x\otimes y\otimes z$.

Given a dualizable object $x\in \cC$ in a C*-tensor category, the canonical evaluation and coevaluations maps 
$\mathrm{ev}_x : \overline{x} \otimes x \to 1$ and $\mathrm{coev}_x : 1 \to x \otimes \overline{x}$,
and their adjoints 
$\mathrm{ev}_x^* : 1\to \overline{x} \otimes x$ and $\mathrm{coev}_x^* : x \otimes \overline{x}\to 1$
are denoted graphically as follows:
\[
\mathrm{ev}_x:
\begin{tikzpicture}[baseline=.4cm]
\draw (0,.3) -- (0,.6) arc (180:0:.3cm) -- (.6,.3);
\node at (0,.1) {\scriptsize{$\overline x$}};
\node at (.6,.08) {\scriptsize{$x$}};
\end{tikzpicture}
\qquad
\mathrm{coev}_x:
\begin{tikzpicture}[baseline=-.55cm]
\draw (0,-.3) -- (0,-.6) arc (-180:0:.3cm) -- (.6,-.3);
\node at (0,-.12) {\scriptsize{$x$}};
\node at (.6,-.1) {\scriptsize{$\overline x$}};
\end{tikzpicture}
\qquad
\mathrm{ev}_x^*:
\begin{tikzpicture}[baseline=-.55cm]
\draw (0,-.3) -- (0,-.6) arc (-180:0:.3cm) -- (.6,-.3);
\node at (0,-.1) {\scriptsize{$\overline x$}};
\node at (.6,-.12) {\scriptsize{$x$}};
\end{tikzpicture}
\qquad
\mathrm{coev}_x^*:
\begin{tikzpicture}[baseline=.4cm]
\draw (0,.3) -- (0,.6) arc (180:0:.3cm) -- (.6,.3);
\node at (0,.08) {\scriptsize{$x$}};
\node at (.6,.1) {\scriptsize{$\overline x$}};
\end{tikzpicture}.
\]%
They satisfy:
\[
\begin{tikzpicture}[baseline=-.65cm]
\draw (0,0) -- (0,-.6) arc (-180:0:.3cm) arc (180:0:.3cm) -- (1.2,-1.2);
\node at (-.18,-.02) {\scriptsize{$x$}};
\node at (1.02,-1.15) {\scriptsize{$x$}};
\node at (.44,-.5) {\scriptsize{$\overline x$}};
\end{tikzpicture}
\,\,\,=\!\!
\begin{tikzpicture}[baseline=-.65cm]
\draw (0,0) -- (0,-1.2);
\node at (-.18,-1.15) {\scriptsize{$x$}};
\end{tikzpicture}
\,\,,\quad
\begin{tikzpicture}[baseline=.55cm]
\draw (0,0) -- (0,.6) arc (180:0:.3cm) arc (-180:0:.3cm) -- (1.2,1.2);
\node at (-.2,.05) {\scriptsize{$\overline x$}};
\node at (1.01,1.15) {\scriptsize{$\overline x$}};
\node at (.44,.55) {\scriptsize{$x$}};
\end{tikzpicture}
\,\,\,=\!\!
\begin{tikzpicture}[baseline=-.65cm]
\draw (0,0) -- (0,-1.2);
\node at (-.18,-1.15) {\scriptsize{$\overline x$}};
\end{tikzpicture}
\,\,,\qquad
\begin{tikzpicture}[baseline=-.1cm]
	\draw (0,.3) -- (0,.48) arc (180:0:.3cm) -- (.6,-.48) arc (0:-180:.3cm)-- (0,-.3);
	\roundNbox{unshaded}{(0,0)}{.3}{0}{0}{$f$};
	\node at (-.18,-.53) {\scriptsize{$x$}};
	\node at (.8,-.5) {\scriptsize{$\overline x$}};
\end{tikzpicture}
=
\begin{tikzpicture}[baseline=-.1cm, xscale=-1]
	\draw (0,.3) -- (0,.48) arc (180:0:.3cm) -- (.6,-.48) arc (0:-180:.3cm)-- (0,-.3);
	\roundNbox{unshaded}{(0,0)}{.3}{0}{0}{$f$};
	\node at (-.18,-.53) {\scriptsize{$x$}};
	\node at (.8,-.5) {\scriptsize{$\overline x$}};
\end{tikzpicture}
\,\,\,\forall\,f:x\to x\,,
\]
along with the equations
$\overline{\mathrm{ev}_x}=j\circ\mathrm{ev}_x
\circ(\id_{\overline x}\otimes\varphi_x^{-1})\circ
\nu_{x,\overline x}^{-1}$
and
$\overline{\mathrm{coev}_x}=\nu_{\overline x,x}\circ(\varphi_x\otimes\id_{\overline x})\circ\mathrm{coev}_x\circ j^{-1}$ which,
after omitting the coherences $j$, $\nu$, and $\varphi$, can be conveniently abbreviated
\[
\overline{\mathrm{ev}_x}=\mathrm{ev}_x
\qquad\qquad\text{and}\qquad\qquad
\overline{\mathrm{coev}_x}=\mathrm{coev}_x.
\]

The \emph{dimension} of a dualizable object $x\in \cC$ is given by
\[
d_x:=\mathrm{coev}_x^*\circ\mathrm{coev}_x=\mathrm{ev}_x\circ\mathrm{ev}_x^*\,\in\,{\mathbb R}_{\ge 0}.
\]
Given dualizable objects $x,y,z\in\cC$, Frobenius reciprocity (or pivotality) provides canonical isomorphisms
\[
\!\!\tikz{
\node[scale=1.2] at (0,.35) {$
\scriptstyle
\mathrm{Hom}(1,x\otimes y\otimes z) \,\,\cong\,\,
\mathrm{Hom}(1,y\otimes z\otimes x) \,\,\cong\,\,
\mathrm{Hom}(1,z\otimes x\otimes y) \,\,\cong\,\,
\mathrm{Hom}(\overline z,x\otimes y) \,\,\cong\,\,
\mathrm{Hom}(\overline x,y\otimes z) \,\,\cong\,\,
\mathrm{Hom}(\overline y,z\otimes x)
$};
\node[scale=1.2] at (0,-.35) {$
\scriptstyle \,\,\cong\,\,
\mathrm{Hom}(\overline z\otimes \overline y,x)\,\,\cong\,\,
\mathrm{Hom}(\overline x\otimes \overline z,y)\,\,\cong\,\,
\mathrm{Hom}(\overline y\otimes \overline x,z)\,\,\cong\,\,
\mathrm{Hom}(\overline z\otimes \overline y\otimes \overline x,1)\,\,\cong\,\,
\mathrm{Hom}(\overline x\otimes \overline z\otimes \overline y,1)\,\,\cong\,\,
\mathrm{Hom}(\overline y\otimes \overline x\otimes \overline z,1).$};}
\]
The sesquilinear pairing o, for $f,g\in \mathrm{Hom}(1,x\otimes y\otimes z)$, equips this vector space with the structure of a finite dimensional Hilbert space.
The dual (or complex conjugate) Hilbert space is then given by any one the following canonically isomorphic vector spaces:
\[
\!\!\tikz{
\node[scale=1.2] at (0,.35) {$
\scriptstyle
\mathrm{Hom}(1,\overline z\otimes \overline y\otimes \overline x) \,\,\cong\,\,
\mathrm{Hom}(1,\overline y\otimes \overline x\otimes \overline z) \,\,\cong\,\,
\mathrm{Hom}(1,\overline x\otimes \overline z\otimes \overline y) \,\,\cong\,\,
\mathrm{Hom}( x,\overline z\otimes \overline y) \,\,\cong\,\,
\mathrm{Hom}( z,\overline y\otimes \overline x) \,\,\cong\,\,
\mathrm{Hom}( y,\overline x\otimes \overline z)
$};
\node[scale=1.2] at (0,-.35) {$
 \scriptstyle \,\,\cong\,\,
\mathrm{Hom}( x\otimes  y,\overline z)\,\,\cong\,\,
\mathrm{Hom}( z\otimes  x,\overline y)\,\,\cong\,\,
\mathrm{Hom}( y\otimes  z,\overline x)\,\,\cong\,\,
\mathrm{Hom}( x\otimes  y\otimes  z,1) \,\,\cong\,\,
\mathrm{Hom}( z\otimes  x\otimes  y,1) \,\,\cong\,\,
\mathrm{Hom}( y\otimes  z\otimes  x,1).
$};}
\]

Let $e_i\in\mathrm{Hom}(1,x\otimes y\otimes z)$ and $e^i\in\mathrm{Hom}(1,\overline z\otimes \overline y\otimes \overline x)$ be dual bases, and consider the canonical element 
\[
\sqrt{d_xd_yd_z}\,\cdot\,\sum_i e_i\otimes e^i. 
\]
We will be making great use of string diagrams where pairs of trivalent nodes are labeled by the above canonical element.
These will be denoted by pairs of circular colored nodes, as follows:
\begin{equation}
\label{eq: coloured dots notation}
\begin{tikzpicture}[baseline=.2cm]
	\draw (.2,.6) -- (0,.3) -- (-.2,.6);
	\draw (0,0) -- (0,.3);
	\draw[fill=\betacolor] (0,.3) circle (.05cm);	
	\node at (-.2,.8) {\scriptsize{$x$}};
	\node at (.2,.8) {\scriptsize{$y$}};
	\node at (0,-.2) {\scriptsize{$z$}};
\end{tikzpicture}
\otimes
\begin{tikzpicture}[baseline=-.38cm]
	\draw (.2,-.6) -- (0,-.3) -- (-.2,-.6);
	\draw (0,0) -- (0,-.3);
	\draw[fill=\betacolor] (0,-.3) circle (.05cm);	
	\node at (-.2,-.8) {\scriptsize{$x$}};
	\node at (.2,-.8) {\scriptsize{$y$}};
	\node at (0,.2) {\scriptsize{$z$}};
\end{tikzpicture}
:=\,
\sqrt{d_xd_yd_z}\,\cdot\,\sum_i\,\,\,
\begin{tikzpicture}[baseline=.17cm]
	\draw (.15,.8) -- (.15,.4) (-.15,.4) -- (-.15,.8);
	\draw (0,-.2) -- (0,.4);
	\draw[fill=\betacolor] (0,.3) circle (.05cm);	
	\node at (-.15,1) {\scriptsize{$x$}};
	\node at (.15,1) {\scriptsize{$y$}};
	\node at (0,-.4) {\scriptsize{$z$}};
	\roundNbox{unshaded}{(0,.3)}{.3}{0}{0}{$e_i$};
\end{tikzpicture}
\,\otimes\,
\begin{tikzpicture}[baseline=-.45cm]
	\draw (.15,-.8) -- (.15,-.4) (-.15,-.4) -- (-.15,-.8);
	\draw (0,.2) -- (0,-.4);
	\draw[fill=\betacolor] (0,-.3) circle (.05cm);	
	\node at (-.15,-1) {\scriptsize{$x$}};
	\node at (.15,-1) {\scriptsize{$y$}};
	\node at (0,.4) {\scriptsize{$z$}};
	\roundNbox{unshaded}{(0,-.3)}{.3}{0}{0}{$e^i$};
\end{tikzpicture}
\end{equation}
\begin{rem}
The element
$
\begin{tikzpicture}[scale=.8, baseline=.13cm]
	\draw (.2,.6) -- (0,.3) -- (-.2,.6);
	\draw (0,0) -- (0,.3);
	\draw[fill=\betacolor] (0,.3) circle (.05cm);	
	\node at (-.2,.77) {\scriptsize{$x$}};
	\node at (.2,.75) {\scriptsize{$y$}};
	\node at (0,-.15) {\scriptsize{$z$}};
\end{tikzpicture}
{\otimes}
\begin{tikzpicture}[scale=.8, baseline=-.35cm]
	\draw (.2,-.6) -- (0,-.3) -- (-.2,-.6);
	\draw (0,0) -- (0,-.3);
	\draw[fill=\betacolor] (0,-.3) circle (.05cm);	
	\node at (-.2,-.75) {\scriptsize{$x$}};
	\node at (.2,-.77) {\scriptsize{$y$}};
	\node at (0,.15) {\scriptsize{$z$}};
\end{tikzpicture}
$
lies in $\mathrm{Hom}(z,x\otimes y)\,\otimes\, \mathrm{Hom}(x\otimes y,z)$,\vspace{-.4cm}
and should not be confused with
$
\begin{tikzpicture}[scale=.8, baseline=.13cm]
	\draw (.2,.6) -- (0,.3) -- (-.2,.6);
	\draw (0,0) -- (0,.3);
	\draw[fill=\betacolor] (0,.3) circle (.05cm);	
	\node at (-.2,.77) {\scriptsize{$x$}};
	\node at (.2,.75) {\scriptsize{$y$}};
	\node at (0,-.15) {\scriptsize{$z$}};
\end{tikzpicture}\!\!%
\begin{tikzpicture}[scale=.8, baseline=-.35cm]
	\draw (.2,-.6) -- (0,-.3) -- (-.2,-.6);
	\draw (0,0) -- (0,-.3);
	\draw[fill=\betacolor] (0,-.3) circle (.05cm);	
	\node at (-.2,-.75) {\scriptsize{$x$}};
	\node at (.2,-.77) {\scriptsize{$y$}};
	\node at (0,.15) {\scriptsize{$z$}};
\end{tikzpicture}\!
\in\,\mathrm{Hom}(z\otimes x\otimes y,x\otimes y\otimes z)$.
\end{rem}

When occurring in a bigger diagram, it might happen that we need to use the above canonical elements in more that one place. 
In that case, we will use multiple colors to indicate the various pairs of nodes
(often, the coupling can also be inferred from the string labels).
The remaining coupons will be sometimes denoted by little squares.
For example:
\begin{equation}
%\label{eq: coloured dots notation}
\begin{tikzpicture}[xscale=1.1, baseline=.25cm]
\coordinate (A) at (-1.1,.8);
\coordinate (B) at (0,.9);
\coordinate (C) at (1.2,.8);
\coordinate (D) at (-.8,-.3);
\coordinate (E) at (.25,-.3);
\coordinate (F) at (1.7,-.3);
\draw (A) -- +(0,.9);
\draw (B) -- +(0,.8);
\draw (C) -- +(.35,.9);
\draw (C) -- +(-.35,.9);
\draw (C) -- +(-.35,-1.7);
\draw (A) -- (D);
\draw (B) -- (D);
\draw (B) -- (E);
\draw (C) -- (E);
\draw (C) -- (F);
\draw (D) -- +(0,-.6);
\draw (E) -- +(0,-.6);
\draw (F) -- +(-.3,-.6);
\draw (F) -- +(.3,-.6);
	\draw[fill=\betacolor] (B) circle (.05cm);	
	\draw[fill=\betacolor] (E) circle (.05cm);	
	\draw[fill=\alphacolor] (D) circle (.05cm);	
	\draw[fill=\alphacolor] (F) circle (.05cm);	
	\filldraw (A) +(.05,.05) rectangle +(-.05,-.05) (A) node[left, scale=.8]{$f$};	
	\filldraw (C) +(.05,.05) rectangle +(-.05,-.05) (C) node[left, scale=.8, xshift=-1]{$g$};	
\node[scale=.9] at (-1.1,.2) {\scriptsize{$x$}};
\node[scale=.9] at (-.5,.4) {\scriptsize{$y$}};
\node[scale=.9] at (0.02,.25) {\scriptsize{$z$}};
\node[scale=.9] at (.57,.27) {\scriptsize{$\overline v$}};
\node[scale=.9] at (1.38,.08) {\scriptsize{$\overline y$}};
\node[scale=.9] at (-.98,-.85) {\scriptsize{$u$}};
\node[scale=.9] at (.1,-.83) {\scriptsize{$\overline y$}};
\node[scale=.9] at (.8,-.4) {\scriptsize{$w$}};
\node[scale=.9] at (1.25,-.82) {\scriptsize{$\overline u$}};
\node[scale=.9] at (1.8,-.85) {\scriptsize{$x$}};
\node[scale=.9] at (-1.28,1.65) {\scriptsize{$u$}};
\node[scale=.9] at (-.15,1.65) {\scriptsize{$v$}};
\node[scale=.9] at (.7,1.65) {\scriptsize{$s$}};
\node[scale=.9] at (1.37,1.66) {\scriptsize{$t$}};
\end{tikzpicture}
\!:=\,\,
%{\scriptstyle
\sqrt{d_xd_yd_u}\sqrt{d_vd_yd_z}
%}
\,\cdot\,
\sum_{i,j}\,\,
\begin{tikzpicture}[xscale=1.2, baseline=.2cm]
\coordinate (A) at (-1.1,.85);
\coordinate (B) at (.1,1);
\coordinate (C) at (1.3,.8);
\coordinate (D) at (-.8,-.3);
\coordinate (E) at (.3,-.4);
\coordinate (F) at (1.7,-.3);
\draw (A) -- +(0,.85);
\draw (B) -- +(0,.7);
\draw (C) to[in=-90-20,out=90-20] +(.35,.9);
\draw (C) to[in=-90+20,out=90+20] +(-.35,.9);
\draw (C) to[in=87,out=-90] +(-.37,-1.8);
\draw (A) to[in=90,out=-90] (D);
\draw (B) to[in=90-10,out=-90-10] (D);
\draw (B) to[in=90+10,out=-90+10] (E);
\draw (C) to[in=90-15,out=-90-15] (E);
\draw (C)+(.1,0) to[in=85,out=-95] (F);
\draw (D) -- +(0,-.7);
\draw (E) -- +(0,-.6);
\draw (F) -- +(-.3,-.7);
\draw (F) -- +(.3,-.7);
	\roundNbox{unshaded}{(B)}{.3}{0}{0}{$e^i$};
	\roundNbox{unshaded}{(F)}{.3}{0}{0}{$e^j$};
	\roundNbox{unshaded}{(E)}{.3}{0}{0}{$e_i$};
	\roundNbox{unshaded}{(D)}{.3}{0}{0}{$e_j$};
	\roundNbox{unshaded}{(C)}{.3}{0}{0}{$g$};
	\roundNbox{unshaded}{(A)}{.3}{0}{0}{$f$};
\node[scale=.85] at (-1.1,.24) {\scriptsize{$x$}};
\node[scale=.85] at (-.5,.4) {\scriptsize{$y$}};
\node[scale=.85] at (0.07,.25) {\scriptsize{$z$}};
\node[scale=.85] at (.65,.27) {\scriptsize{$\overline v$}};
\node[scale=.85] at (1.4,.2) {\scriptsize{$\overline y$}};
\node[scale=.85] at (-.98,-.93) {\scriptsize{$u$}};
\node[scale=.85] at (.15,-.92) {\scriptsize{$\overline y$}};
\node[scale=.85] at (.9,-.35) {\scriptsize{$w$}};
\node[scale=.85] at (1.25,-.9) {\scriptsize{$\overline u$}};
\node[scale=.85] at (1.8,-.93) {\scriptsize{$x$}};
\node[scale=.85] at (-1.26,1.65) {\scriptsize{$u$}};
\node[scale=.85] at (-.07,1.65) {\scriptsize{$v$}};
\node[scale=.85] at (.8,1.65) {\scriptsize{$s$}};
\node[scale=.85] at (1.45,1.66) {\scriptsize{$t$}};
\end{tikzpicture}
\end{equation}

When $x,y,z\in\cC$ are irreducible objects, we will write $N_{x,y}^z$ for the dimension of $\mathrm{Hom}( x\otimes  y,z)$.
Let us also fix a set $\Irr(\cC)\subset \Ob(\cC)$ of representatives of the isomorphism classes of irreducible objects.

The following lemma lists the most important relations satisfied in the above graphical calculus.
To our knowledge, the following relations have not appeared in this exact form in the literature, but they are certainly well known to experts:
%They have already appeared in the literature (for example in \cite{MR1642584,MR1832764}, with slightly different normalizations \textcolor{red}{check whether the normalizations are still different}).

\begin{lem}
The following relations hold:
\begin{align*}
\begin{tikzpicture}[baseline=-.1cm]
	\draw (0,-.6) -- (0,-.3);
	\draw (0,-.3) .. controls ++(45:.2cm) and ++(-45:.2cm) .. (0,.3);
	\draw (0,-.3) .. controls ++(135:.2cm) and ++(225:.2cm) .. (0,.3);
	\draw (0,.6) -- (0,.3);
	\draw[fill=\betacolor] (0,-.3) circle (.05cm);
	\draw[fill=\betacolor] (0,.3) circle (.05cm);	
	\node at (0,-.8) {\scriptsize{$z$}};
	\node at (.3,0) {\scriptsize{$y$}};
	\node at (-.3,0) {\scriptsize{$x$}};
	\node at (0,.8) {\scriptsize{$z$}};
\end{tikzpicture}
=& \,\,\,
\sqrt{d_xd_yd_z^{-1}}\cdot N_{x,y}^z\,
\begin{tikzpicture}[baseline=-.1cm]
	\draw (-.2,-.6) -- (-.2,.6);
	\node at (-.2,-.8) {\scriptsize{$z$}};
\end{tikzpicture}
&&
\text{\emph{(Bigon 1)}}
%\label{eq:Bigon 1}
\\
\begin{tikzpicture}[baseline=-.1cm]
	\draw (0,-.6) -- (0,-.3);
	\draw (0,-.3) .. controls ++(45:.2cm) and ++(-45:.2cm) .. (0,.3);
	\draw (0,-.3) .. controls ++(135:.2cm) and ++(225:.2cm) .. (0,.3);
	\draw (0,.6) -- (0,.3);
	\draw[fill=\betacolor] (0,-.3) circle (.05cm);
	\draw[fill=\alphacolor] (0,.3) circle (.05cm);	
	\node at (0,-.8) {\scriptsize{$z$}};
	\node at (.3,0) {\scriptsize{$y$}};
	\node at (-.3,0) {\scriptsize{$x$}};
	\node at (0,.8) {\scriptsize{$z$}};
\end{tikzpicture}
\otimes
\begin{tikzpicture}[baseline=.2cm]
	\draw (.2,.6) -- (0,.3) -- (-.2,.6);
	\draw (0,0) -- (0,.3);
	\draw[fill=\alphacolor] (0,.3) circle (.05cm);	
	\node at (-.2,.8) {\scriptsize{$x$}};
	\node at (.2,.8) {\scriptsize{$y$}};
	\node at (0,-.2) {\scriptsize{$z$}};
\end{tikzpicture}
\otimes
\begin{tikzpicture}[baseline=-.4cm]
	\draw (.2,-.6) -- (0,-.3) -- (-.2,-.6);
	\draw (0,0) -- (0,-.3);
	\draw[fill=\betacolor] (0,-.3) circle (.05cm);	
	\node at (-.2,-.8) {\scriptsize{$x$}};
	\node at (.2,-.8) {\scriptsize{$y$}};
	\node at (0,.2) {\scriptsize{$z$}};
\end{tikzpicture}
&\,= \,\,
\sqrt{d_xd_yd_z^{-1}}\cdot
\begin{tikzpicture}[baseline=-.1cm]
	\draw (-.2,-.6) -- (-.2,.6);
	\node at (-.2,-.8) {\scriptsize{$z$}};
\end{tikzpicture}
\otimes
\begin{tikzpicture}[baseline=.2cm]
	\draw (.2,.6) -- (0,.3) -- (-.2,.6);
	\draw (0,0) -- (0,.3);
	\draw[fill=\betacolor] (0,.3) circle (.05cm);	
	\node at (-.2,.8) {\scriptsize{$x$}};
	\node at (.2,.8) {\scriptsize{$y$}};
	\node at (0,-.2) {\scriptsize{$z$}};
\end{tikzpicture}
\otimes
\begin{tikzpicture}[baseline=-.4cm]
	\draw (.2,-.6) -- (0,-.3) -- (-.2,-.6);
	\draw (0,0) -- (0,-.3);
	\draw[fill=\betacolor] (0,-.3) circle (.05cm);	
	\node at (-.2,-.8) {\scriptsize{$x$}};
	\node at (.2,-.8) {\scriptsize{$y$}};
	\node at (0,.2) {\scriptsize{$z$}};
\end{tikzpicture}
&&
\text{\emph{(Bigon 2)}}
%\label{eq:Bigon 2}
\\
\sum_{z\in \Irr(\cC)}
\sqrt{d_z}&
\begin{tikzpicture}[baseline=-.1cm]
	\draw (.2,-.6) -- (0,-.3) -- (-.2,-.6);
	\draw (.2,.6) -- (0,.3) -- (-.2,.6);
	\draw (0,-.3) -- (0,.3);
	\draw[fill=\betacolor] (0,-.3) circle (.05cm);
	\draw[fill=\betacolor] (0,.3) circle (.05cm);	
	\node at (-.2,-.8) {\scriptsize{$x$}};
	\node at (.2,-.8) {\scriptsize{$y$}};
	\node at (-.2,.8) {\scriptsize{$x$}};
	\node at (.2,.8) {\scriptsize{$y$}};
	\node at (.2,0) {\scriptsize{$z$}};
\end{tikzpicture}
\,=\,\sqrt{d_xd_y}\cdot
\begin{tikzpicture}[baseline=-.1cm]
	\draw (.2,-.6) -- (.2,.6);
	\draw (-.2,-.6) -- (-.2,.6);
	\node at (-.2,-.8) {\scriptsize{$x$}};
	\node at (.2,-.8) {\scriptsize{$y$}};
\end{tikzpicture}
&&
\,\text{\emph{(Fusion)}}
%\label{eq:Fusion}
\\
\qquad
\sum_{v\in \Irr(\cC)}
\begin{tikzpicture}[baseline=-.1cm]
	\draw (.2,-.6) -- (0,-.3) -- (-.2,-.6);
	\draw (.2,.6) -- (0,.3) -- (-.2,.6);
	\draw (0,-.3) -- (0,.3);
	\draw[fill=\betacolor] (0,-.3) circle (.05cm);
	\draw[fill=\alphacolor] (0,.3) circle (.05cm);
	\node at (-.2,-.8) {\scriptsize{$x$}};
	\node at (.2,-.8) {\scriptsize{$w$}};
	\node at (-.2,.8) {\scriptsize{$y$}};
	\node at (.2,.8) {\scriptsize{$z$}};
	\node at (.2,0) {\scriptsize{$v$}};
\end{tikzpicture}
\,\otimes
\begin{tikzpicture}[baseline=-.1cm]
	\draw (.2,-.6) -- (0,-.3) -- (-.2,-.6);
	\draw (.2,.6) -- (0,.3) -- (-.2,.6);
	\draw (0,-.3) -- (0,.3);
	\draw[fill=\betacolor] (0,-.3) circle (.05cm);
	\draw[fill=\alphacolor] (0,.3) circle (.05cm);
	\node at (-.2,-.8) {\scriptsize{$\overline w$}};
	\node at (.2,-.8) {\scriptsize{$\overline x$}};
	\node at (-.2,.8) {\scriptsize{$\overline z$}};
	\node at (.2,.8) {\scriptsize{$\overline y$}};
	\node at (.2,0) {\scriptsize{$\overline v$}};
\end{tikzpicture}
&\,=
\sum_{u\in \Irr(\cC)}
\begin{tikzpicture}[baseline=-.1cm, rotate=90]
	\draw (.4,-.5) -- (-.03,-.3) -- (-.4,-.5);
	\draw (.4,.5) -- (.03,.3) -- (-.4,.5);
	\draw (-.03,-.3) -- (.03,.3);
	\draw[fill=\betacolor] (-.03,-.3) circle (.05cm);
	\draw[fill=\alphacolor] (.03,.3) circle (.05cm);
	\node at (-.6,-.5) {\scriptsize{$w$}};
	\node at (.6,-.5) {\scriptsize{$z$}};
	\node at (-.6,.5) {\scriptsize{$x$}};
	\node at (.6,.5) {\scriptsize{$y$}};
	\node at (.2,0) {\scriptsize{$u$}};
\end{tikzpicture}
\otimes
\begin{tikzpicture}[baseline=-.1cm, rotate=90]
	\draw (.4,-.5) -- (.03,-.3) -- (-.4,-.5);
	\draw (.4,.5) -- (-.03,.3) -- (-.4,.5);
	\draw (.03,-.3) -- (-.03,.3);
	\draw[fill=\alphacolor] (.03,-.3) circle (.05cm);
	\draw[fill=\betacolor] (-.03,.3) circle (.05cm);
	\node at (-.6,-.5) {\scriptsize{$\overline x$}};
	\node at (.6,-.5) {\scriptsize{$\overline y$}};
	\node at (-.6,.5) {\scriptsize{$\overline w$}};
	\node at (.6,.5) {\scriptsize{$\overline z$}};
	\node at (.2,0) {\scriptsize{$\overline u$}};
\end{tikzpicture}
&&
\,\,\,\,\text{\emph{(I=H)}}
%\label{eq:I=H}
\end{align*}
\end{lem}

\begin{proof}
By definition, the dual basis $e_i\in\mathrm{Hom}(z,x\otimes y)$ and $e^i\in\mathrm{Hom}(x\otimes y,z)$ satisfy
\[
\mathrm{tr}(e^j \circ e_i)\,
=\,\,
\begin{tikzpicture}[baseline=-.1cm]
	\draw (.15,-.3) -- (.15,.3);
	\draw (-.15,-.3) -- (-.15,.3);
	\draw (0,.8) arc (180:0:.3cm) -- (.6,-.8) arc (0:-180:.3cm);
	\roundNbox{unshaded}{(0,.5)}{.3}{0}{0}{$e^j$};
	\roundNbox{unshaded}{(0,-.5)}{.3}{0}{0}{$e_i$};
	\node at (-.1,-1) {\scriptsize{$z$}};
	\node at (.3,0) {\scriptsize{$y$}};
	\node at (-.3,0) {\scriptsize{$x$}};
	\node at (-.1,1) {\scriptsize{$z$}};
\end{tikzpicture}
\,\,\,=\,\delta_{i,j}.
\]
By `undoing the trace' it follows that, for $e_i$ and $e^j$ as above,
\begin{equation} \label{eq: ej ei = dz delta_ij}
e^j \circ e_i = d_z^{-1}\delta_{i,j}\cdot\id_z.
\end{equation}
The two Bigon relations are immediate consequences of the above equation:
\begin{align*}
\sqrt{d_xd_yd_z}\cdot\sum_i e^i\circ e_i
&=
\sqrt{d_xd_yd_z}\cdot\sum_i d_z^{-1}\cdot\id_z=
\sqrt{d_xd_yd_z^{-1}}\,N_{x,y}^z\cdot\id_z
\\
d_xd_yd_z\sum_{i,j} (e^j\circ e_i)\otimes e_j\otimes e^i
=\,&
d_xd_y\sum_{i,j} \delta_{i,j} \id_z\otimes e_j\otimes e^i
=
\sqrt{d_xd_yd_z^{-1}}\sqrt{d_xd_yd_z}\sum_i \id_z\otimes e_i\otimes e^i.
\end{align*}

In order to prove the Fusion relation
\[
\sum_{z,j} \sqrt{d_z}\sqrt{d_xd_yd_z}\cdot e_j \circ e^j=\sqrt{d_xd_y}\cdot\id_{x\otimes y}\,,
\]
it is enough to argue that it holds after precomposition with an arbitrary basis element $e_i\in\mathrm{Hom}(z',x\otimes y)$ and object $z'\in\Irr(\cC)$.
So we must show that the equation
$\sum_{z,j} d_z\cdot e_j \circ e^j\circ e_i = e_i$ holds.
This is again a consequence of equation (\ref{eq: ej ei = dz delta_ij}):
\[
\sum_{z,j} d_z\cdot e_j \circ e^j\circ e_i
=\sum_{z,j} d_z\cdot e_j\circ( d_z^{-1}\delta_{z,z'}\delta_{i,j}\cdot\id_z)\,=\,e_i.
\]

To prove the I=H relation, we rewrite it as
\[
\sqrt{d_xd_yd_zd_w}\,\cdot\,\sum_{v,i,j}\, d_v
\begin{tikzpicture}[baseline=-.1cm]
\coordinate (d) at (.4,-.5);
\coordinate (e) at (.4,.5);
    \draw (.4,-.5) -- (.4,.5);
	\draw (.25,1) -- (.25,.5);
	\draw (.55,1) -- (.55,.5);
	\draw (.25,-1) -- (.25,-.5);
	\draw (.55,-1) -- (.55,-.5);
    \roundNbox{unshaded}{(d)}{.3}{0}{0}{$e^i$}
    \roundNbox{unshaded}{(e)}{.3}{0}{0}{$e_j$}
\node at (.25,-1.2) {\scriptsize{$x$}};
\node at (.55,-1.2) {\scriptsize{$w$}};
\node at (.25,1.17) {\scriptsize{$y$}};
\node at (.55,1.2) {\scriptsize{$z$}};
\node at (.57,0) {\scriptsize{$v$}};
\end{tikzpicture}
\,\,\otimes\,\,
\begin{tikzpicture}[baseline=-.1cm]
\coordinate (d) at (.4,-.5);
\coordinate (e) at (.4,.5);
    \draw (.4,-.5) -- (.4,.5);
	\draw (.25,1) -- (.25,.5);
	\draw (.55,1) -- (.55,.5);
	\draw (.25,-1) -- (.25,-.5);
	\draw (.55,-1) -- (.55,-.5);
    \roundNbox{unshaded}{(d)}{.3}{0}{0}{$e_i$}
    \roundNbox{unshaded}{(e)}{.3}{0}{0}{$e^j$}
\node at (.25,-1.18) {\scriptsize{$\overline w$}};
\node at (.55,-1.18) {\scriptsize{$\overline x$}};
\node at (.25,1.23) {\scriptsize{$\overline z$}};
\node at (.55,1.2) {\scriptsize{$\overline y$}};
\node at (.57,0) {\scriptsize{$\overline v$}};
\end{tikzpicture}
\,\,=\,
\sqrt{d_xd_yd_zd_w}\,\cdot\,\sum_{u,i,j}\, d_u\,\,
\begin{tikzpicture}[baseline=-.1cm]
\coordinate (d) at (-.5,0);
\coordinate (e) at (.5,0);
\draw (d) -- (e) (-.5,-.5) -- (-.5,.5) (.5,-.5) -- (.5,.5);
\roundNbox{unshaded}{(d)}{.3}{0}{0}{$e^i$}
\roundNbox{unshaded}{(e)}{.3}{0}{0}{$e_j$}
	\node at (-.5,-.65) {\scriptsize{$x$}};
	\node at (.5,-.65) {\scriptsize{$w$}};
	\node at (-.5,.65) {\scriptsize{$y$}};
	\node at (.5,.65) {\scriptsize{$z$}};
	\node at (0,.17) {\scriptsize{$u$}};
\end{tikzpicture}
\,\,\otimes\;\!
\,\begin{tikzpicture}[baseline=-.1cm]
\coordinate (d) at (-.5,0);
\coordinate (e) at (.5,0);
\draw (d) -- (e) (-.5,-.5) -- (-.5,.5) (.5,-.5) -- (.5,.5);
\roundNbox{unshaded}{(d)}{.3}{0}{0}{$e^{j'}$}
\roundNbox{unshaded}{(e)}{.3}{0}{0}{$e_{i'}$}
\node at (-.5,-.65) {\scriptsize{$\overline w$}};
\node at (.5,-.65) {\scriptsize{$\overline x$}};
\node at (-.5,.65) {\scriptsize{$\overline z$}};
\node at (.5,.65) {\scriptsize{$\overline y$}};
\node at (0,.2) {\scriptsize{$\overline u$}};
\end{tikzpicture}
\]
and note that both sides are of the form $\sqrt{d_xd_yd_zd_w}\,\sum f_a\otimes f^a$ for $\{f_a\}$ a basis of $\mathrm{Hom}(x\otimes w, y\otimes z)$ and $\{f^a\}$ the dual basis of $\mathrm{Hom}(\overline w\otimes \overline x, \overline z\otimes \overline y)$ with respect to the pairing
\[
\big\langle
\begin{tikzpicture}[baseline=-.1cm]
	\coordinate (b) at (-.4,0);
	\coordinate (c) at (-.4,-.5);
	\draw (-.25,-.5) -- (-.25,.5);
	\draw (-.55,-.5) -- (-.55,.5);
    \roundNbox{unshaded}{(b)}{.3}{0}{0}{$f$}
	\node at (-.6,.65) {\scriptsize{$y$}};
	\node at (-.2,.65) {\scriptsize{$z$}};
	\node at (-.6,-.65) {\scriptsize{$x$}};
	\node at (-.2,-.65) {\scriptsize{$w$}};
\end{tikzpicture}
,
\begin{tikzpicture}[baseline=-.1cm]
	\coordinate (b) at (-.4,0);
	\draw (-.25,-.5) -- (-.25,.5);
	\draw (-.55,-.5) -- (-.55,.5);
    \roundNbox{unshaded}{(b)}{.3}{0}{0}{$g$}
\node at (-.6,.67) {\scriptsize{$\overline z$}};
\node at (-.2,.67) {\scriptsize{$\overline y$}};
\node at (-.6,-.67) {\scriptsize{$\overline w$}};
\node at (-.2,-.67) {\scriptsize{$\overline x$}};
\end{tikzpicture}
\big\rangle
\,\,:=\,\,
\begin{tikzpicture}[baseline=-.1cm]
	\coordinate (b) at (-.4-.05,0);
	\coordinate (c) at (.4+.05,0);
	\draw (-.3,.45) -- (-.3,-.45) arc (-180:0:.3);
	\draw (-.6,.45) -- (-.6,-.45) arc (-180:0:.6);
    \roundNbox{unshaded}{(b)}{.3}{0}{0}{$f$}
	\draw (.3,-.45) -- (.3,.45) arc (0:180:.3);
	\draw (.6,-.45) -- (.6,.45) arc (0:180:.6);
    \roundNbox{unshaded}{(c)}{.3}{0}{0}{$g$}
\end{tikzpicture}\,.
\vspace{-.2cm}
\]
To see that
$
\,d_v\begin{tikzpicture}[baseline=-.1cm]
\coordinate (d) at (.4,-.5);
\coordinate (e) at (.4,.5);
    \draw (.4,-.5) -- (.4,.5);
	\draw (.25,1) -- (.25,.5);
	\draw (.55,1) -- (.55,.5);
	\draw (.25,-1) -- (.25,-.5);
	\draw (.55,-1) -- (.55,-.5);
    \roundNbox{unshaded}{(d)}{.3}{0}{0}{$e^i$}
    \roundNbox{unshaded}{(e)}{.3}{0}{0}{$e_j$}
\node at (.25,-1.2) {\scriptsize{$x$}};
\node at (.55,-1.2) {\scriptsize{$w$}};
\node at (.25,1.17) {\scriptsize{$y$}};
\node at (.55,1.2) {\scriptsize{$z$}};
\node at (.57,0) {\scriptsize{$v$}};
\end{tikzpicture}\,
$
and
$
\begin{tikzpicture}[baseline=-.1cm]
\coordinate (d) at (.4,-.5);
\coordinate (e) at (.4,.5);
    \draw (.4,-.5) -- (.4,.5);
	\draw (.25,1) -- (.25,.5);
	\draw (.55,1) -- (.55,.5);
	\draw (.25,-1) -- (.25,-.5);
	\draw (.55,-1) -- (.55,-.5);
    \roundNbox{unshaded}{(d)}{.3}{0}{0}{$e_i$}
    \roundNbox{unshaded}{(e)}{.3}{0}{0}{$e^j$}
\node at (.25,-1.18) {\scriptsize{$\overline w$}};
\node at (.55,-1.18) {\scriptsize{$\overline x$}};
\node at (.25,1.23) {\scriptsize{$\overline z$}};
\node at (.55,1.2) {\scriptsize{$\overline y$}};
\node at (.57,0) {\scriptsize{$\overline v$}};
\end{tikzpicture}\,
$
are indeed dual bases, we use the relation (\ref{eq: ej ei = dz delta_ij}) twice:
\[
d_v\,\,\,
\begin{tikzpicture}[baseline=-.1cm]
\coordinate (a) at (.4-1,-.5);
\coordinate (b) at (.4-1,.5);
    \draw (.4-1,-.5) -- (.4-1,.5);
	\draw (.25-1,.5) -- (.25-1,.9) arc (180:0:.65);
	\draw (.55-1,.5) -- (.55-1,.9) arc (180:0:.35);
\draw (.25-1,-.5) -- (.25-1,-.9) arc (-180:0:.65);
\draw (.55-1,-.5) --(.55-1,-.9) arc (-180:0:.35);
    \roundNbox{unshaded}{(a)}{.3}{0}{0}{$e^i$}
    \roundNbox{unshaded}{(b)}{.3}{0}{0}{$e_j$}
\node at (.57-1,0) {\scriptsize{$v$}};
\coordinate (d) at (.4,-.5);
\coordinate (e) at (.4,.5);
    \draw (.4,-.5) -- (.4,.5);
	\draw (.25,.9) -- (.25,.5);
	\draw (.55,.9) -- (.55,.5);
	\draw (.25,-.9) -- (.25,-.5);
	\draw (.55,-.9) -- (.55,-.5);
    \roundNbox{unshaded}{(d)}{.3}{0}{0}{$e_{i'}$}
    \roundNbox{unshaded}{(e)}{.3}{0}{0}{$e^{j'}$}
\node at (.57,0) {\scriptsize{$\overline {v'}$}};
\end{tikzpicture}
\,\,=\,d_vd_v^{-1}\delta_{v,v'}\delta_{j,j'}\,\,\,
\begin{tikzpicture}[baseline=-.1cm]
\coordinate (a) at (.4-1,-.5);
\coordinate (b) at (.4-1,.5);
    \draw (.4-1,-.5) -- (.4-1,.5) arc (180:0:.5);
\draw (.25-1,-.5) -- (.25-1,-.9) arc (-180:0:.65);
\draw (.55-1,-.5) --(.55-1,-.9) arc (-180:0:.35);
    \roundNbox{unshaded}{(a)}{.3}{0}{0}{$e^i$}
\node at (.57-1,0) {\scriptsize{$v$}};
\coordinate (d) at (.4,-.5);
\coordinate (e) at (.4,.5);
    \draw (.4,-.5) -- (.4,.5);
	\draw (.25,-.9) -- (.25,-.5);
	\draw (.55,-.9) -- (.55,-.5);
    \roundNbox{unshaded}{(d)}{.3}{0}{0}{$e_{i'}$}
\node at (.57,0) {\scriptsize{$\overline v$}};
\end{tikzpicture}
\,\,=\,d_vd_v^{-2}\delta_{v,v'}\delta_{j,j'}\delta_{i,i'}\cdot
\begin{tikzpicture}[baseline=-.1cm]
\coordinate (a) at (.4-1,-.5);
\coordinate (b) at (.4-1,.5);
    \draw (.4-1,-.5) -- (.4-1,.5) arc (180:0:.5) -- (.4,-.5) arc (0:-180:.5);
\node at (.57-1,-.3) {\scriptsize{$v$}};
\end{tikzpicture}
\,\,=\delta_{v,v'}\delta_{j,j'}\delta_{i,i'}.
\]
The verification that
$
d_u\,\,
\begin{tikzpicture}[baseline=-.1cm]
\coordinate (d) at (-.5,0);
\coordinate (e) at (.5,0);
\draw (d) -- (e) (-.5,-.5) -- (-.5,.5) (.5,-.5) -- (.5,.5);
\roundNbox{unshaded}{(d)}{.3}{0}{0}{$e^i$}
\roundNbox{unshaded}{(e)}{.3}{0}{0}{$e_j$}
	\node at (-.5,-.65) {\scriptsize{$x$}};
	\node at (.5,-.65) {\scriptsize{$w$}};
	\node at (-.5,.65) {\scriptsize{$y$}};
	\node at (.5,.65) {\scriptsize{$z$}};
	\node at (0,.17) {\scriptsize{$u$}};
\end{tikzpicture}\,
$
and
$
\,\begin{tikzpicture}[baseline=-.1cm]
\coordinate (d) at (-.5,0);
\coordinate (e) at (.5,0);
\draw (d) -- (e) (-.5,-.5) -- (-.5,.5) (.5,-.5) -- (.5,.5);
\roundNbox{unshaded}{(d)}{.3}{0}{0}{$e^{j'}$}
\roundNbox{unshaded}{(e)}{.3}{0}{0}{$e_{i'}$}
\node at (-.5,-.65) {\scriptsize{$\overline w$}};
\node at (.5,-.65) {\scriptsize{$\overline x$}};
\node at (-.5,.65) {\scriptsize{$\overline z$}};
\node at (.5,.65) {\scriptsize{$\overline y$}};
\node at (0,.2) {\scriptsize{$\overline u$}};
\end{tikzpicture}\,
$
are dual bases is entirely similar.
\end{proof}

Let us now assume that $\cC$ is furthermore a fusion category, and let
$
\dim(\cC):=\sum_{x\in\Irr(\cC)}\, d_x^2
$
be its global dimension.
We then have the following result.

\begin{lem}
\label{lem:Sumdxdy}
The following relation holds:
\begin{equation}\label{eq: two bigons}
\sum_{a,b\in\Irr(\cC)}\begin{tikzpicture}[baseline=-.1cm]
	\draw (0,-.6) -- (0,-.3);
	\draw (0,-.3) .. controls ++(45:.2cm) and ++(-45:.2cm) .. (0,.3);
	\draw (0,-.3) .. controls ++(135:.2cm) and ++(225:.2cm) .. (0,.3);
	\draw (0,.6) -- (0,.3);
	\draw[fill=\betacolor] (0,-.3) circle (.05cm);
	\draw[fill=\alphacolor] (0,.3) circle (.05cm);	
	\node at (0,-.8) {\scriptsize{$x$}};
	\node at (.27,.02) {\scriptsize{$b$}};
	\node at (-.3,0) {\scriptsize{$a$}};
	\node at (0,.8) {\scriptsize{$y$}};
\end{tikzpicture}
\otimes
\begin{tikzpicture}[baseline=-.1cm]
	\draw (0,-.6) -- (0,-.3);
	\draw (0,-.3) .. controls ++(45:.2cm) and ++(-45:.2cm) .. (0,.3);
	\draw (0,-.3) .. controls ++(135:.2cm) and ++(225:.2cm) .. (0,.3);
	\draw (0,.6) -- (0,.3);
	\draw[fill=\betacolor] (0,-.3) circle (.05cm);
	\draw[fill=\alphacolor] (0,.3) circle (.05cm);	
	\node at (0,-.8) {\scriptsize{$\overline x$}};
	\node at (.3,0) {\scriptsize{$\overline a$}};
	\node at (-.3,.05) {\scriptsize{$\overline b$}};
	\node at (0,.83) {\scriptsize{$\overline y$}};
\end{tikzpicture}
\,=\,\,
\dim(\cC)\cdot\delta_{x,y}
\begin{tikzpicture}[baseline=-.1cm]
	\draw (-.2,-.6) -- (-.2,.6);
	\node at (-.2,-.8) {\scriptsize{$x$}};
\end{tikzpicture}
\otimes
\begin{tikzpicture}[baseline=-.1cm]
	\draw (-.2,-.6) -- (-.2,.6);
	\node at (-.2,-.8) {\scriptsize{$\overline x$}};
\end{tikzpicture}
\end{equation}
\end{lem}

\begin{proof}
Recall that $d_a=d_{\overline a}$.
For every $x\in\Irr(\cC)$, we have
\[
\sum_{a, b} d_a d_b N_{a,b}^x 
=
\sum_{a}  d_a\big(\sum_{b} N_{\overline{a}, x}^b d_b \big)
=
\sum_{a} d_a(d_ad_x)
=
d_x\;\!\dim(\cC).
\]
Using the two Bigon relations, the left hand side of (\ref{eq: two bigons}) then
simplifies to
\[
\sum_{a,b}d_ad_bd_x^{-1}N_{a,b}^x\,\delta_{x,y}\cdot\id_x\otimes\id_{\overline x}
\,=\,
\dim(\cC)\,\delta_{x,y}\cdot\id_x\otimes\id_{\overline x}.\qedhere
\]
\end{proof}

There is an alternative proof of Lemma \ref{lem:Sumdxdy} which proceeds as follows.
We use the I=H relation to rewrite the left hand side of (\ref{eq: two bigons}) as
\[
\sum_{a,b\in\Irr(\cC)}
\begin{tikzpicture}[baseline=-.1cm]
	\draw (0,-.5) -- (0,.5) (0,0) -- (-.4,0);
	\draw (-.4,0) .. controls ++(180+50:.8cm) and ++(180-50:.8cm) .. (-.4,0);
	\draw[fill=\betacolor] (0,0) circle (.05cm);
	\draw[fill=\alphacolor] (-.4,0) circle (.05cm);	
	\node at (0,-.7) {\scriptsize{$x$}};
	\node at (-.21,.19) {\scriptsize{$b$}};
	\node at (-.95,0) {\scriptsize{$a$}};
	\node at (0,.7) {\scriptsize{$y$}};
\end{tikzpicture}
\otimes
\begin{tikzpicture}[baseline=-.1cm, xscale=-1]
	\draw (0,-.5) -- (0,.5) (0,0) -- (-.4,0);
	\draw (-.4,0) .. controls ++(180+50:.8cm) and ++(180-50:.8cm) .. (-.4,0);
	\draw[fill=\betacolor] (0,0) circle (.05cm);
	\draw[fill=\alphacolor] (-.4,0) circle (.05cm);	
	\node at (0,-.7) {\scriptsize{$\overline x$}};
	\node at (-.21,.19+.02) {\scriptsize{$\overline b$}};
	\node at (-.95,0) {\scriptsize{$\overline a$}};
	\node at (0,.73) {\scriptsize{$\overline y$}};
\end{tikzpicture}
\]
We then note that the only terms which contribute to the sum are the ones with $b = 1$, and so we are left with
\[
\sum_{a\in\Irr(\cC)}
\begin{tikzpicture}[baseline=-.1cm]
	\draw circle(.2);
	\node at (0,-.39) {\scriptsize{$a$}};
\end{tikzpicture}
\cdot
\begin{tikzpicture}[baseline=-.1cm]
	\draw circle(.2);
	\node at (0,-.4) {\scriptsize{$\overline a$}};
\end{tikzpicture}
\cdot\! %\delta_{x,y}
\begin{tikzpicture}[baseline=-.1cm]
	\draw (-.2,-.6) -- (-.2,.6);
	\node at (-.2,-.8) {\scriptsize{$x$}};
\end{tikzpicture}
\otimes
\begin{tikzpicture}[baseline=-.1cm]
	\draw (-.2,-.6) -- (-.2,.6);
	\node at (-.2,-.8) {\scriptsize{$\overline x$}};
\end{tikzpicture}
\,=\,\,
\dim(\cC)\cdot\!%\delta_{x,y}
\begin{tikzpicture}[baseline=-.1cm]
	\draw (-.2,-.6) -- (-.2,.6);
	\node at (-.2,-.8) {\scriptsize{$x$}};
\end{tikzpicture}
\otimes
\begin{tikzpicture}[baseline=-.1cm]
	\draw (-.2,-.6) -- (-.2,.6);
	\node at (-.2,-.8) {\scriptsize{$\overline x$}};
\end{tikzpicture}
\]

%%%%%%%%%%%%%%%%%%%%%%%%%%%%%%%%%%%%%%%%%%%%%%%%%
\subsection{Cyclic fusion}

Given rings $R_i$ and bimodules ${}_{R_{i-1}}(M_i)_{R_i}$ for $i\in\{1,\ldots,n\}$ (indices modulo $n$), we may define the cyclic tensor product
\begin{equation}\label{eq: cyclic tensor product}
\Big[ M_1 \otimes_{R_1} M_2 \otimes_{R_2} \ldots \otimes_{R_{n-1}} M_n \otimes_{R_n} - \Big]_{\text{cyclic}}:=\,
\big(M_1 \otimes_{\mathbb Z} M_2 \otimes_{\mathbb Z} \ldots \otimes_{\mathbb Z} M_n\big)\big/\sim\,
\end{equation}
where $\sim$ is the equivalence relation generated by
\begin{align*}
\qquad m_1\ldots\otimes m_{i-1} r\otimes m_i\otimes\ldots  m_n \,&\sim\, m_1\ldots\otimes m_{i-1} \otimes r m_i\otimes\ldots m_n
&\text{for}\,\,\,\, r\in R_i\phantom{.}
\\
\text{and}\qquad\quad m_1\otimes\ldots  \otimes m_nr \,&\sim\, rm_1\otimes\ldots  \otimes m_n
&\text{for}\,\,\,\, r\in R_n.
\end{align*}
The \emph{cyclic Connes fusion}, first introduced in \cite[Appendix A]{1409.8672},
is the analogue of the above construction for Connes fusion.

Unlike the cyclic tensor product,
the cyclic fusion is not always defined.
Let us explain by an analogy why it is not always defined, and when we can expect it to be defined.
If one takes the point of view that a bimodule between rings is something that categorifies the notion of a linear map, then the expression (\ref{eq: cyclic tensor product}) categorifies the number
\[
tr(f_1\circ f_2\circ\ldots\circ f_n).
\]
Now, we like to think of bimodules between von Neumann algebras as categorifying the notion of a bounded linear map between infinite dimensional Hilbert spaces.
Given bounded linear maps $f_i:H_{i-1}\to H_i$, $i\in\{1,\ldots,n\}$ (indices modulo $n$), then the above trace is not always defined.
It is however defined if \emph{at least two of the maps are Hilbert-Schmidt}.

For bimodules between von Neumann algebras, we propose the following as a categorification of the Hilbert-Schmidt condition:

\begin{defn}\label{definition of coarse}
A bimodule ${}_AH_B$ between von Neumann algebras is \emph{coarse} if the action of the
algebraic tensor product $A\odot B^{\text{op}}$ extends to the spatial tensor product $A\,\bar\otimes\, B^{\text{op}}$.
Equivalently, a bimodule is coarse if it is a direct summand of a bimodule of the form 
\begin{equation}\label{eq: coarse bimodule}
{}_A(H_1)\otimes_{\mathbb C} (H_2)_B
\end{equation}
(and if $A$ or $B$ are factors, then any coarse bimodule is of the form (\ref{eq: coarse bimodule})).
\end{defn}

Coarse bimodules form an ideal in the sense that
if ${}_AH_B$ is coarse and ${}_BK_C$ is any bimodule, then ${}_AH\boxtimes_BK_C$ is coarse.

\begin{defn}
\label{def: cyclic fusion}
Let $R_i$ be von Neumann algebras, and let ${}_{R_{i-1}}(H_i)_{R_i}$, $i\in\{1,\ldots,n\}$, be bimodules (indices modulo $n$).
Assume that \emph{at least two} of the $H_i$ are coarse.
Then we define the \emph{cyclic fusion} by:
\begin{gather*} %\label{eq: cyclic tensor product}
\Big[ H_1 \boxtimes_{R_1} H_2 \boxtimes_{R_2} \ldots \boxtimes_{R_{n-1}} H_n \boxtimes_{R_n} - \Big]_{\text{cyclic}}:=
\\
\Big(H_{a+1} \boxtimes_{R_{a+1}} \ldots \boxtimes_{R_{b-1}} H_b\Big)
\boxtimes_{R_a^{\text{op}}\bar\otimes R_b}
\Big(H_{b+1} \boxtimes_{R_{b+1}} \ldots \boxtimes_{R_{a-1}} H_a\Big)
\end{gather*}
(cyclic numbering), where the indices $a$ and $b$ are chosen so that at least one of the $\{H_{a+1},\ldots,H_b\}$ is coarse,
and at least one of the $\{H_{b+1},\ldots,H_a\}$ is coarse.
\end{defn}

\begin{rem}
A priori, the above description depends on the choice of locations $a$ and $b$ used to ``cut the circle'':
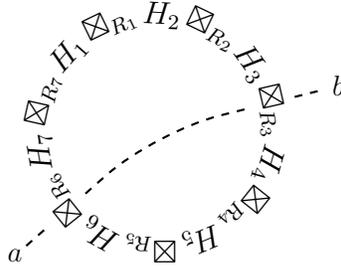
\vspace{-.15cm}
\[
\begin{tikzpicture}
\foreach \n in {1,...,7}
{
\node[rotate=-\n*360/7+100] at (190-\n*360/7:1.55) {$H_{\n}$};
\node[rotate=-\n*360/7+100-180/7+3] at (-\n*360/7+190-180/7:1.55) {$\boxtimes_{R_{\n}}$};
}
\draw[thick, dashed, rounded corners=25]
(170+360/7:2.3) -- (170+360/7:1.9)
(170+360/7:1.25) -- (0,0) -- (14:1.25)
(14:1.9) -- (14:2.3);
\node[scale=.9] at (170+360/7:2.5) {$a$};
\node[scale=.9] at (14:2.5) {$b$};
\end{tikzpicture}
\]
In \cite[Appendix A]{1409.8672}, it was shown that when all the $H_i$ are coarse (and as long as there are at least two of them), the cyclic fusion is well defined up to canonical unitary isomorphism.
It is also well defined in the presence of non-coarse bimodules:
let the $H_{i_1},\ldots,H_{i_k}$ be coarse, and let the other bimodules be non-coarse. Then we may define the cyclic fusion  in terms of the operation described in \cite[Appendix~A]{1409.8672} as
\[
\Big[ 
\big(H_{i_1+1}\boxtimes\ldots\boxtimes H_{i_2}\big)\boxtimes_{R_{i_2}}
\big(H_{i_2+1}\boxtimes\ldots\boxtimes H_{i_3}\big)\boxtimes_{R_{i_3}}\ldots\ldots 
\boxtimes_{R_{i_k}}\big(H_{i_k+1}\boxtimes\ldots\boxtimes H_{i_1}\big)\boxtimes_{R_{i_1}}
-
\Big]_{\text{cyclic}}
\]
%This does not depend on any choices, and is therefore visibly well defined.
\end{rem}

Inspired by \cite{MR3095324}, %Ponto, Shulman: Shadows and traces in bicategories
we propose the following graphical calculus for morphisms between cyclic fusions.
The Hilbert space $[ H_1 \boxtimes_{R_1} \ldots\boxtimes_{R_{n-1}} H_n \boxtimes_{R_n} - ]_{\text{cyclic}}$ corresponds to an arrangement of parallel strands (labelled by the various Hilbert spaces) on the surface of a cylinder.
A string diagram on the cylinder represents a morphism:
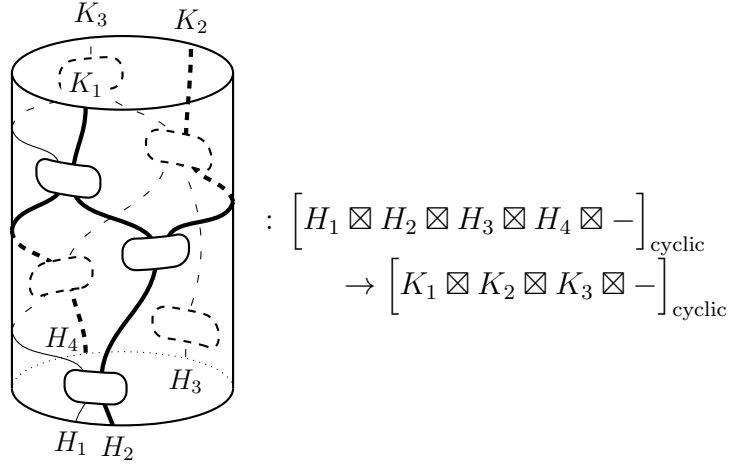
\[
\qquad\qquad
\begin{tikzpicture}[scale=.7]
\draw[thick] (0,6) circle (2.1 and .7);
\draw[thick] (0,0) +(2.1,0) arc(0:-180:2.1 and .7);
\draw[dotted] (0,0) +(2.1,0) arc(0:180:2.1 and .7);
\draw[thick] (2.1,0) -- +(0,6);
\draw[thick] (-2.1,0) -- +(0,6);
\draw (-115:2.1 and .7)node[below, scale=.9, xshift=-2]{$H_1$} to[out=75,in=-90] (-.6,.05); % bottom to h
\draw[ultra thick] (-95:2.1 and .7)node[below, scale=.9, xshift=2]{$H_2$} to[ultra thick, out=110,in=-90] (-.4,.05); % bottom to h
\draw[ultra thick] (-.4,.15) to[out=90,in=-90] (.65,2.55);  % h to g
\draw[ultra thick] (.65,2.55) to[out=90,in=-90] (-.95,3.9);  % g to f
\draw[ultra thick] (.85,2.75) to[out=90,in=-105] (2.1,3.55);  % g to right side
\draw[ultra thick, dashed] (2.1,3.55) to[out=105,in=-90] (1.3,4.4);%right side to b
\draw[loosely dashed] (-1.1,2.1) to[out=90,in=-90] (.85,4.5); % c to b
\draw[ultra thick, dashed] (1.2,4.5) to[out=90,in=-90] (1.3,6.5)node[above, scale=.9, yshift=1]{$K_2$};% b to top
\draw[loosely dashed] (-.6,5.9) -- (-.6,6.7)node[above, scale=.9]{$K_3$};% a to top
\draw[loosely dashed] (1,4.6) to[out=90,in=-90] (-.4,5.9);% b to a
\draw (-1.2,4.2) to[out=90,in=-80] (-2.1,4.95);  % f to left side
\draw[loosely dashed] (-2.1,4.95) to[out=80,in=-80] (-.8,5.9);  % left side to a
\draw[ultra thick, dashed] (-1.2,2.1) to[out=90,in=-85] (-2.1,3);% c to left side
\draw[ultra thick] (-2.1,3) to[out=85,in=-80] (-1.15,3.97);  % left side to f
\draw[loosely dashed] (1.15,1.15) to[bend right=25] (1,4.5); % d to b
\draw[ultra thick, dashed] (-.7,.68)node[above, scale=.9, xshift=-10, yshift=-4]{$H_4$} to[out=90,in=-85] (-1,2.1); % bottom to c
\draw[loosely dashed] (1.2,.6)node[below, scale=.9, yshift=-1]{$H_3$} -- (1.2,1.1);  % bottom to d
\draw (-.7,.15) to[out=90,in=-80] (-2.1,1.1);  % h to left side
\draw[loosely dashed] (-2.1,1.1) to[out=80,in=-90] (-1.4,2);  % left side to c
\begin{scope}[rounded corners=5pt]
\draw[fill=white, thick]   %h
(0,.4) +(-88:2.2 and .7) arc(-88:-120:2.2 and .7)
-- +(0,.6) arc(-120:-88:2.2 and .7) -- cycle;
\draw[fill=white, thick, dashed]
(0,.3) +(34:2.2 and .7) arc(34:75:2.2 and .7)
-- +(0,.6) arc(75:34:2.2 and .7) -- cycle;
\draw[fill=white, thick, dashed]
(0,1.25) +(105:2.2 and .7) arc(105:143:2.2 and .7)
-- +(0,.6) arc(143:105:2.2 and .7) -- cycle;
\draw[fill=white, thick]   %g
(0,2.95) +(-55:2.2 and .7) arc(-55:-90:2.2 and .7)
-- +(0,.6) arc(-90:-55:2.2 and .7) -- cycle;
\draw[fill=white, thick, dashed]
(0,3.6) +(40:2.2 and .7) arc(33:74:2.2 and .7)
-- +(0,.6) arc(74:33:2.2 and .7) -- cycle;
\draw[fill=white, thick, dashed]
(0,5) +(90:2.2 and .7) arc(90:123:2.2 and .7)
-- +(0,.6) arc(123:90:2.2 and .7) -- cycle;
\draw[ultra thick] (-.95,3.97) to[out=90,in=-90] (-.7,5.33)node[above, scale=.9, fill=white, inner sep=1, yshift=3]{$\,K_1\,$};  % f to top
\draw[fill=white, thick]   %f
(0,4.3) +(-101:2.2 and .7) arc(-101:-138:2.2 and .7)
-- +(0,.6) arc(-138:-101:2.2 and .7) -- cycle;
\end{scope}
\node[right] at (2.5,3.13) {$:\,\Big[ H_1 \boxtimes H_2 \boxtimes H_3 \boxtimes H_4 \boxtimes - \Big]_{\text{cyclic}}$};
\node[right] at (4,1.9) {$\to\Big[ K_1 \boxtimes K_2 \boxtimes K_3 \boxtimes - \Big]_{\text{cyclic}}$};
\end{tikzpicture}
\]
We draw thick strands for the coarse bimodules, and thin strands for the bimodules which are not coarse.
For a morphism to be well defined, any horizontal plane intersecting the cylinder should cross at least two thick strands (and if the plane crosses through the middle of a coupon which is connected to at least one thick strand, then this coupon counts as \emph{one} thick strand).

Later on in this paper, we will combine the above cylinder graphical calculus with 
the coloured dots notation from (\ref{eq: coloured dots notation}).

%%%%%%%%%%%%%%%%%%%%%%%%%%%%%%%%%%%%%%%%%%%%%%%%%
%%%%%%%%%%%%%%%%%%%%%%%%%%%%%%%%%%%%%%%%%%%%%%%%%
%%%%%%%%%%%%%%%%%%%%%%%%%%%%%%%%%%%%%%%%%%%%%%%%%

\section{Bicommutant categories}
\label{sec: Bicommutant categories}

Let $R$ be a hyperfinite factor, and let $\Bim(R)$ be the category of $R$-$R$-bimodules whose underlying Hilbert space is separable.
The latter is a bi-involutive tensor category under the operation of Connes fusion, as discussed in Section \ref{sec:Connes fusion}.

Recall that a bi-involutve tensor functor between two bi-involutve tensor categories $\cC$ and $\cD$  is a  quadruple $(F,\mu,i,\upsilon)$, where $F:\cC\to \cD$ is a functor, and 
\[
\mu_{x,y}:F(x)\otimes_{\scriptscriptstyle\cD} F(y)\to F(x\otimes_{\scriptscriptstyle\cC} y),\quad
i:1_\cD\to F(1_\cC),\quad
\upsilon_x:F({\overline x}^{\scriptscriptstyle\cC})\to \overline{F(x)}^{\scriptscriptstyle\cD}
\]
are unitary isomorphisms.

\begin{nota}
Given a bi-involutive tensor category $\cC$ and a bi-involutive tensor functor $\cC\to \Bim(R)$,
we will write
\[
\cC' := \cZ^*_{\Bim(R)}(\cC)
\]
for the unitary commutant of $\cC$ in $\Bim(R)$.
\end{nota}

There is an obvious bi-involutive tensor functor $\cC'\to \Bim(R)$ given by forgetting the half-braiding.
It therefore makes sense to consider the commutant of the commutant.
There is also an `inclusion' functor $\iota:\cC\to \cC''$ from the category to its bicommutant.
It sends an object $X\in\cC$ to the object $(X,e'_X)\in\cC''$ with half-braiding 
given by $e'_{X,(Y,e_Y)}:=e_{Y,X}^{-1}$
for $(Y,e_Y)\in\cC'$.
The coherence data $\mu$, $i$, $\upsilon$ for $\iota$ are all identity morphisms.

\begin{defn}\label{def: bicommutant category}
A \emph{bicommutant category} is a bi-involutive tensor category $\cC$ for which there exists a hyperfinite factor $R$ and a bi-involutive tensor functor $\cC\to \Bim(R)$, such that the `inclusion' functor $\iota:\cC\to \cC''$ is an equivalence.
\end{defn}

If a bi-involutive tensor functor $\alpha:\cC\to \Bim(R)$ is such that the corresponding `inclusion' functor $\iota$ is an equivalence, then we say that \emph{$\alpha$ exhibits $\cC$ as a bicommutant category}.

%%%%%%%%%%%%%%%%%%%%%%%%%%%%%%%%%%%%%%%%%%%%%%%%%
\subsection{Representing tensor categories in \texorpdfstring{$\Bim(R)$}{Bim(R)}}

A representation of a $*$-algebra $A$ on a Hilbert space $H$ is a $*$-algebra homomorphism $A\to B(H)$.
By analogy, we define a \emph{representation} of a bi-involutive tensor category $\cC$ to be a bi-involutive tensor functor
$\cC\to\Bim(R)$, for some von Neumann algebra $R$. 
One can alternatively describe this as an action of $\cC$ on the category $\mathrm{Mod}(R)$ of left $R$-modules.

\begin{defn}
\label{defn:Representation}
A morphism between two representations $\alpha_1:\cC\to\Bim(R_1)$ and $\alpha_2:\cC\to\Bim(R_2)$ of $\cC$ consists of an $R_2$-$R_1$-bimodule $\Phi$, along with unitary natural isomorphisms
\[
\phi_X: \Phi \boxtimes_{R_1} \alpha_1(X) \to \alpha_2(X)\boxtimes_{R_2} \Phi
\]
for every $X\in \cC$, subject to the coherence condition
\[
\xymatrix@C=1.4cm{
\Phi \boxtimes_{R_1} \alpha_1(X) \boxtimes_{R_1} \alpha_1(Y) \ar[d]^{\id{\scriptscriptstyle\boxtimes}\mu_1}
\ar[r]^{\phi_X{\scriptscriptstyle\boxtimes}\id}
&
\alpha_2(X)\boxtimes_{R_2} \Phi \boxtimes_{R_1}
\alpha_1(Y)
\ar[r]^{\id{\scriptscriptstyle\boxtimes}\phi_Y}
&
\alpha_2(X)\boxtimes_{R_2} \alpha_2(Y)\boxtimes_{R_2} \Phi
\ar[d]^{\mu_2{\scriptscriptstyle\boxtimes}\id}
\\
\Phi\boxtimes_{R_1}\alpha_1(X\otimes Y)
\ar[rr]^{\phi_{X\otimes Y}}
&&
\alpha_2(X\otimes Y)\boxtimes_{R_2} \Phi
.}
\]
A morphism $(\Phi,\phi)$ between two representations is an equivalence if the bimodule $\Phi$ is invertible, or equivalently if the induced map $\mathrm{Mod}(R_1)\to \mathrm{Mod}(R_2)$ is an equivalence of categories.
\end{defn}

A representation $\cC\to\Bim(R)$ is called \emph{fully faithful} if non-isomorphic objects of $\cC$ remain non-isomorphic in $\Bim(R)$, and if simple objects of $\cC$ remain simple in $\Bim(R)$ (this agrees with the usual notion of fully faithfulness from category theory).
In the next theorem, we will see that if we restrict the von Neumann algebra $R$ to be a hyperfinite factor which is not of type $\rm I$, then every unitary fusion category admits a fully faithful representation in $\Bim(R)$.
We begin with the following well known lemma:

\begin{lem}\label{lem:R(x)II_1=R}
Let $R$ be a hyperfinite factor which is not of type ${\rm I}$, and let $R_{{\rm II}_1}$ be a hyperfinite ${\rm II}_1$-factor.
Then $R\,\bar\otimes\, R_{{\rm II}_1}\cong R$.
\end{lem}

\begin{proof}
If $R$ is either of type ${\rm II}_1$ or ${\rm II}_\infty$, then the result follows from the uniqueness of the hyperfinite ${\rm II}_1$ and ${\rm II}_\infty$ factors \cite[Thm.\,XIV]{MR0009096}. % Connes. Classification of injective factors. Cases II_1, II_inf, III_l, l not 1..
We may therefore assume that $R$ is of type~${\rm III}$.

Let $\sigma:\bbR \to\mathrm{Aut}(R)$ be the modular flow of $R$.
The \emph{flow of weights} \cite{MR480760} %Connes and Takesaki The flow of weights on factors of type III
is the dual action of $\bbR$ on the von Neumann algebra $S(R):=Z(R\rtimes_\sigma \bbR)$.\footnote{Unlike the modular flow, which depends on a choice of state, the crossed product $R\rtimes_\sigma \bbR$ does not depend of any choices, up to canonical isomorphism.}
By the work of Connes, Haagerup, and Krieger 
\cite{MR0454659, % Connes. Classification of injective factors. Cases II_1, II_inf, III_l, l not 1.
MR880070, %Haagerup. Connes' bicentralizer problem and uniqueness of the injective factor of type III_1. 
MR0415341} %Krieger. On ergodic flows and the isomorphism of factors.
(see also \cite[Chapt.\,XVIII]{MR1943007}%Takesaki. Theory of operator algebras. III.
), the map $R\mapsto S(R)$ establishes a bijective correspondence between isomorphism classes of hyperfinite type ${\rm III}$ factors, and isomorphism types of ergodic actions of $\bbR$ on abelian von Neumann algebras, provided one excludes the standard action of $\bbR$ on $L^\infty(\bbR)$.
(The latter is the flow of weights of the hyperfinite ${\rm II}_1$ and ${\rm II}_\infty$ factors.)

Given abelian von Neumann algebras $Z_1$ and $Z_2$ with actions of $\bbR$, we write $Z_1\wedge_\bbR Z_2:=(Z_1\,\bar\otimes\,Z_2)^{\bbR_{\rm diag}}$ for the fixed-point algebra with respect to $\bbR_{\rm diag}:=\{(t,-t):t\in \bbR\}\subset\bbR^2$,
along with the residual $\bbR^2/\bbR_{\rm diag}$ action.
The algebra $L^\infty(\bbR)$ with its standard $\bbR$ action is a unit for that operation: $Z\wedge_\bbR L^\infty(\bbR) = Z$.
Now, by \cite[Cor.\,II.6.8]{MR480760}, %Connes and Takesaki The flow of weights on factors of type III, , also applies when one is type ${\rm II}_1$}
given two factors $M_1$ and $M_2$, there is a canonical isomorphism
$S(M_1\,\bar\otimes\,M_2) \cong S(M_1) \wedge_\bbR S(M_2)$.\footnote{The result in \cite{MR480760} is only stated for type ${\rm III}$ factors, but the proof never uses the type ${\rm III}$ assumption.}
It follows that
\[
S(R\,\bar\otimes\, R_{{\rm II}_1}) \,\cong\, S(R) \wedge_\bbR S(R_{{\rm II}_1})
\,\cong\, S(R) \wedge_\bbR L^\infty(\bbR) \,\cong\, S(R).
\]
Using the Connes--Haagerup--Krieger classification theorem of hyperfinite type ${\rm III}$ factors, it follows that $R\,\bar\otimes\, R_{{\rm II}_1}\cong R$.
\end{proof}

\begin{thm}
\label{thm: exists faithful rep}
Let $R$ be a hyperfinite 
factor which is not of type ${\rm I}$.
Then every unitary fusion category $\cC$ admits a fully faithful representation $\cC\to\Bim(R)$.
\end{thm}

\begin{proof}
Let $R_{{\rm II}_1}$ be a hyperfinite ${\rm II}_1$ factor.
By the work of Popa \cite[Thm.\,3.1]{MR1334479} (see also \cite[Thm.\,4.1]{MR3028581}),
there exists a fully faithful representation
\[
\cC\hookrightarrow\Bim(R_{{\rm II}_1}).
\]
Let now $R$ be an arbitrary hyperfinite factor which is not of type ${\rm I}$.
By Lemma \ref{lem:R(x)II_1=R}, we have $R\,\bar\otimes\, R_{{\rm II}_1}\cong R$.
We may therefore compose the above embedding with the map
\[
\Bim(R_{{\rm II}_1})\stackrel{L^2R\otimes_\C-}{\lhook\joinrel\relbar\joinrel\relbar\joinrel\relbar\joinrel\longrightarrow} \Bim(R \,\bar\otimes\, R_{{\rm II}_1})\cong \Bim(R).
\qedhere
\]
\end{proof}

The above result raises the question of uniqueness.
We believe that the following conjecture should follow straightforwardly from Popa's uniqueness theorems for hyperfinite finite depth subfactors of types ${\rm II}_1$ \cite{MR1055708,MR1278111} and ${\rm III}_1$ \cite{MR1339767}.
However, we do not attempt to prove it here as it would take us too far afield.

\begin{conj}
Let $\cC$ be a unitary fusion category,
and let $R$ be a hyperfinite factor which is either of type ${\rm II}_1$ or ${\rm III}_1$.
Then any two fully faithful representations $\cC\to \Bim(R)$ are equivalent in the sense of Definition \ref{defn:Representation}. 
\end{conj}

%%%%%%%%%%%%%%%%%%%%%%%%%%%%%%%%%%%%%%%%%%%%%%%%%
%%%%%%%%%%%%%%%%%%%%%%%%%%%%%%%%%%%%%%%%%%%%%%%%%
%%%%%%%%%%%%%%%%%%%%%%%%%%%%%%%%%%%%%%%%%%%%%%%%%
\section{The commutant of a fusion category} \label{sec: The commutant of a fusion category}

Throughout this section, we fix a factor $R$ (not necessarily hyperfinite), a unitary fusion category $\cC$, and a representation $\cC\to \Bim(R)$.
To simplify the notation, we will assume that the representation is fully faithful and identify $\cC$ with its image in $\Bim(R)$,
but the fully faithfulness condition is actually not required for the results of this section. It will however be needed later on, in Section~\ref{sec:Absorbing objects}.

%%%%%%%%%%%%%%%%%%%%%%%%%%%%%%%%%%%%%%%%%%%%%%%%%
\subsection{Constructing objects in \texorpdfstring{$\cC'$}{C'}}
\label{sec: functorial construction of objects in C'}

The goal of this section is to construct a functor 
\[
\quad
\underline\Delta:\Bim(R)\to \cC'
\qquad
\underline\Delta(\Lambda)=(\Delta(\Lambda),e_{\Delta(\Lambda)}).
\]
For simplicity of notation, we will denote the underlying object $\Delta(\Lambda)$ of $\underline\Delta(\Lambda)$ simply by $\Delta$. It is given by
\begin{equation} \label{eq: def of Delta}
\Delta\, := \bigoplus_{x\in \Irr(\cC)} x\boxtimes \Lambda \boxtimes \overline x\,.
\end{equation}
Note that this object does not depend, up to canonical unitary isomorphism, on the choice of representatives of the simple objects of $\cC$.

For $a\in \cC$ an irreducible object, the half-braiding
$e_{\Delta,a}:\Delta\boxtimes a\to a\boxtimes \Delta$
is given by
\begin{equation}
\label{eq:HalfBraiding}
e_{\Delta, a} 
\,:=\,
\sum_{x,y\in \Irr(\cC)} \!
\sqrt{d_a^{-1}}\,
\begin{tikzpicture}[baseline=-.1cm]
	\draw[very thick] (0,-.6) -- (0,.6);
	\draw (-.2,-.6) -- (-.2,.6);
	\draw (.2,-.6) -- (.2,.6);
	\draw (-.6,.6) -- (-.2,.2);
	\draw (.6,-.6) -- (.2,-.2);
	\draw[fill=\betacolor] (-.2,.2) circle (.05cm);
	\draw[fill=\betacolor] (.2,-.2) circle (.05cm);
	\node at (0,-.8) {\scriptsize{$\Lambda$}};
	\node at (-.2,-.84) {\scriptsize{$x$}};
	\node at (.23,-.8) {\scriptsize{$\overline x$}};
	\node at (.6,-.84) {\scriptsize{$a$}};
	\node at (0,.8) {\scriptsize{$\Lambda$}};
	\node at (-.2,.77) {\scriptsize{$y$}};
	\node at (.23,.8) {\scriptsize{$\overline y$}};
	\node at (-.6,.77) {\scriptsize{$a$}};
\end{tikzpicture}
\end{equation}
where the projection $\Delta\boxtimes a\to x\boxtimes \Lambda\boxtimes \overline x\boxtimes a$ and inclusion $a\boxtimes y\boxtimes \Lambda\boxtimes \overline y\to a\boxtimes\Delta$ are implicit in the notation.
The half-braiding is natural with respect to morphisms $a\to a'$ between simple objects, and we extend it by additivity to all objects.

\begin{prop}
\label{Prop:eDeltaHalfBraiding}
$e_{\Delta}=(e_{\Delta,a}:\Delta\boxtimes a\to a\boxtimes \Delta)_{a\in \cC}$ is a unitary half-braiding.
\end{prop}

\begin{proof}
The maps $e_{\Delta,a}$ are natural in $a$ by construction.
To see that $e_{\Delta,a}$ is unitary, we use the Bigon and Fusion relations:
$$
e_{\Delta,a}^*\circ e_{\Delta,a}=\sum_{x,y,z\in \Irr(\cC)}\!
d_a^{-1}
\begin{tikzpicture}[baseline=-.1cm]
	\coordinate (b) at (-.2,.4);
	\coordinate (c) at (-.2,-.4);
	\coordinate (d) at (.2,-.6);
	\coordinate (e) at (.2,.6);
	\draw[very thick] (0,-1) -- (0,1);
	\draw (-.2,-1) -- (-.2,1);
	\draw (.2,-1) -- (.2,1);
	\draw (c) .. controls ++(135:.5cm) and ++(225:.5cm) .. (b);
	\draw (d) -- (.6,-1);
	\draw (e) -- (.6,1);
	\draw[fill=\alphacolor] (b) circle (.05cm);
	\draw[fill=\betacolor] (c) circle (.05cm);
	\draw[fill=\betacolor] (d) circle (.05cm);
	\draw[fill=\alphacolor] (e) circle (.05cm);
	\node at (0,-1.2) {\scriptsize{$\Lambda$}};
	\node at (-.2,-1.24) {\scriptsize{$x$}};
	\node at (.23,-1.2) {\scriptsize{$\overline x$}};
	\node at (.6,-1.24) {\scriptsize{$a$}};
	\node at (0,1.2) {\scriptsize{$\Lambda$}};
	\node at (-.2,1.17) {\scriptsize{$z$}};
	\node at (.23,1.2) {\scriptsize{$\overline z$}};
	\node at (.6,1.17) {\scriptsize{$a$}};
	\node at (-.6,0) {\scriptsize{$a$}};
	\node at (-.325,0) {\scriptsize{$y$}};
	\node at (.35,0) {\scriptsize{$\overline y$}};
\end{tikzpicture}
=
\sum_{x,y,z\in \Irr(\cC)}\!
\sqrt{d_yd_x^{-1}d_a^{-1}}\cdot\delta_{x,z}
\begin{tikzpicture}[baseline=-.1cm]
	\coordinate (d) at (.2,-.6);
	\coordinate (e) at (.2,.6);
	\draw[very thick] (0,-1) -- (0,1);
	\draw (-.2,-1) -- (-.2,1);
	\draw (.2,-1) -- (.2,1);
	\draw (d) -- (.6,-1);
	\draw (e) -- (.6,1);
	\draw[fill=\betacolor] (d) circle (.05cm);
	\draw[fill=\betacolor] (e) circle (.05cm);
	\node at (0,-1.2) {\scriptsize{$\Lambda$}};
	\node at (-.2,-1.24) {\scriptsize{$x$}};
	\node at (.23,-1.2) {\scriptsize{$\overline x$}};
	\node at (.6,-1.24) {\scriptsize{$a$}};
	\node at (0,1.2) {\scriptsize{$\Lambda$}};
	\node at (-.2,1.17) {\scriptsize{$x$}};
	\node at (.23,1.2) {\scriptsize{$\overline x$}};
	\node at (.6,1.17) {\scriptsize{$a$}};
	\node at (.35,0) {\scriptsize{$\overline y$}};
\end{tikzpicture}
=
\sum_{x\in \Irr(\cC)}
\begin{tikzpicture}[baseline=-.1cm]
	\draw[very thick] (0,-1) -- (0,1);
	\draw (-.2,-1) -- (-.2,1);
	\draw (.2,-1) -- (.2,1);
	\draw (.6,-1) -- (.6,1);
	\node at (0,-1.2) {\scriptsize{$\Lambda$}};
	\node at (-.2,-1.24) {\scriptsize{$x$}};
	\node at (.23,-1.2) {\scriptsize{$\overline x$}};
	\node at (.6,-1.24) {\scriptsize{$a$}};
	\node at (0,1.2) {\scriptsize{$\Lambda$}};
	\node at (-.2,1.17) {\scriptsize{$x$}};
	\node at (.23,1.2) {\scriptsize{$\overline x$}};
	\node at (.6,1.17) {\scriptsize{$a$}};
\end{tikzpicture}
\,.
$$
The verification that $e_{\Delta,a}\circ e_{\Delta,a}^*=\id_{a\boxtimes \Delta}$ is similar.

It remains to verify the `hexagon' axiom %(\ref{eq: hexagon axiom}):
$e_{\Delta,a\boxtimes b} =
(\id_{a}\boxtimes e_{\Delta,b}) \circ (e_{\Delta,a}\boxtimes \id_{b}$).
We do this with the help of the Fusion and I=H relations: %\eqref{eq:I=H}: 
\begin{align*}
\begin{tikzpicture}[baseline=-.1cm, xscale=1.2]
	\draw[very thick] (0,-.8) -- (0,.8);
	\draw (-.6,.8) -- (.8,-.8);
	\draw (-.8,.8) -- (.6,-.8);
\roundNbox{unshaded}{(0,0)}{.3}{.4}{.4}{$e_{\Delta,a\boxtimes b}$};
	\node at (0,1) {\scriptsize{$\Delta$}};
	\node at (0,-1) {\scriptsize{$\Delta$}};
	\node at (-.85,.97) {\scriptsize{$a$}};
	\node at (-.6,1) {\scriptsize{$b$}};
	\node at (.6,-1.02) {\scriptsize{$a$}};
	\node at (.85,-.99) {\scriptsize{$b$}};
\end{tikzpicture}
&=
\,\,\sum_{c\in \Irr(\cC)}
\sqrt{d_cd_a^{-1}d^{-1}}\hspace{-.6cm}{\phantom{d}}_b^{\phantom{-1}}
\begin{tikzpicture}[baseline=-.1cm, xscale=1.2]
	\draw[very thick] (0,-.8) -- (0,.8);
	\draw (.8,-.8) -- (.6,-.6) -- (.6,-.8)
    (.6,-.6) -- (.4,-.4)
    (.2,-.2) -- (.4,-.4) -- (.4,-.2);
	\draw (-.6,.8) -- (.1,0);
	\draw (-.8,.8) -- (-.1,0);
	\draw[fill=\betacolor] (.6,-.6) circle (.05cm);
	\draw[fill=\betacolor] (.4,-.4) circle (.05cm);
\roundNbox{unshaded}{(0,.1)}{.3}{.4}{.4}{$e_{\Delta,a\boxtimes b}$};
	\node at (0,1) {\scriptsize{$\Delta$}};
	\node at (0,-1) {\scriptsize{$\Delta$}};
	\node at (-.85,.97) {\scriptsize{$a$}};
	\node at (-.6,1) {\scriptsize{$b$}};
	\node at (.6,-1.02) {\scriptsize{$a$}};
	\node at (.85,-.99) {\scriptsize{$b$}};
	\node at (.37,-.67) {\scriptsize{$c$}};
\end{tikzpicture}
=
\sum_{c\in\Irr(\cC)}
\sqrt{d_cd_a^{-1}d^{-1}}\hspace{-.6cm}{\phantom{d}}_b^{\phantom{-1}}
\begin{tikzpicture}[baseline=-.1cm, xscale=.9]
	\draw[very thick] (0,-.8) -- (0,.8);
	\draw (-1,.8) -- (-.6,.5) -- (-.6,.8);
	\draw (-.6,.5) -- (-.2,.2);
	\draw (1,-.8) -- (.6,-.5) -- (.6,-.8);
	\draw (.6,-.5) -- (.2,-.2);
	\draw[fill=\betacolor] (-.6,.5) circle (.05cm);
	\draw[fill=\betacolor] (.6,-.5) circle (.05cm);
\roundNbox{unshaded}{(0,0)}{.28}{.23}{.23}{$e_{\Delta,c}$};
	\node at (0,-1) {\scriptsize{$\Delta$}};
	\node at (.6,-.97) {\scriptsize{$a$}};
	\node at (1,-.97) {\scriptsize{$b$}};
	\node at (0,1) {\scriptsize{$\Delta$}};
	\node at (-.6,.97) {\scriptsize{$b$}};
	\node at (-1,.97) {\scriptsize{$a$}};
    \node at (.31,-.49) {\scriptsize{$c$}};
	\node at (-.31,.49) {\scriptsize{$c$}};
\end{tikzpicture}=
\\
&=
\sum_{x,z,c\in\Irr(\cC)}
\sqrt{d_a^{-1}d^{-1}}\hspace{-.6cm}{\phantom{d}}_b^{\phantom{-1}}
\begin{tikzpicture}[baseline=-.1cm, xscale=.95]
	\draw[very thick] (0,-.8) -- (0,.8);
	\draw (-.2,-.8) -- (-.2,.8);
	\draw (.2,-.8) -- (.2,.8);
	\draw (-1,.8) -- (-.6,.5) -- (-.6,.8);
	\draw (-.6,.5) -- (-.2,.2);
	\draw (1,-.8) -- (.6,-.5) -- (.6,-.8);
	\draw (.6,-.5) -- (.2,-.2);
	\draw[fill=\betacolor] (-.6,.5) circle (.05cm);
	\draw[fill=\alphacolor] (-.2,.2) circle (.05cm);
	\draw[fill=\alphacolor] (.2,-.2) circle (.05cm);
	\draw[fill=\betacolor] (.6,-.5) circle (.05cm);
	\node at (0,-1) {\scriptsize{$\Lambda$}};
	\node at (-.2,-1.04) {\scriptsize{$x$}};
	\node at (.23,-1) {\scriptsize{$\overline x$}};
	\node at (.6,-.97) {\scriptsize{$a$}};
	\node at (1,-.97) {\scriptsize{$b$}};
	\node at (0,1) {\scriptsize{$\Lambda$}};
	\node at (-.2,.97) {\scriptsize{$z$}};
	\node at (.23,1) {\scriptsize{$\overline z$}};
	\node at (-.6,.97) {\scriptsize{$b$}};
	\node at (-1,.97) {\scriptsize{$a$}};
	\node at (.5,-.2) {\scriptsize{$c$}};
	\node at (-.5,.2) {\scriptsize{$c$}};
\end{tikzpicture}
=
\sum_{x,y,z\in \Irr(\cC)}
\sqrt{d_a^{-1}d^{-1}}\hspace{-.6cm}{\phantom{d}}_b^{\phantom{-1}}
\begin{tikzpicture}[baseline=-.1cm, xscale=1]
	\coordinate (b) at (-.2,.4);
	\coordinate (c) at (-.2,-.2);
	\coordinate (d) at (.2,-.4);
	\coordinate (e) at (.2,.2);
	\draw[very thick] (0,-.8) -- (0,.8);
	\draw (-.2,-.8) -- (-.2,.8);
	\draw (.2,-.8) -- (.2,.8);
	\draw (-.6,.8) -- (b);
	\draw (-1.2,.8) -- (c);
	\draw (d) -- (.6,-.8);
	\draw (e) -- (1.2,-.8);
	\draw[fill=\alphacolor] (b) circle (.05cm);
	\draw[fill=\betacolor] (c) circle (.05cm);
	\draw[fill=\betacolor] (d) circle (.05cm);
	\draw[fill=\alphacolor] (e) circle (.05cm);
	\node at (0,-1) {\scriptsize{$\Lambda$}};
	\node at (-.2,-1.04) {\scriptsize{$x$}};
	\node at (.23,-1) {\scriptsize{$\overline x$}};
	\node at (.6,-.97) {\scriptsize{$a$}};
	\node at (1.2,-.97) {\scriptsize{$b$}};
	\node at (0,1) {\scriptsize{$\Lambda$}};
	\node at (-.2,.97) {\scriptsize{$z$}};
	\node at (.23,1) {\scriptsize{$\overline z$}};
	\node at (-.6,.97) {\scriptsize{$b$}};
	\node at (-1.2,.97) {\scriptsize{$a$}};
	\node at (-.35,.2) {\scriptsize{$y$}};
	\node at (.35,-.2) {\scriptsize{$\overline y$}};
\end{tikzpicture}
\qedhere
\end{align*}
\end{proof}

\begin{prop}
The assignment $\Lambda \mapsto (\Delta, e_\Delta)$ defines a functor $\Bim(R)\to \cC'$.
\end{prop}
\begin{proof}
Given a morphism $f: \Lambda_1 \to \Lambda_2$ in $\Bim(R)$, we let
\[
\Delta(f) :=\, \sum \id_x\boxtimes f \boxtimes \id_{\overline x}\,:\,\Delta(\Lambda_1)\to\Delta(\Lambda_2).
\]
In order to check that this is a morphism in $\cC'$, we need to verify that
$e_{\Delta(\Lambda_2),a}\circ(\Delta(f)\boxtimes \id_a)=(\id_a\boxtimes \Delta(f))\circ e_{\Delta(\Lambda_1),a}$.
This is straightforward using the definition (\ref{eq:HalfBraiding}) of the half-braiding:
\[
\sum_{x,y\in \Irr(\cC)} \!\sqrt{d_a^{-1}}\,
\begin{tikzpicture}[baseline=-.1cm, xscale=1.15]
	\draw[very thick] (0,-.6) -- (0,.6);
	\draw (-.3,-.6) -- (-.3,.6);
	\draw (.3,-.6) -- (.3,.6);
	\draw (-.6,.6) -- (-.3,.2);
	\draw (.6,-.6) -- (.3,-.2);
	\draw[fill=\betacolor] (-.3,.2) circle (.05cm);
	\draw[fill=\betacolor] (.3,-.2) circle (.05cm);
	\draw[fill=black] (0,-.4) node[left, xshift=1]{$\scriptstyle f$} +(.05,.05) rectangle +(-.05,-.05);
	\node at (.05,-.8) {\scriptsize{$\Lambda_1$}};
	\node at (-.3,-.84) {\scriptsize{$x$}};
	\node at (.33,-.8) {\scriptsize{$\overline x$}};
	\node at (.6,-.84) {\scriptsize{$a$}};
	\node at (.05,.8) {\scriptsize{$\Lambda_2$}};
	\node at (-.3,.77) {\scriptsize{$y$}};
	\node at (.33,.8) {\scriptsize{$\overline y$}};
	\node at (-.6,.77) {\scriptsize{$a$}};
\end{tikzpicture}
\,=\,
\sum_{x,y\in \Irr(\cC)}  \!\sqrt{d_a^{-1}}\,
\begin{tikzpicture}[baseline=-.1cm, xscale=1.15]
	\draw[very thick] (0,-.6) -- (0,.6);
	\draw (-.3,-.6) -- (-.3,.6);
	\draw (.3,-.6) -- (.3,.6);
	\draw (-.6,.6) -- (-.3,.2);
	\draw (.6,-.6) -- (.3,-.2);
	\draw[fill=\betacolor] (-.3,.2) circle (.05cm);
	\draw[fill=\betacolor] (.3,-.2) circle (.05cm);
	\draw[fill=black] (0,.4) node[left, xshift=1]{$\scriptstyle f$} +(.05,.05) rectangle +(-.05,-.05);
	\node at (.05,-.8) {\scriptsize{$\Lambda_1$}};
	\node at (-.3,-.84) {\scriptsize{$x$}};
	\node at (.33,-.8) {\scriptsize{$\overline x$}};
	\node at (.6,-.84) {\scriptsize{$a$}};
	\node at (.05,.8) {\scriptsize{$\Lambda_2$}};
	\node at (-.3,.77) {\scriptsize{$y$}};
	\node at (.33,.8) {\scriptsize{$\overline y$}};
	\node at (-.6,.77) {\scriptsize{$a$}};
\end{tikzpicture}
.
\qedhere
\]
\end{proof}

\begin{rem}
The construction of $\underline\Delta(\Lambda)=(\Delta(\Lambda),e_{\Delta(\Lambda)})$ works under the greater generality of a rigid C*-tensor category (in particular semisimple) represented in $\Bim(R)$, not necessarily fully faithfully.
The half-braiding (\ref{eq:HalfBraiding}) is unitary by Proposition~\ref{Prop:eDeltaHalfBraiding}, and thus bounded.
\end{rem}

%%%%%%%%%%%%%%%%%%%%%%%%%%%%%%%%%%%%%%%%%%%%%%%%%
\subsection{The endomorphism algebra}\label{sec:The endomorphism algebra}

In this section, we fix a bimodule $\Lambda\in\Bim(R)$.
Our goal is to compute the endomorphism algebra of $\underline\Delta(\Lambda)$.
As in the previous section, we will write $\Delta$ for the underlying object of $\underline\Delta(\Lambda)$.

\begin{thm}\label{thm: endo}
The map that sends
$$
f\,=\,\,\big(f_a: \Lambda \boxtimes a \to a\boxtimes \Lambda\big)_{a\in\Irr(\cC)}
$$
to
$$
T_f :=
\sum_{a,x,y\in\Irr(\cC)}
\begin{tikzpicture}[baseline=-.1cm]
	\draw[very thick] (0,-.8) -- (0,.8);
	\draw (.6,-.4) -- (-.6,.4);
	\draw (.6,-.8) -- (.6,.8);
	\draw (-.6,-.8) -- (-.6,.8);
	\draw[fill=\betacolor] (-.6,.4) circle (.05cm);
	\draw[fill=\betacolor] (.6,-.4) circle (.05cm);
	\roundNbox{unshaded}{(0,0)}{.3}{0}{0}{$f_a$};
	\node at (-.6,1) {\scriptsize{$y$}};
	\node at (.6,1) {\scriptsize{$\overline{y}$}};
	\node at (-.6,-1.02) {\scriptsize{$x$}};
	\node at (.6,-1) {\scriptsize{$\overline{x}$}};
	\node at (-.4,.45) {\scriptsize{$a$}};
	\node at (.4,-.45) {\scriptsize{$a$}};
	\node at (0,1) {\scriptsize{$\Lambda$}};
	\node at (0,-1) {\scriptsize{$\Lambda$}};
\end{tikzpicture}\,:\,\Delta\to \Delta
$$
induces a vector space isomorphism
\begin{equation*} %\label{eq: End(Delta)}
\bigoplus_{a\in \Irr(\cC)} \Hom_{\Bim(R)}(\Lambda\boxtimes a, a\boxtimes\Lambda) \,\,\cong\,\, \End_{\cC'}(\underline\Delta(\Lambda)).
\end{equation*}

Under the above isomorphism, the left hand side %of (\ref{eq: End(Delta)}) 
acquires the following $*$-algebra structure: The $*$-operation is given by
\begin{equation}\label{eq: f*}
(f^*)_a\,:=\,
\begin{tikzpicture}[baseline=-.1cm]
	\draw[very thick] (0,-.7) -- (0,.7);
	\draw (-.7,.7) .. controls ++(-90:1.3cm) and ++(240:.6cm) .. (-.15,-.3);
	\draw (.7,-.7) .. controls ++(90:1.3cm) and ++(60:.6cm) .. (.15,.3);
	\roundNbox{unshaded}{(0,0)}{.3}{.2}{.2 }{$(f_{\overline a})^*$};
	\node at (-.7,.9) {\scriptsize{$a$}};
	\node at (.7,-.9) {\scriptsize{$a$}};
	\node at (0,.9) {\scriptsize{$\Lambda$}};
	\node at (0,-.9) {\scriptsize{$\Lambda$}};
\end{tikzpicture}
\end{equation}
and the product is given by
\begin{equation}\label{eq: (fg)_a}
(f{\cdot}g)_a\,:=
\sum_{b,c\in\Irr(\cC)}
\begin{tikzpicture}[baseline=-.1cm, xscale=1.2]
\draw (-.17,-.2) to[rounded corners=5] (-.45,.1) to[sharp corners] (-.45,1.1) -- (-.55,1.4);
\draw (-.15,.8) to[bend right=25] (-.45,1.1);
\draw (.17,.2) to[rounded corners=5] (.45,-.1) to[sharp corners] (.45,-1.1) -- (.55,-1.4);
\draw (.15,-.8) to[bend right=25] (.45,-1.1);
\draw[fill=\betacolor] (-.45,1.1) circle (.05cm);
\draw[fill=\betacolor] (.45,-1.1) circle (.05cm);
	\draw[very thick] (0,-1.4) -- (0,1.4);
\roundNbox{unshaded}{(0,.5)}{.3}{0}{0}{$f_b$};
\roundNbox{unshaded}{(0,-.5)}{.3}{0}{0}{$g_c$};
	\node at (-.58,1.6) {\scriptsize{$a$}};
	\node at (.58,-1.6) {\scriptsize{$a$}};
	\node at (-.25,1.25) {\scriptsize{$b$}};
	\node at (.25,-1.25) {\scriptsize{$c$}};
	\node at (-.58,.5) {\scriptsize{$c$}};
	\node at (.58,-.5) {\scriptsize{$b$}};
	\node at (0,1.6) {\scriptsize{$\Lambda$}};
	\node at (0,-1.6) {\scriptsize{$\Lambda$}};
\end{tikzpicture}\,.
\end{equation}
\end{thm}

\begin{rem}
The map $f_{\overline a}:\Lambda \boxtimes \overline a \to \overline a\boxtimes \Lambda$, which appears in the right hand side of (\ref{eq: f*}) requires the choice of an isomorphism between $\overline a$ and the unique element of $\Irr(\cC)$ to which it is isomorphic. It is important to note that, because $\overline a$ appears in both the domain and the codomain, the map $f_{\overline a}$ does not depend on that choice.
\end{rem}
\begin{rem}
If we take $\Lambda=\bigoplus_{x\in\Irr(C)} x$, then the two equations (\ref{eq: f*}) and (\ref{eq: (fg)_a}) are exactly the ones describing Ocneanu's tube algebra \cite{MR1642584, MR1832764}.
\end{rem}

\begin{proof}[Proof of Theorem \ref{thm: endo}]
We begin by checking, using the I=H relation, that the formula $(\id_b\boxtimes T_f)\circ e_{\Delta,b}=e_{\Delta,b}\circ(T_f\boxtimes \id_b)$ holds:
$$
\sum_{a,x,y\in\Irr(\cC)}
%\sqrt{d_b^{-1}}
\sqrt{d^{-1}}\hspace{-.6cm}{\phantom{d}}_b^{\phantom{-1}}
\begin{tikzpicture}[baseline=-.4cm, xscale=1.1]
	\draw[very thick] (0,-1.4) -- (0,.8);
	\draw (.4,-.4) -- (-.4,.4);
	\draw (.4,-1.4) -- (.4,.8);
	\draw (-.4,-1.4) -- (-.4,.8);
	\draw (.4,-1) -- (.8,-1.4);
	\draw (-.4,-.2) -- (-1.4,.8);
	\draw[fill=\alphacolor] (-.4,.4) circle (.05cm);
	\draw[fill=\alphacolor] (.4,-.4) circle (.05cm);
	\draw[fill=\betacolor] (-.4,-.2) circle (.05cm);
	\draw[fill=\betacolor] (.4,-1) circle (.05cm);
	\draw[fill=black] (0,0)node[left, xshift=2, yshift=-1.5]{$\scriptstyle f_{\!\!\;a}$} +(.05,.05) rectangle +(-.05,-.05);
	\node at (-.4,1) {\scriptsize{$z$}};
	\node at (.4,1) {\scriptsize{$\overline{z}$}};
	\node at (-.5,.1) {\scriptsize{$y$}};
	\node at (.55,-.7) {\scriptsize{$\overline{y}$}};
	\node at (-.4,-1.6) {\scriptsize{$x$}};
	\node at (.4,-1.6) {\scriptsize{$\overline{x}$}};
	\node at (-.2,.45) {\scriptsize{$a$}};
	\node at (.2,-.45) {\scriptsize{$a$}};
	\node at (0,1) {\scriptsize{$\Lambda$}};
	\node at (0,-1.6) {\scriptsize{$\Lambda$}};
	\node at (-1.4,1) {\scriptsize{$b$}};
	\node at (.8,-1.6) {\scriptsize{$b$}};
\end{tikzpicture}
=
\sum_{a,x,y\in\Irr(\cC)}
%\sqrt{d_b^{-1}}
\sqrt{d^{-1}}\hspace{-.6cm}{\phantom{d}}_b^{\phantom{-1}}
\begin{tikzpicture}[baseline=.2cm, xscale=1.1]
	\draw[very thick] (0,1.4) -- (0,-.8);
	\draw (.4,-.4) -- (-.4,.4);
	\draw (.4,1.4) -- (.4,-.8);
	\draw (-.4,1.4) -- (-.4,-.8);
	\draw (-.4,1) -- (-.8,1.4);
	\draw (.4,.2) -- (1.4,-.8);
	\draw[fill=\alphacolor] (-.4,.4) circle (.05cm);
	\draw[fill=\alphacolor] (.4,-.4) circle (.05cm);
	\draw[fill=\betacolor] (.4,.2) circle (.05cm);
	\draw[fill=\betacolor] (-.4,1) circle (.05cm);
	\draw[fill=black] (0,0)node[left, xshift=2, yshift=-1.5]{$\scriptstyle f_{\!\!\;a}$} +(.05,.05) rectangle +(-.05,-.05);
	\node at (-.4,1.6) {\scriptsize{$z$}};
	\node at (.4,1.6) {\scriptsize{$\overline{z}$}};
	\node at (-.5,.7) {\scriptsize{$y$}};
	\node at (.55,-.2) {\scriptsize{$\overline{y}$}};
	\node at (-.4,-1) {\scriptsize{$x$}};
	\node at (.4,-1) {\scriptsize{$\overline{x}$}};
	\node at (-.2,.45) {\scriptsize{$a$}};
	\node at (.2,-.45) {\scriptsize{$a$}};
	\node at (0,-1) {\scriptsize{$\Lambda$}};
	\node at (0,1.6) {\scriptsize{$\Lambda$}};
	\node at (-.8,1.6) {\scriptsize{$b$}};
	\node at (1.4,-1) {\scriptsize{$b$}};
\end{tikzpicture}.
$$
This ensures that $T_f\in\End_{\cC'}(\underline\Delta(\Lambda))$.

We now show that the map 
$\bigoplus_{a\in Irr(\cC)} \Hom(\Lambda\boxtimes a, a\boxtimes\Lambda) \to \End_{\cC'}(\Delta)$ 
given by $f\mapsto T_f$ is an isomorphism.
For that, we define a map the other way as follows.
It sends $T\in \End_{\cC'}(\underline\Delta(\Lambda))$ to the element $f_T =(f_{T,a}: \Lambda \boxtimes a \to a \boxtimes \Lambda)$ given by
$$
f_{T,a} \,:= \,
\dim(\cC)^{-1}\!
\sum_{x,y\in\Irr(\cC)}
%\frac{d_a\sqrt{d_x d_y}}{\dim(\cC)}\,
\begin{tikzpicture}[baseline=-.1cm, xscale=1.1]
	\draw[very thick] (0,-1) -- (0,1);
	\draw (-.15,-.3) arc (0:-180:.15cm) -- (-.45,.6) -- (-.6,1);
	\draw (-.15,.3) arc (0:90:.3cm);
	\draw (.15,.3) arc (180:0:.15cm) -- (.45,-.6) -- (.6,-1);
	\draw (.15,-.3) arc (180:270:.3cm);
	\draw[fill=\betacolor] (-.45,.6) circle (.05cm);
	\draw[fill=\betacolor] (.45,-.6) circle (.05cm);
	\roundNbox{unshaded}{(0,0)}{.3}{0}{0}{$T$}
	\node at (0,1.2) {\scriptsize{$\Lambda$}};
	\node at (0,-1.2) {\scriptsize{$\Lambda$}};
	\node at (-.6,1.2) {\scriptsize{$a$}};
	\node at (.6,-1.2) {\scriptsize{$a$}};
	\node at (-.2,.7) {\scriptsize{$y$}};
	\node at (.63,-.05) {\scriptsize{$y$}};
	\node at (.22,.7) {\scriptsize{$\overline y$}};
	\node at (-.65,0) {\scriptsize{$\overline x$}};
	\node at (.22,-.7) {\scriptsize{$\overline x$}};
	\node at (-.22,-.63) {\scriptsize{$x$}};
\end{tikzpicture}
$$
We now check that these two maps are each other's inverses.
The equation $f_{T_f}=f$ is an easy consequence of Lemma \ref{lem:Sumdxdy}:
$$
f_{T_{f},a}
\,=\,\,
\dim(\cC)^{-1}
\sum_{x,y,b}\,
\begin{tikzpicture}[baseline=-.1cm]
	\draw[very thick] (0,-.9) -- (0,.9);
	\draw (.6,-.4) -- (-.6,.4);
	\draw (.9,-.7) -- (1.2,-.9);
	\draw (-.9,.7) -- (-1.2,.9);
	\draw (-.75,.55) circle (.21cm);
	\draw (.75,-.55) circle (.21cm);
	\draw[fill=\betacolor] (-.6,.4) circle (.05cm);
	\draw[fill=\betacolor] (.6,-.4) circle (.05cm);
	\draw[fill=\alphacolor] (-.9,.7) circle (.05cm);
	\draw[fill=\alphacolor] (.9,-.7) circle (.05cm);
	\roundNbox{unshaded}{(0,0)}{.3}{0}{0}{$f_b$};
	\node at (-1.2,1.1) {\scriptsize{$a$}};
	\node at (1.2,-1.1) {\scriptsize{$a$}};
	\node at (1.1,-.4) {\scriptsize{$y$}};
	\node at (-.4,.7) {\scriptsize{$y$}};
\node at (-1.15,.4) {\scriptsize{$\overline x$}};
\node at (.35,-.7) {\scriptsize{$\overline x$}};
	\node at (-.5,.1) {\scriptsize{$b$}};
	\node at (.5,-.1) {\scriptsize{$b$}};
	\node at (0,1.1) {\scriptsize{$\Lambda$}};
	\node at (0,-1.1) {\scriptsize{$\Lambda$}};
\end{tikzpicture}
%\!=\,\,
%\sum_{x,y,b}\,
%\delta_{a,b}
%\frac{d_x d_y N_{\overline x,y}^a}{d_a\;\!\dim(\cC)}
%\begin{tikzpicture}[baseline=-.1cm]
%	\draw[very thick] (0,-.9) -- (0,.9);
%	\draw (.4,-.9) -- (-.4,.9);
%	\roundNbox{unshaded}{(0,0)}{.3}{0}{0}{$f_a$};
%
%	\node at (-.4,1.1) {\scriptsize{$a$}};
%	\node at (.4,-1.1) {\scriptsize{$a$}};
%	\node at (0,1.1) {\scriptsize{$\Lambda$}};
%	\node at (0,-1.1) {\scriptsize{$\Lambda$}};
%\end{tikzpicture}
=\,\,
f_a\,.
$$
For the other direction, we need to check that $T_{f_T}=T$ holds for every $T\in \End_{\cC'}(\underline\Delta(\Lambda))$:
\begin{align*}
T_{f_T}\,&=\,\,
\dim(\cC)^{-1}
\sum_{a,b,c,x,y}
\begin{tikzpicture}[baseline=-.1cm]
	\draw[very thick] (0,-1.2) -- (0,1.2);
	\draw (-.8,-1.2) -- (-.8,1.2);
	\draw (.8,-1.2) -- (.8,1.2);
	\draw (-.15,-.3) arc (0:-180:.15cm) -- (-.45,.6) -- (-.8,.8);
	\draw (-.15,.3) arc (0:90:.3cm);
	\draw (.15,.3) arc (180:0:.15cm) -- (.45,-.6) -- (.8,-.8);
	\draw (.15,-.3) arc (180:270:.3cm);
	\draw[fill=\betacolor] (-.45,.6) circle (.05cm);
	\draw[fill=\betacolor] (.45,-.6) circle (.05cm);
	\draw[fill=\alphacolor] (-.8,.8) circle (.05cm);
	\draw[fill=\alphacolor] (.8,-.8) circle (.05cm);
	\roundNbox{unshaded}{(0,0)}{.3}{0}{0}{$T$}
	\node at (0,1.4) {\scriptsize{$\Lambda$}};
	\node at (0,-1.4) {\scriptsize{$\Lambda$}};
	\node at (-.6,.9) {\scriptsize{$b$}};
	\node at (.6,-.9) {\scriptsize{$b$}};
	\node at (-.15,.7) {\scriptsize{$y$}};
	\node at (.2,.7){\scriptsize{$\overline y$}};
	\node at (-.2,-.7) {\scriptsize{$x$}};
	\node at (.2,-.72){\scriptsize{$\overline x$}};
	\node at (-.8,-1.4) {\scriptsize{$a$}};
	\node at (.8,-1.4){\scriptsize{$\overline a$}};
	\node at (-.8,1.4) {\scriptsize{$c$}};
	\node at (.8,1.4){\scriptsize{$\overline c$}};
\end{tikzpicture}
=\,
\dim(\cC)^{-1}
\sum_{a,b,c,x,y}
\begin{tikzpicture}[baseline=-.1cm]
	\draw[very thick] (0,-1.2) -- (0,1.2);
	\draw (-.15,-1.2) -- (-.15,1.2);
	\draw (.15,-1.2) -- (.15,1.2);
	\draw (-.15,.7) .. controls ++(180:.7cm) and ++(180:.7cm) .. (-.15,-.7);
	\draw (.15,.7) .. controls ++(0:.7cm) and ++(0:.7cm) .. (.15,-.7);
	\draw[fill=\betacolor] (-.15,-.7) circle (.05cm);
	\draw[fill=\alphacolor] (-.15,.7) circle (.05cm);
	\draw[fill=\betacolor] (.15,-.7) circle (.05cm);
	\draw[fill=\alphacolor] (.15,.7) circle (.05cm);
	\roundNbox{unshaded}{(0,0)}{.3}{0}{0}{$T$}
	\node at (0,1.4) {\scriptsize{$\Lambda$}};
	\node at (0,-1.4) {\scriptsize{$\Lambda$}};
	\node at (-.85,0) {\scriptsize{$b$}};
	\node at (.88,0){\scriptsize{$\overline b$}};
	\node at (-.3,.5) {\scriptsize{$y$}};
	\node at (.32,.48){\scriptsize{$\overline y$}};
	\node at (-.3,-.5) {\scriptsize{$x$}};
	\node at (.32,-.5){\scriptsize{$\overline x$}};
	\node at (-.25,1.37) {\scriptsize{$c$}};
	\node at (.22,1.4){\scriptsize{$\overline c$}};
	\node at (-.2,-1.43) {\scriptsize{$a$}};
	\node at (.22,-1.4){\scriptsize{$\overline a$}};
\end{tikzpicture}
\\
&=\,
\dim(\cC)^{-1}
\sum_{a,b,c,x,y}
\begin{tikzpicture}[baseline=-.1cm]
	\draw[very thick] (0,-1.2) -- (0,1.2);
	\draw (-.15,-1.2) -- (-.15,1.2);
	\draw (.15,-1.2) -- (.15,1.2);
	\draw (-.15,.8) .. controls ++(180:.5cm) and ++(180:.5cm) .. (-.15,0);
	\draw (.15,.8) .. controls ++(0:.5cm) and ++(0:.5cm) .. (.15,0);
	\draw[fill=\betacolor] (-.15,0) circle (.05cm);
	\draw[fill=\alphacolor] (-.15,.8) circle (.05cm);
	\draw[fill=\betacolor] (.15,0) circle (.05cm);
	\draw[fill=\alphacolor] (.15,.8) circle (.05cm);
	\roundNbox{unshaded}{(0,-.7)}{.3}{0}{0}{$T$}
	\node at (0,1.4) {\scriptsize{$\Lambda$}};
	\node at (0,-1.4) {\scriptsize{$\Lambda$}};
	\node at (-.7,.41) {\scriptsize{$b$}};
	\node at (.7,.4){\scriptsize{$\overline b$}};
	\node at (-.3,.38) {\scriptsize{$y$}};
	\node at (.31,.4){\scriptsize{$\overline y$}};
	\node at (-.3,-.2) {\scriptsize{$x$}};
	\node at (.32,-.2){\scriptsize{$\overline x$}};
	\node at (-.25,1.37) {\scriptsize{$c$}};
	\node at (.22,1.4){\scriptsize{$\overline c$}};
	\node at (-.2,-1.43) {\scriptsize{$a$}};
	\node at (.22,-1.4){\scriptsize{$\overline a$}};
\end{tikzpicture}
%=
%\sum_{a,b,c,x,y}
%\delta_{x,c}
%\frac{d_b d_y N_{b,y}^c}{d_c\;\! \dim(\cC)}\,
%\begin{tikzpicture}[baseline=-.1cm]
%	\draw[very thick] (0,-.6) -- (0,.6);
%	\draw (-.15,-.6) -- (-.15,.6);
%	\draw (.15,-.6) -- (.15,.6);
%	\roundNbox{unshaded}{(0,0)}{.3}{0}{0}{$T$}
%	\node at (0,.8) {\scriptsize{$\Lambda$}};
%	\node at (0,-.8) {\scriptsize{$\Lambda$}};
%	\node at (-.25,.77) {\scriptsize{$c$}};
%	\node at (.22,.8){\scriptsize{$\overline c$}};
%	\node at (-.2,-.83) {\scriptsize{$a$}};
%	\node at (.22,-.8){\scriptsize{$\overline a$}};
%\end{tikzpicture}
\,=\,T.
\end{align*}
Here, we have used the I=H relation, followed by the fact that $T$ commutes with (a scalar multiple of) the half-braiding,
%(note how this replaces the scalar $\sqrt{d_a d_x}$ by $\sqrt{d_x d_y}$),
and finally Lemma \ref{lem:Sumdxdy}.

At last, we check that the isomorphism 
$\bigoplus_{a\in\Irr(\cC)} \Hom(\Lambda\boxtimes a, a\boxtimes\Lambda) \cong \End_{\cC'}(\underline\Delta(\Lambda))$ 
is compatible with the $*$-operation (\ref{eq: f*}) and the multiplication (\ref{eq: (fg)_a}):
\[
(T_{f})^{*}\,=\,\,\sum_{a,x,y}\,
\begin{tikzpicture}[baseline=-.1cm, xscale=-1]
	\draw[very thick] (0,-.8) -- (0,.8);
	\draw (.6,-.45) -- (-.6,.45);
	\draw (.6,-.8) -- (.6,.8);
	\draw (-.6,-.8) -- (-.6,.8);
	\draw[fill=\betacolor] (-.6,.45) circle (.05cm);
	\draw[fill=\betacolor] (.6,-.45) circle (.05cm);
	\roundNbox{unshaded}{(0,0)}{.3}{.18}{.18}{$(f_a)^*$};
	\node at (-.6,1) {\scriptsize{$\overline{y}$}};
	\node at (.6,1) {\scriptsize{$y$}};
	\node at (-.6,-1) {\scriptsize{$\overline{x}$}};
	\node at (.6,-1) {\scriptsize{$x$}};
	\node at (-.4,.5) {\scriptsize{$a$}};
	\node at (.4,-.5) {\scriptsize{$a$}};
	\node at (0,1) {\scriptsize{$\Lambda$}};
	\node at (0,-1) {\scriptsize{$\Lambda$}};
\end{tikzpicture}
\,=\,\,
\sum_{a,x,y}\,
\begin{tikzpicture}[baseline=-.1cm]
	\draw[very thick] (0,-.8) -- (0,.8);
	\draw (-.8,.4) .. controls ++(-45:.2cm) and ++(240:.7cm) .. (-.18,-.3);
	\draw (.8,-.4) .. controls ++(135:.2cm) and ++(60:.7cm) .. (.18,.3);
	\draw (.8,-.8) -- (.8,.8);
	\draw (-.8,-.8) -- (-.8,.8);
	\draw[fill=\betacolor] (-.8,.4) circle (.05cm);
	\draw[fill=\betacolor] (.8,-.4) circle (.05cm);
	\roundNbox{unshaded}{(0,0)}{.3}{.18}{.18}{$(f_a)^*$};
	\node at (-.8,1) {\scriptsize{$y$}};
	\node at (.8,1) {\scriptsize{$\overline{y}$}};
	\node at (-.8,-1) {\scriptsize{$x$}};
	\node at (.8,-1) {\scriptsize{$\overline{x}$}};
\node at (-.6,.3) {\scriptsize{$\overline{a}$}};
\node at (.6,-.3) {\scriptsize{$\overline{a}$}};
	\node at (0,1) {\scriptsize{$\Lambda$}};
	\node at (0,-1) {\scriptsize{$\Lambda$}};
\end{tikzpicture}
\,=\,\,
\sum_{a,x,y}\,
\begin{tikzpicture}[baseline=-.1cm]
	\draw[very thick] (0,-.8) -- (0,.8);
	\draw (.6,-.45) -- (-.6,.45);
	\draw (.6,-.8) -- (.6,.8);
	\draw (-.6,-.8) -- (-.6,.8);
	\draw[fill=\betacolor] (-.6,.45) circle (.05cm);
	\draw[fill=\betacolor] (.6,-.45) circle (.05cm);
	\roundNbox{unshaded}{(0,0)}{.3}{.19}{.21}{$(f^*\!\!\;)_{\overline a}$};
	\node at (-.6,1) {\scriptsize{$y$}};
	\node at (.6,1) {\scriptsize{$\overline{y}$}};
	\node at (-.6,-1) {\scriptsize{$x$}};
	\node at (.6,-1) {\scriptsize{$\overline{x}$}};
\node at (-.4,.55) {\scriptsize{$\overline a$}};
\node at (.4,-.55) {\scriptsize{$\overline a$}};
	\node at (0,1) {\scriptsize{$\Lambda$}};
	\node at (0,-1) {\scriptsize{$\Lambda$}};
\end{tikzpicture}
\,=\, T_{f^*}
\]
\[
T_f\circ T_g\,=
\sum_{a,b,x,y,z}\,
\begin{tikzpicture}[baseline=-.1cm]
	\draw[very thick] (0,-1) -- (0,1);
	\draw (.4,-.8+.06) -- (-.4,-.2+.06);
	\draw (.4,.2-.06) -- (-.4,.8-.06);
	\draw (.4,-1) -- (.4,1);
	\draw (-.4,-1) -- (-.4,1);
\draw[fill=\betacolor] (-.4,-.2+.06) circle (.05cm);
\draw[fill=\betacolor] (.4,-.8+.06) circle (.05cm);
\draw[fill=\alphacolor] (-.4,.8-.06) circle (.05cm);
\draw[fill=\alphacolor] (.4,.2-.06) circle (.05cm);
	\draw[fill=black] (0,-.5+.06)node[left, xshift=2.3, yshift=-2.5]{$\scriptstyle g_{\!\!\;a}$} +(.05,.05) rectangle +(-.05,-.05);
	\draw[fill=black] (0,.5-.06)node[left, xshift=2.3, yshift=-2.5]{$\scriptstyle f_{\!\!\;b}$} +(.05,.05) rectangle +(-.05,-.05);
	\node at (-.4,1.2) {\scriptsize{$z$}};
	\node at (.4,1.2) {\scriptsize{$\overline{z}$}};
	\node at (-.4,-1.2) {\scriptsize{$x$}};
	\node at (.4,-1.2) {\scriptsize{$\overline{x}$}};
	\node at (-.6,.3) {\scriptsize{$y$}};
	\node at (.6,-.3) {\scriptsize{$\overline{y}$}};
	\node at (-.19,-.18+.06) {\scriptsize{$a$}};
	\node at (.22,-.44+.06) {\scriptsize{$a$}};
	\node at (-.2,.88-.06) {\scriptsize{$b$}};
	\node at (.22,.58-.06) {\scriptsize{$b$}};
	\node at (0,1.2) {\scriptsize{$\Lambda$}};
	\node at (0,-1.2) {\scriptsize{$\Lambda$}};
\end{tikzpicture}
\,=\,
\sum_{a,b,c,x,z}\,
\begin{tikzpicture}[baseline=-.1cm]
	\draw[very thick] (0,-1) -- (0,1);
	\draw (0,-.3) -- (-.4,.3) -- (0,.3);
	\draw (0,-.3) -- (.4,-.3) -- (0,.3);
	\draw (-.4,.3) -- (-.8,.6);
	\draw (.4,-.3) -- (.8,-.6);
	\draw (.8,-1) -- (.8,1);
	\draw (-.8,-1) -- (-.8,1);
	\draw[fill=\alphacolor] (.8,.-.6) circle (.05cm);
	\draw[fill=\betacolor] (.4,-.3) circle (.05cm);
	\draw[fill=\betacolor] (-.4,.3) circle (.05cm);
	\draw[fill=\alphacolor] (-.8,.6) circle (.05cm);
	\draw[fill=black] (0,-.3)node[left, xshift=2.3, yshift=-3.5]{$\scriptstyle g_{\!\!\;a}$} +(.05,.05) rectangle +(-.05,-.05);
	\draw[fill=black] (0,.3)node[right, xshift=-2.3, yshift=5.5]{$\scriptstyle f_{\!\!\;b}$} +(.05,.05) rectangle +(-.05,-.05);
	\node at (-.8,1.2) {\scriptsize{$z$}};
	\node at (.8,1.2) {\scriptsize{$\overline{z}$}};
	\node at (-.8,-1.2) {\scriptsize{$x$}};
	\node at (.8,-1.2) {\scriptsize{$\overline{x}$}};
	\node at (-.65,.25) {\scriptsize{$c$}};
	\node at (.55,-.65) {\scriptsize{$c$}};
	\node at (-.3,-.1) {\scriptsize{$a$}};
	\node at (.2,-.5) {\scriptsize{$a$}};
	\node at (-.25,.5) {\scriptsize{$b$}};
	\node at (.3,.15) {\scriptsize{$b$}};
	\node at (0,1.2) {\scriptsize{$\Lambda$}};
	\node at (0,-1.2) {\scriptsize{$\Lambda$}};
\end{tikzpicture}
=\,
T_{f{\cdot}g}\,.
\]
Here, the last line's middle equality follows from the I=H relation.
\end{proof}

\begin{rem}
The \vspace{.1cm} map $f\mapsto T_f:\bigoplus_{a\in \Irr(\cC)} \Hom_{\Bim(R)}(\Lambda\boxtimes a, a\boxtimes\Lambda)\to \End_{\cC'}(\underline\Delta(\Lambda))$
makes sense in the greater generality of a rigid \vspace{-.13cm} C*-tensor category represented in $\Bim(R)$.
In particular, the operator $T_f$ is always bounded (this follows from 
$
{\scriptstyle\sqrt{d_a^{-1}}}\!
\raisebox{.3ex}{$\underset{\scriptscriptstyle x,y\in \Irr(\cC)}{\sum}$}
\begin{tikzpicture}[baseline=0cm, scale=.6]
	\draw[very thick] (0,-.6) -- (0,.6);
	\draw (-.3,-.6) -- (-.3,.6);
	\draw (.6,0)-- (.3,.6) (-.3-.3,-.6) -- (-.6-.3,0);
	\draw (.6,-.6) -- (.6,.6);
	\draw (-.6-.3,-.6) -- (-.6-.3,.6);
	\draw[fill=\betacolor] (-.6-.3,0) circle (.07cm);
	\draw[fill=\betacolor] (.6,0) circle (.07cm);
\node at (-.6-.3,.8) {{$\scriptscriptstyle y$}};
\node at (.6,.8) {{$\scriptscriptstyle \overline{y}$}};
\node at (-.6-.3,-.83) {{$\scriptscriptstyle x$}};
\node at (.6,-.8) {{$\scriptscriptstyle \overline{x}$}};
\node at (-.3-.3,-.83) {{$\scriptscriptstyle a$}};
\node at (.3,.8) {{$\scriptscriptstyle a$}};
\node at (0,.8) {{$\scriptscriptstyle \Lambda$}};
\node at (0,-.8) {{$\scriptscriptstyle \Lambda$}};
\node at (-.3,.83) {{$\scriptscriptstyle\overline{a}$}};
\node at (-.3,-.8) {{$\scriptscriptstyle\overline{a}$}};
\end{tikzpicture}
$
being\vspace{-.3cm} unitary, and hence bounded).
\end{rem}

%%%%%%%%%%%%%%%%%%%%%%%%%%%%%%%%%%%%%%%%%%%%%%%%%
%%%%%%%%%%%%%%%%%%%%%%%%%%%%%%%%%%%%%%%%%%%%%%%%%
%%%%%%%%%%%%%%%%%%%%%%%%%%%%%%%%%%%%%%%%%%%%%%%%%
\section{Absorbing objects}\label{sec:Absorbing objects}

A tensor category $\cC$ has \emph{no zero-divisors}
if for every non-zero object $X$ and every objects $Y_1, Y_2$,
the maps
\begin{equation*} %\label{eq: faithful object}
\Hom(Y_1,Y_2)\to\Hom(X\otimes Y_1,X\otimes Y_2)
\quad\text{and}\quad
\Hom(Y_1,Y_2)\to\Hom(Y_1\otimes X,Y_2\otimes X)
\end{equation*}
are injective.
Note that for categories with involutions, it is enough to check that
one of the above maps is injective.

\begin{ex}
The tensor category $\Bim(R)$ has no zero-divisors.
Indeed, since $R$ is a factor, every non-zero module is faithful, and
the claim follows from Lemma~\ref{lem: fusion is faithful}.
\end{ex}

\begin{ex}
Fusion categories have no zero-divisors.
To see that, consider an object $X$ and a morphism $f:Y_1\to Y_2$ such that $\id_X\otimes f=0$.
We need to show that $X\not\cong 0$ implies $f=0$.
Since $X$ is non-zero,
$\ev_X$ is an epimorphism (indeed a projection onto a direct summand).
The morphism $\ev_X\otimes \id_{Y_1}$ is then also an an epimorphism, and we may reason as follows:
\[
f\circ(\underbrace{\ev_X\otimes \id_{Y_1}}_{\text{epi.}})
= \ev_X\otimes f
= (\ev_X\otimes 1_{Y_2})\circ(\id_{X^\vee}\otimes\underbrace{\id_X\otimes f}_{=0})=0
\quad\Rightarrow\quad f=0.
\]
\end{ex}

\begin{defn}
Let $\cC$ be a tensor category with no zero-divisors.
A non-zero object $X$ is called
\begin{itemize}
\item
\emph{right absorbing} if for every non-zero object $Y\in \cC$, we have $X\otimes Y\cong X$,
\item
\emph{left absorbing} if for every non-zero object $Y\in \cC$, we have $Y\otimes X\cong X$, and
\item
\emph{absorbing} if $X$ is both right and left absorbing.
\end{itemize}
\end{defn}

Clearly, if $\cC$ admits an absorbing object, then such an object is unique up to (non-canonical) isomorphism.
Note also that if a category has both right absorbing and left absorbing objects, then any such object is in fact absorbing.

If $\cC$ is equipped with a conjugation, then $X$ is right absorbing if and only if $\overline{X}$ is left absorbing.
In this case, any right absorbing object is automatically absorbing, and isomorphic to its conjugate.
By taking $Y=1\oplus 1$, we can also readily see that any absorbing object satisfies $X\oplus X \cong X$.\bigskip

Let $\mathsf{Hilb}$ be the category of separable Hilbert spaces.

\begin{ex}
The Hilbert space $\ell^2(\mathbb N)$ is absorbing in $\mathsf{Hilb}$.
\end{ex}

\begin{ex}
If $\cC$ is a unitary fusion category, then
the object
\[
\bigoplus_{x\in\Irr(\cC)} x\otimes \ell^2(\mathbb N)
\]
of $\cC\otimes_{\mathsf{Vec}}\mathsf{Hilb}$ is absorbing.
Indeed, for any simple objects $y$ and $z$ of $\cC$, there exists an $x$ such that $z$ occurs as a summand of $x\otimes y$.
The object $y\otimes (\bigoplus_{x\in\Irr(\cC)} x)$ therefore contains each simple object at least once. It follows that $y\otimes (\bigoplus_{x\in\Irr(\cC)} x\otimes \ell^2(\mathbb N))$ contains each simple object infinitely many times.
The same remains true when $y$ gets replaced by an arbitrary non-zero object of $\cC\otimes_{\mathsf{Vec}}\mathsf{Hilb}$.
\end{ex}

\begin{ex}
Let $G$ be an infinite countable group, and let 
$\Rep(G)$ denote the category of unitary representation of $G$ whose underlying Hilbert spaces is separable.
Then
\[
\ell^2(G)\otimes \ell^2(\mathbb N)
\]
is absorbing in $\Rep(G)$.
Indeed, if $V$ is a unitary representation with orthonormal basis $\{v_i\}_{i\in I}$, then 
$e_g\otimes e_i \mapsto (g\cdot v_i)\otimes e_g$
defines a unitary isomorphism
$\ell^2(G)\otimes \ell^2(I)\to V\otimes \ell^2(G)$.
It follows that $V\otimes \ell^2(G)\otimes \ell^2(\mathbb N)\cong\ell^2(G)\otimes \ell^2(I\times \N)\cong\ell^2(G)\otimes \ell^2(\mathbb N)$.
\end{ex}

Let $R$ be a separable factor and let $\Bim(R)$ be the category of $R$-$R$-bimodules whose underlying Hilbert space is separable.
Let also $\mathrm{Mod}(R)$ be the category of left $R$-modules whose underlying Hilbert space is separable.
We say that $H\in\mathrm{Mod}(R)$ is \emph{infinite} if it is non-zero and satisfies $H\oplus H\cong H$.
It is well known that an infinite module exists, and is unique up to isomorphism.

\begin{ex}
\label{ex:Existence}
The bimodule
\[
{\sb{R}L^2(R)}\otimes \ell^2(\bbN) \otimes L^2(R)_R
\]
is absorbing in $\Bim(R)$.
To see that, let $\sb{R}H_R \in \Bim(R)$ be any non-zero bimodule.
The following two modules are infinite, and therefore isomorphic:
$\sb{R} H \boxtimes_R L^2(R) \otimes \ell^2(\bbN)$
and
${\sb{R}L^2(R)} \otimes \ell^2(\bbN)$.
It follows that
$\sb{R} H \boxtimes_R L^2(R) \otimes \ell^2(\bbN)\otimes L^2(R)_R \,\cong\, {\sb{R}L^2(R)} \otimes \ell^2(\bbN)\otimes L^2(R)_R$.
\end{ex}

\begin{rem}
If we had taken $\Bim(R)$ to be the category of \emph{all} bimodules, with no restriction on cardinality, then it would not admit an absorbing object (and similarly for the previous examples).
\end{rem}

Absorbing objects are useful because \emph{they control half-braidings}:

\begin{prop}
\label{prop:DeltaDeterminesHalfBraidings}
Let $\Omega$ be an absorbing object of $\cC$, and let $(X,e_X)$ be an object of $\cC'$. Then $e_X$ is completely determined by its value on $\Omega$.
\end{prop}
\begin{proof}
%Let $(X,e_X)\in \cC''$ where $X\in \Bim(R)$ and $e_X$ is a half-braiding with $\cC'$.
Let $Y$ be a non-zero object of $\cC$.
Since $e_X$ is a half-braiding, we have a commutative diagram
\[
\xymatrix@C=.5cm{
&
Y \boxtimes X \boxtimes \Omega
\ar[dr]^{\,\,\,\,\id_Y\boxtimes\, e_{X,\Omega}}
&
\\
X\boxtimes Y\boxtimes \Omega\ar[rr]^{e_{X,Y\boxtimes\Omega}}
\ar[ur]^{e_{X,Y}\boxtimes\, \id_\Omega\,\,\,\,\,}
&&
Y \boxtimes \Omega \boxtimes X.
}
\]
Fix an isomorphism $\phi:Y\boxtimes \Omega\to \Omega$.
The following square is commutative
\[
\xymatrix@C=1.2cm{
X\boxtimes (Y\boxtimes \Omega)\ar[rr]^{e_{X,Y\boxtimes\Omega}} \ar[d]^{\id_X \boxtimes\, \phi}
&&
(Y\boxtimes \Omega)\boxtimes X \ar[d]^{\phi\, \boxtimes \id_X}
\\
X\boxtimes \Omega\ar[rr]^{e_{X,\Omega}} 
&&
\Omega\boxtimes X
}
\]
and so we get an equation
\[
e_{X,Y}\boxtimes \id_\Omega=(\id_Y\boxtimes e_{X,\Omega}^{-1})\circ(\phi^{-1} \boxtimes \id_X)\circ e_{X,\Omega}\circ(\id_X \boxtimes \phi).
\]
In particular, we see that $e_{X,Y}\boxtimes \id_\Omega$ is completely determined by $e_{X,\Omega}$.
Since $\Bim(R)$ has no zero-divisors,
$e_{X,Y}$ is completely determined by $e_{X,Y}\boxtimes \id_\Omega$.
Putting those two facts together, we see that $e_{X,Y}$ is completely determined by $e_{X,\Omega}$.
\end{proof}

%%%%%%%%%%%%%%%%%%%%%%%%%%%%%%%%%%%%%%%%%%%%%%%%%
\subsection{The absorbing object of  \texorpdfstring{$\cC'$}{C'}}

We now return to our usual setup, which is that of a separable factor $R$ %that is either of type ${\rm II}_1$ or of type ${\rm III}_1$, 
equipped with a fully faithful representation $\cC\to \Bim(R)$ of some unitary fusion category $\cC$.
Our next goal is to show that $\cC'$ admits absorbing objects.
Recall the construction
\[
\quad
\underline\Delta:\Bim(R)\to \cC'
\qquad
\underline\Delta(\Lambda)=(\Delta(\Lambda),e_{\Delta(\Lambda)})
\]
from Section \ref{sec: functorial construction of objects in C'}.

\begin{thm}
\label{thm:Absorbing}
The functor $\underline\Delta$ sends absorbing objects to absorbing objects.
In particular, the category $\cC'$ admits absorbing objects.
\end{thm}

The proof of this theorem will depend on Theorem \ref{thm:Factor}, proved in next section, according to which the endomorphism algebra of 
$\underline\Delta(\Lambda)$ is a factor whenever $\Lambda$ is absorbing in $\Bim(R)$.
We begin with the following technical lemma:

\begin{lem}
\label{lem:CornerToWhole}
Suppose that $\underline \Omega=(\Omega, e_\Omega)\in \cC'$ is such that $\Omega$ is absorbing in $\Bim(R)$,
and such that $\underline \Omega\oplus \underline \Omega\;\!\cong\;\! \underline \Omega$ in $\cC'$.
Then $\underline \Omega$ is (non-canonically) isomorphic to $\underline{\Delta}(\Omega)$.
\end{lem}

\begin{proof}
Let $\varphi:\Omega \to \Delta(\Omega)$
be the map given by
$$
\varphi\,:=
\sum_{x\in\Irr(\cC)}\sqrt{d_x}\,\,\,
\begin{tikzpicture}[baseline=-.1cm]
	\draw[very thick] (-.15,-.8) -- (-.15,0);
	\draw[very thick] (.2,.8) -- (.2,0);
	\draw (-.15,0) -- (-.15,.8);
	\draw (.2,-.3) -- (.2,-.4) arc (-180:0:.2cm) -- (.6,.8);
	\roundNbox{unshaded}{(.01,0)}{.3}{.15}{.15}{$e_{\Omega,x}$}
	\node at (-.2,-1) {\scriptsize{$\Omega$}};
	\node at (-.2,1) {\scriptsize{$x$}};
	\node at (.2,1) {\scriptsize{$\Omega$}};
	\node at (.6,1) {\scriptsize{$\overline{x}$}};
\end{tikzpicture}
\,=
\sum_{x\in \Irr(\cC)}\sqrt{d_x}\,
\begin{tikzpicture}[baseline=-.1cm]
	\draw (-.3,.8) -- (-.3,.6) .. controls ++(270:.5cm) and ++(270:1.5cm) .. (.6,.4) -- (.6,.8);
	\draw[super thick, white] (0,-.8) -- (0,.8);
	\draw[very thick] (0,-.8) -- (0,.8);
	\node at (0,-1) {\scriptsize{$\Omega$}};
	\node at (-.3,1) {\scriptsize{$x$}};
	\node at (0,1) {\scriptsize{$\Omega$}};
	\node at (.6,1) {\scriptsize{$\overline{x}$}};
\end{tikzpicture}
.
$$
By the Fusion relation, this map is compatible with the half-braidings:
$$
(\id_y\boxtimes \varphi)\circ e_{\Omega,y}
\,=\,
\sum_{x}\sqrt{d_x}\!
\begin{tikzpicture}[baseline=-.2cm]
	\draw (-.8,.8) -- (.6,-1);
	\draw (-.3,.8) -- (-.3,.6) .. controls ++(270:.5cm) and ++(270:1.5cm) .. (.6,.6) -- (.6,.8);
	\draw[super thick, white] (0,-1) -- (0,.8);
	\draw[very thick] (0,-1) -- (0,.8);
	\node at (0,-1.2) {\scriptsize{$\Omega$}};
	\node at (-.3,1) {\scriptsize{$x$}};
	\node at (-.8,1) {\scriptsize{$y$}};
	\node at (0,1) {\scriptsize{$\Omega$}};
	\node at (.6,1) {\scriptsize{$\overline x$}};
	\node at (.6,-1.2) {\scriptsize{$y$}};
\end{tikzpicture}
=
\sum_{x,z}
\sqrt{d_zd^{-1}}\hspace{-.6cm}{\phantom{d}}_y^{\phantom{-1}}
%\sqrt{d_zd_y^{-1}}
\begin{tikzpicture}[baseline=-.2cm]
	\draw (-.6,.8) -- (-.3,.4);
	\draw (.6,0) -- (1.2,-1);
	\draw (-.3,.8) -- (-.3,.4) .. controls ++(270:.5cm) and ++(270:1.5cm) .. (.6,0) -- (.6,.8);
	\draw[super thick, white] (0,-1) -- (0,.8);
	\draw[very thick] (0,-1) -- (0,.8);
	\draw[fill=\betacolor] (-.3,.4) circle (.05cm);
	\draw[fill=\betacolor] (.6,0) circle (.05cm);
	\node at (0,-1.2) {\scriptsize{$\Omega$}};
	\node at (-.3,1) {\scriptsize{$x$}};
	\node at (-.6,1) {\scriptsize{$y$}};
	\node at (-.4,0) {\scriptsize{$z$}};
	\node at (0,1) {\scriptsize{$\Omega$}};
	\node at (.6,1) {\scriptsize{$\overline x$}};
	\node at (1.2,-1.2) {\scriptsize{$y$}};
	\node at (.4,-.25){\scriptsize{$\overline z$}};
\end{tikzpicture}
\!\!=\,
e_{\Delta(\Omega),y}\circ (\varphi\boxtimes\id_y),
$$
and therefore defines a morphism $\varphi:\underline\Omega\to \underline\Delta(\Omega)$ in $\cC'$.

The coevaluation map $\mathrm{coev}_x:L^2R\to x\boxtimes \overline x$ is, up to a constant, the inclusion of a direct summand. So $\varphi$ is manifestly injective.
By polar decomposition in $\cC'$, the map $\varphi$ therefore induces a unitary isomorphism between 
$\underline\Omega$ and a certain subobject of $\underline\Delta(\Omega)$.

Now, the subobjects of $\underline\Delta(\Omega)$ are in one-to-one correspondence with the projections in $M:=\End_{\cC'}(\underline\Delta(\Omega))$, which is a factor by Theorem \ref{thm:Factor}.
Let $p\in M$ be the projection corresponding to $\underline\Omega$.
Since $\underline\Omega\oplus \underline\Omega\cong \underline\Omega$ and $\underline\Omega\neq 0$, that projection is infinite (its range is an infinite module).
So there is a partial isometry $u\in M$ with $p=uu^*$ and $u^*u=1$.
The latter provides an isomorphism $u:\underline\Delta(\Omega)\to \underline\Omega$ in $\cC'$.
\end{proof}

\begin{proof}[Proof of Theorem \ref{thm:Absorbing}]
Let $\Lambda$ be an absorbing object of $\Bim(R)$ and let $X$ be an arbitrary non-zero object of $\cC'$.
We wish to show that $\underline \Omega:=\underline \Delta(\Lambda)\boxtimes X$ is isomorphic to $\underline \Delta(\Lambda)$.
Let $\Omega$ denote the underlying object of $\underline\Omega$.
If we could show that $\underline \Omega$ satisfies the hypotheses of Lemma \ref{lem:CornerToWhole}, then we could reason as follows:
\[
\underline\Delta(\Lambda)\boxtimes X\,=\,\,\underline \Omega
\,\,\cong\,\,\underline\Delta(\Omega)
\,\,\cong\,\, \underline\Delta(\Lambda),
\]
where the last isomorphism holds because $\Omega$ and $\Lambda$ are both absorbing in $\Bim(R)$.

So let us show that $\underline \Omega$ satisfies the hypotheses of Lemma \ref{lem:CornerToWhole}.
Since $\Lambda$ is absorbing in $\Bim(R)$, the object $\Omega=\bigoplus_x x\boxtimes \Lambda\boxtimes \overline x\boxtimes X$ is clearly absorbing in $\Bim(R)$.
And since $\Lambda\oplus \Lambda \cong \Lambda$ in $\Bim(R)$ and $\Lambda\mapsto\underline \Delta(\Lambda)\boxtimes X$ is a linear functor, the same holds true for $\underline\Omega$, namely,
$\underline \Omega\oplus \underline \Omega\;\!\cong\;\! \underline \Omega$.
\end{proof}

%%%%%%%%%%%%%%%%%%%%%%%%%%%%%%%%%%%%%%%%%%%%%%%%%
\subsection{The endomorphism algebra is a factor}
\label{sec: The endomorphism algebra is a factor}

The goal of this section is to prove that when $\Lambda$ is absorbing,
the endomorphism algebra of $\underline\Delta(\Lambda)$ is a factor (a von Neumann algebra with trivial center).
We emphasize the fact that, for the above result to hold, it is essential that the representation $\cC\to\Bim(R)$ be fully faithful
(this is used in the last paragraph of the proof of Theorem~\ref{thm: trivial relative commutant}).

\begin{thm}
\label{thm:Factor}
If $\Lambda$ is absorbing in $\Bim(R)$, then
$\End_{\cC'}(\underline\Delta(\Lambda))$
is a factor.
\end{thm}

It will be easier to prove the following stronger result:

\begin{thm}
\label{thm: trivial relative commutant}
If $\Lambda$ is absorbing, then $\End_{\Bim(R)}(\Lambda)$ has trivial commutant in
$\End_{\cC'}(\underline\Delta(\Lambda))$.
In other words,
the inclusion
\begin{equation}\label{eq: irred subfactor}
\End_{\Bim(R)}(\Lambda)\subset \End_{\cC'}(\underline\Delta(\Lambda))
\end{equation}
is an irreducible subfactor.
\end{thm}

\begin{proof}
The absorbing object is unique up to isomorphism.
So without loss of generality, we may take $\Lambda$ to be the one from Example \ref{ex:Existence}, namely
\(
\Lambda={\sb{R}L^2(R)}\otimes \ell^2(\bbN) \otimes L^2(R)_R.
\)
Let
\[
\Lambda_0:={\sb{R}L^2R} \otimes L^2R_R.
\]
Writing $H$ for $\ell^2(\bbN)$, we have
\begin{align*}
\End_{\Bim(R)}(\Lambda)\,&\cong\, \End_{\Bim(R)}(\Lambda_0)\,\bar\otimes\,B(H)\quad\text{and}
&
\End_{\cC'}(\underline\Delta(\Lambda))\,&\cong\,\End_{\cC'}(\underline\Delta(\Lambda_0))\,\bar\otimes\,B(H),
\end{align*}
and so
$Z_{\End(\underline\Delta(\Lambda))}(\End(\Lambda))\cong Z_{\End(\underline\Delta(\Lambda_0))}(\End(\Lambda_0))$.
It is therefore equivalent to prove the statement of the theorem for $\Lambda_0$ instead of $\Lambda$.
Recall from Theorem \ref{thm: endo} that
\[
\End_{\cC'}(\underline\Delta(\Lambda_0)) \,\,\cong\,\, 
\bigoplus_{x\in \Irr(\cC)} \Hom_{\Bim(R)}(\Lambda_0\boxtimes x, x\boxtimes\Lambda_0),\vspace{-.2cm}
\]
with product as in (\ref{eq: (fg)_a}).

Let $f=(f_x: \Lambda_0 \boxtimes x \to x\boxtimes \Lambda_0)_{x\in\Irr(\cC)}$ be an element that commutes with every $g\in \End_{\Bim(R)}(\Lambda_0)=R^{\text{op}}\,\bar\otimes\,R$:
\begin{equation}\label{eq: fg = gf}
\qquad\qquad
\begin{tikzpicture}[baseline=0cm]
	\draw[very thick] (0,-1) -- (0,1);
	\draw (-.6,1) -- (.3,.1) .. controls ++(-45:.3cm) and ++(90:.3cm) ..  (.6,-1);
	\roundNbox{unshaded}{(0,.4)}{.3}{0}{0}{$f_x$};
	\roundNbox{unshaded}{(0,-.4)}{.3}{0}{0}{$g$};
	\node at (-.6,1.2) {\scriptsize{$x$}};
	\node at (.6,-1.2) {\scriptsize{$x$}};
\node at (.05,1.2) {\scriptsize{$\Lambda_0$}};
\node at (.05,-1.2) {\scriptsize{$\Lambda_0$}};
\end{tikzpicture}
=
\begin{tikzpicture}[baseline=0cm, xscale=-1, yscale=-1]
	\draw[very thick] (0,-1) -- (0,1);
	\draw (-.6,1) -- (.3,.1) .. controls ++(-45:.3cm) and ++(90:.3cm) ..  (.6,-1);
	\roundNbox{unshaded}{(0,.4)}{.3}{0}{0}{$f_x$};
	\roundNbox{unshaded}{(0,-.4)}{.3}{0}{0}{$g$};
	\node at (-.6,1.2) {\scriptsize{$x$}};
	\node at (.6,-1.2) {\scriptsize{$x$}};
\node at (-.05,1.2) {\scriptsize{$\Lambda_0$}};
\node at (-.05,-1.2) {\scriptsize{$\Lambda_0$}};
\end{tikzpicture}\qquad
\begin{tikzpicture}[baseline=-.1cm]
\node[scale=.9]{$\forall\,x\in \Irr(\cC),\,\,\,\forall\,g\in \End(\Lambda_0).$};
\end{tikzpicture}
\end{equation}
The bimodule $\Lambda_0$ is of the form \eqref{eq: coarse bimodule}, and thus coarse. The action of the algebraic tensor product $R\odot R^{\text{op}}$ (the one which equips it with the structure of an $R$-$R$-bimodule) therefore extends to an action of the spatial tensor product $R\,\bar\otimes\, R^{\text{op}}$.
We may therefore treat $\Lambda_0$ as a left $(R\,\bar\otimes\, R^{\text{op}})$-module.
Writing $1$ for $L^2(R)$, we then have canonical isomorphisms
\begin{align*}
{}_R(x\boxtimes_R\Lambda_0)_R
&\,\cong\,
{}_{R\bar\otimes R^\mathrm{op}}((x\otimes 1)\boxtimes_{R\bar\otimes R^\mathrm{op}}\Lambda_0)
\\
{}_R(\Lambda_0\boxtimes_R x)_R
&\,\cong\,
{}_{R\bar\otimes R^\mathrm{op}}
((1\otimes x)\boxtimes_{R\bar\otimes R^\mathrm{op}}\Lambda_0).
\end{align*}
Under those identifications equation (\ref{eq: fg = gf}) becomes:
\begin{equation}\label{eq: light grey on the right}
\qquad\quad
\begin{tikzpicture}[baseline=-.1cm]
\fill[gray!80] (.2,-1) -- (.2,1) to[rounded corners=10] (-1.05,1) to[rounded corners=10] (-1.05,-1) -- cycle;
\fill[gray!20] (.2,-1) -- (.2,1) to[rounded corners=10] (1,1) to[rounded corners=10] (1,-1) -- cycle;
\draw[very thick] (.2,-1) -- (.2,1);
	\draw (-.6,1) -- (0,.4) -- (-.3,.1) .. controls ++(225:.3cm) and ++(90:.3cm) ..  (-.6,-1);
	\roundNbox{unshaded}{(0,.4)}{.3}{0}{.4}{$f_x$};
	\roundNbox{unshaded}{(0,-.4)}{.3}{0}{.4}{$g$};
\node at (-.65,1.2) {\scriptsize{$x\otimes 1$}};
\node at (-.65,-1.2) {\scriptsize{$1\otimes x$}};
	\node at (.25,1.2) {\scriptsize{$\Lambda_0$}};
	\node at (.25,-1.2) {\scriptsize{$\Lambda_0$}};
\end{tikzpicture}
\,\,\,\,\,=\,\,\,
\begin{tikzpicture}[baseline=-.1cm, yscale=-1]
\fill[gray!80] (.2,-1) -- (.2,1) to[rounded corners=10] (-1.05,1) to[rounded corners=10] (-1.05,-1) -- cycle;
\fill[gray!20] (.2,-1) -- (.2,1) to[rounded corners=10] (1,1) to[rounded corners=10] (1,-1) -- cycle;
	\draw[very thick] (.2,-1) -- (.2,1);
	\draw (-.6,1) -- (0,.4) -- (-.3,.1) .. controls ++(225:.3cm) and ++(90:.3cm) ..  (-.6,-1);
	\roundNbox{unshaded}{(0,.4)}{.3}{0}{.4}{$f_x$};
	\roundNbox{unshaded}{(0,-.4)}{.3}{0}{.4}{$g$};
\node at (-.65,-1.2) {\scriptsize{$x\otimes 1$}};
\node at (-.65,1.2) {\scriptsize{$1\otimes x$}};
	\node at (.25,1.2) {\scriptsize{$\Lambda_0$}};
	\node at (.25,-1.2) {\scriptsize{$\Lambda_0$}};
\end{tikzpicture}\qquad
\begin{tikzpicture}[baseline=-.1cm]
\node[scale=.9]{$\forall\,x\in \Irr(\cC),\,\,\,\forall\,g\in \End(\Lambda_0)$,};
\end{tikzpicture}
\end{equation}
where
$\tikz[baseline=.1cm]{\draw[fill=gray!80, rounded corners=5, very thin, baseline=1cm] (0,0) rectangle (.5,.5);}=R\,\bar\otimes\, R^{\text{op}}$,
$\tikz[baseline=.1cm]{\draw[fill=gray!20, rounded corners=5, very thin] (0,0) rectangle (.5,.5);}=\mathbb C$,
and we have used the string diagram notation for bicategories reviewed in \cite[\S2]{MR3342166}%Dualizability and index of subfactors
.

Note that $\Lambda_0 = L^2(R\bar\otimes R^\mathrm{op})$.
We may therefore identify
$(x\otimes 1)\boxtimes_{R\bar\otimes R^\mathrm{op}}\Lambda_0$
with $x\otimes 1$,
and
$(1\otimes x)\boxtimes_{R\bar\otimes R^\mathrm{op}}\Lambda_0$
with $1\otimes x$.
The maps $f_x$ can then be viewed as left $(R\,\bar\otimes\, R^{\text{op}})$-module maps:
\[
f_x: 1\otimes x \to x\otimes 1.
\]
The operators $\id_{1\otimes x}\boxtimes g$ and $\id_{x\otimes 1}\boxtimes g$
which appear on the two sides of (\ref{eq: light grey on the right})
are nothing else than the right actions of $g\in R\,\bar\otimes\, R^{\text{op}}$ on $1\otimes x$ and on $x\otimes 1$, and so
equation (\ref{eq: light grey on the right}) is just the statement that $f_x$ is a right $(R\,\bar\otimes\, R^{\text{op}})$-module map.
Each $f_x$ is therefore both a left $(R\,\bar\otimes\, R^{\text{op}})$-module 
and a right $(R\,\bar\otimes\, R^{\text{op}})$-module map. % the latter is thus a bimodule map.

But $1\otimes x$ and $x\otimes 1$ are irreducible $(R\,\bar\otimes\, R^{\text{op}})$-$(R\,\bar\otimes\, R^{\text{op}})$-bimodules, and $1\otimes x\not\cong x\otimes 1$ unless $x=1$.
The maps $f_x$ can therefore only be nonzero when $x=1$, in which case it must be a scalar.
\end{proof}

%Given a right module $H$ over a von Neumann algebra $M$, we write $\rho_H:H\boxtimes_M L^2M\to H$ for the unit isomorphism of Connes fusion.
%We also write $r$ for the right action of $M$ on $L^2M$.
%
%\begin{lem} \label{lem:PeelOffStrand prelim version}
%Let $A$ and $B$ be right modules over some von Neumann algebra $M$, and
%let $f:A\boxtimes_M L^2M\to B\boxtimes_M L^2M$ be a map which satisfies
%\[
%\qquad f\circ (\id_A\boxtimes r(m)) = (\id_B\boxtimes r(m))\circ f
%\]
%for all $m\in M$.
%Then $f':=\rho_B \circ f\circ \rho_A^{-1}:A\to B$ is a right module map.
%\end{lem}
%
%\begin{proof}
%Write $r_A$ and $r_B$ for the right actions of $M$ on and $A$ and $B$, %respectively.
%The unit isomorphisms $\rho_A$ and $\rho_B$ are right $M$-module maps:
%\[
%\rho_A\circ (\id_A\boxtimes r(m)) = r_A(m) \circ \rho_A
%\qquad
%\rho_B\circ (\id_B\boxtimes r(m)) = r_B(m) \circ \rho_B,
%\]
%and so
%\begin{align*}
%f'\circ r_A(m)
%&\,=\, \rho_B \circ f\circ \rho_A^{-1}\circ r_A(m) \\
%&\,=\, \rho_B \circ f\circ (\id_A\boxtimes r(m))\circ \rho_A^{-1}\\
%&\,=\, \rho_B \circ (\id_B\boxtimes r(m))\circ f\circ \rho_A^{-1}\\
%&\,=\, r_B(m) \circ \rho_B \circ f\circ \rho_A^{-1} \,=\, r_B(m) \circ f'.
%\qedhere
%\end{align*}
%\end{proof}

Let us now assume that $\Lambda$ is a coarse bimodule, and that it is given to us as the tensor product of a left $R$-modules with a right $R$-module:
\[
\Lambda={}_RH\otimes_{\mathbb C} K_R\,.
\]
Then we have $\End_{\Bim(R)}(\Lambda)=\End({}_RH)\,\bar\otimes\,\End(K_R)$, and
the subfactor (\ref{eq: irred subfactor}) is of the form
\[
\End({}_RH)\,\bar\otimes\,\End(K_R)\,\subset\,
\End_{\cC'}(\underline\Delta(\Lambda)).
\]

\begin{prop}
\label{prop: RelComm inside End_C'(Delta)}
The algebras $\End({}_RH)$ and $\End(K_R)$ are each other's relative commutants in
$\End_{\cC'}(\underline\Delta(\Lambda))$.
\end{prop}

\begin{proof}
We will only prove that $Z_{\End_{\cC'}(\underline\Delta(\Lambda))}(\End({}_RH))=\End(K_R)$.
The other claim is symmetric and can be proved in a completely analogous way.

Let $b\in\End({}_RH)$ be an endomorphism of $H$, and let $f$ be an element of $\End_{\cC'}(\underline\Delta(\Lambda))$. Let $f_a:\Lambda\boxtimes a\to a \boxtimes \Lambda$ be the maps which correspond to $f\in \End_{\cC'}(\underline\Delta(\Lambda))$ under the bijection established in Theorem~\ref{thm: endo}.
The statement that $b$ and $f$ commute is then equivalent to the statement that for every $a\in \Irr(\cC)$, the following equality holds in
$\Hom(H\otimes_{\mathbb C}K\boxtimes_R a,a\boxtimes_R H\otimes_{\mathbb C}K)$:
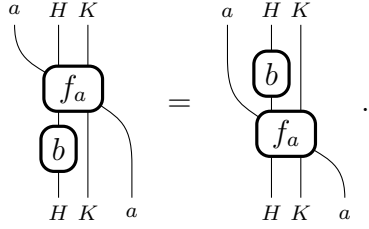
$$
\def\vshorten{.45}
\begin{tikzpicture}[baseline=0cm, xscale=-1.3]
    \draw (-.15,-1.6+\vshorten) -- (-.15,1.6-\vshorten);
	\draw (.15,-1.6+\vshorten) -- (.15,1.6-\vshorten);
    \draw (-.6,-1.6+\vshorten) -- (-.6,-.7) .. controls ++(90:.6cm) and ++(225:.4cm) .. (-.15,.2);
    \draw (.15,.4) .. controls ++(45:.3cm) and ++(270:.5cm) .. (.6,1.6-\vshorten);
\roundNbox{unshaded}{(0,.3)}{.3}{0}{0}{$f_a$}
\roundNbox{unshaded}{(.15,-.5)}{.3}{-.12}{-.12}{$b$}
    \node at (-.15,1.8-\vshorten) {\scriptsize{$K$}};
	\node at (.15,1.8-\vshorten) {\scriptsize{$H$}};
	\node at (-.15,-1.8+\vshorten) {\scriptsize{$K$}};
	\node at (.15,-1.8+\vshorten) {\scriptsize{$H$}};
	\node at (-.6,-1.8+\vshorten) {\scriptsize{$a$}};
	\node at (.6,1.8-\vshorten) {\scriptsize{$a$}};
\end{tikzpicture}
\,\,=\,\,
\begin{tikzpicture}[baseline=0cm, yscale=-1, xscale=1.3]
    \draw (-.15,-1.6+\vshorten) -- (-.15,1.6-\vshorten);
	\draw (.15,-1.6+\vshorten) -- (.15,1.6-\vshorten);
    \draw (-.6,-1.6+\vshorten) -- (-.6,-.7) .. controls ++(90:.6cm) and ++(225:.4cm) .. (-.15,.2);
    \draw (.15,.4) .. controls ++(45:.3cm) and ++(270:.5cm) .. (.6,1.6-\vshorten);
\roundNbox{unshaded}{(0,.3)}{.3}{0}{0}{$f_a$}
\roundNbox{unshaded}{(-.15,-.5)}{.3}{-.12}{-.12}{$b$}
    \node at (-.15,1.8-\vshorten) {\scriptsize{$H$}};
	\node at (.15,1.8-\vshorten) {\scriptsize{$K$}};
	\node at (-.15,-1.8+\vshorten) {\scriptsize{$H$}};
	\node at (.15,-1.8+\vshorten) {\scriptsize{$K$}};
	\node at (-.6,-1.77+\vshorten) {\scriptsize{$a$}};
	\node at (.6,1.82-\vshorten) {\scriptsize{$a$}};
\end{tikzpicture}
.
$$
Treating $K$ as a left $R^{\op}$-module
and letting $R'$ be the commutant of $R$ on $H$ (so that $H$ is an $R$-$R'{}^{\op}$-bimodule),
we may `fold' the above diagram (as we did to get
%go from (\ref{eq: fg = gf}) to 
(\ref{eq: light grey on the right})):
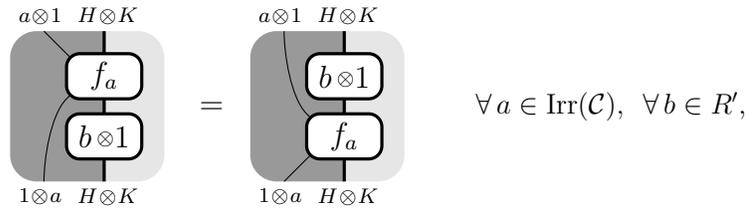
\[
\qquad
\begin{tikzpicture}[baseline=-.1cm]
\fill[gray!80] (.2,-1) -- (.2,1) to[rounded corners=10] (-1.05,1) to[rounded corners=10] (-1.05,-1) -- cycle;
\fill[gray!20] (.2,-1) -- (.2,1) to[rounded corners=10] (1,1) to[rounded corners=10] (1,-1) -- cycle;
\draw[very thick] (.2,-1) -- (.2,1);
	\draw (-.6,1) -- (0,.4) -- (-.3,.1) .. controls ++(225:.3cm) and ++(90:.3cm) ..  (-.6,-1);
	\roundNbox{unshaded}{(0,.4)}{.3}{0}{.4}{$f_a$};
	\roundNbox{unshaded}{(0,-.4)}{.3}{0}{.4}{$b\,\raisebox{.2ex}{$\scriptstyle\otimes$}1$};
\node at (-.65,1.2) {\scriptsize{$a\!\otimes\! 1$}};
\node at (-.65,-1.2) {\scriptsize{$1\!\otimes\! a$}};
	\node at (.25,1.2) {\scriptsize{$H\!\otimes\! K$}};
	\node at (.25,-1.2) {\scriptsize{$H\!\otimes\! K$}};
\end{tikzpicture}
\,\,\,\,\,=\,\,\,
\begin{tikzpicture}[baseline=-.1cm, yscale=-1]
\fill[gray!80] (.2,-1) -- (.2,1) to[rounded corners=10] (-1.05,1) to[rounded corners=10] (-1.05,-1) -- cycle;
\fill[gray!20] (.2,-1) -- (.2,1) to[rounded corners=10] (1,1) to[rounded corners=10] (1,-1) -- cycle;
	\draw[very thick] (.2,-1) -- (.2,1);
	\draw (-.6,1) -- (0,.4) -- (-.3,.1) .. controls ++(225:.3cm) and ++(90:.3cm) ..  (-.6,-1);
	\roundNbox{unshaded}{(0,.4)}{.3}{0}{.4}{$f_a$};
	\roundNbox{unshaded}{(0,-.4)}{.3}{0}{.4}{$b\,\raisebox{.2ex}{$\scriptstyle\otimes$} 1$};
\node at (-.65,-1.2) {\scriptsize{$a\!\otimes\! 1$}};
\node at (-.65,1.2) {\scriptsize{$1\!\otimes\! a$}};
	\node at (.25,1.2) {\scriptsize{$H\!\otimes\! K$}};
	\node at (.25,-1.2) {\scriptsize{$H\!\otimes\! K$}};
\end{tikzpicture}\qquad
\begin{tikzpicture}[baseline=-.1cm]
\node[scale=.9]{$\forall\,a\in \Irr(\cC),\,\,\,\forall\,b\in R'$,};
\end{tikzpicture}
\]
where
$\tikz[baseline=.1cm]{\draw[fill=gray!80, rounded corners=5, very thin, baseline=1cm] (0,0) rectangle (.5,.5);}=R\,\bar\otimes\, R^{\text{op}}$
and 
$\tikz[baseline=.1cm]{\draw[fill=gray!20, rounded corners=5, very thin] (0,0) rectangle (.5,.5);}=\mathbb C$.
It follows %from Lemma \ref{lem:PeelOffStrand}
that $f_a$ is not just in
\begin{gather*}
\Hom_{R\,\bar\otimes\, R^{\text{op}}}\big(
(1\otimes a)\boxtimes_{R\,\bar\otimes\, R^{\text{op}}}(H\otimes K),
(a\otimes 1)\boxtimes_{R\,\bar\otimes\, R^{\text{op}}}(H\otimes K)
\big)
\\
=\,\Hom_{R}(L^2R\boxtimes_R H,a\boxtimes_RH)\,\bar\otimes\,
\Hom_{R^{\op}}(a\boxtimes_{R^{\op}} K,L^2R\boxtimes_{R^{\op}}K),
\end{gather*}
but actually in
\[
\Hom_{R,R'{}^{\op}}\big(L^2R\boxtimes_R H,a\boxtimes_RH\big)
\,\bar\otimes\,
\Hom_{R^{\op}}\big(a\boxtimes_{R^{\op}} K,L^2R\boxtimes_{R^{\op}}K\big).
\]
But $H$ is an invertible $R$-$R'{}^{\op}$-bimodule, and so
\[
\Hom_{R,R'{}^{\op}}(L^2R\boxtimes_R H,a\boxtimes_RH)
=\Hom_{\Bim(R)}(1,a).
\]
It follows that $f_a=0$ unless $a=1$, in which case
$f\in \Hom_{R^{\op}}\big(K,K\big)=\End(K_R)$.
\end{proof}

\begin{rem}
Proposition \ref{prop: RelComm inside End_C'(Delta)}
implies Theorems \ref{thm:Factor} and \ref{thm: trivial relative commutant}.
It shows that, among other things, these two theorems hold in the greater generality of $\Lambda$ a coarse bimodule (as opposed to merely absorbing).
\end{rem}

%} % end comment
%%%%%%%%%%%%%%%%%%%%%%%%%%%%%%%%%%%%%%%%%%%%%%%%%
%%%%%%%%%%%%%%%%%%%%%%%%%%%%%%%%%%%%%%%%%%%%%%%%%
%%%%%%%%%%%%%%%%%%%%%%%%%%%%%%%%%%%%%%%%%%%%%%%%%
\subsection{Algebras acting on cyclic fusions}

Let $\Lambda_1$ and $\Lambda_2$ be coarse bimodules.
In Section \ref{sec:The endomorphism algebra}, we computed the endomorphism algebra of $\underline\Delta(\Lambda_1)=(\Delta(\Lambda_1),e_{\Lambda_1})\in\cC'$.
Our next task is to compute the commutant of $\End_{\cC'}(\underline\Delta(\Lambda_1))$ on the cyclic fusion
\[
\qquad
\Big[ \Delta(\Lambda_1) \boxtimes \Lambda_2 \boxtimes - \Big]_{\text{cyclic}}
=
\bigoplus_{x\in \Irr(\cC)} 
\Big[ x \boxtimes \Lambda_1 \boxtimes \overline{x} \boxtimes \Lambda_2 \boxtimes - \Big]_{\text{cyclic}}
\]
We first note that there is a commuting action of 
$\End_{\cC'}(\underline\Delta(\Lambda_2))$
on that same Hilbert space:
$$
\def\addheight{.8}
\sum_{y\in\Irr(\cC)}\,\,\,\,
\begin{tikzpicture}[scale=1.1, baseline=1.5cm]
    \coordinate (a) at (1.4,.2);   % right red dot
    \coordinate (b) at (.2,1.25);   % left red dot
    \coordinate (c) at (.2,1.6);   % left blue dot
    \coordinate (d) at (1.4,2.62);   % right blue dot
    \coordinate (e) at (.6,.65);     % f box
    \coordinate (f) at (1,2.1);     % g box
	\draw[thick] (.8,2.2+\addheight) ellipse (.8 and .2);
	\draw[thick] (0,0) -- (0,2.2+\addheight);
	\draw[thick] (1.6,0) -- (1.6,2.2+\addheight);
	\halfDottedEllipse{(0,0)}{.8}{.2}
	\draw[thick, dotted] (.2,.13) -- (.2,2.12+\addheight);
	\draw[thick] (.2,2.07+\addheight) -- (.2,2.33+\addheight);
	\draw[very thick] (.6,-.2) -- (.6,2+\addheight);
	\draw[very thick, dotted] (1,.2) -- (1,2+\addheight);
	\draw[very thick] (1,2+\addheight) -- (1,2.4+\addheight);
	\draw[thick] (1.4, -.13) -- (1.4,2.07+\addheight);
	\draw (a) .. controls ++(160:.2cm) and ++(-45:.2cm) .. ($ (e) + (.3,-.3) $) -- ($ (e) + (-.3,.3) $) .. controls ++(135:.1cm) and ++(-20:.1cm) .. ($ (b) + (-.2,-.15) $);
	\draw[thick, dotted] ($ (b) + (-.2,-.15) $) .. controls ++(90:.1cm) and ++(225:.1cm) .. (b);
	\draw[thick, dotted] (c) .. controls ++(20:.2cm) and ++(-135:.2cm) .. ($ (f) + (-.3,-.3) $) -- ($ (f) + (.3,.25) $) .. controls ++(30:.1cm) and ++(-140:.1cm) .. ($ (d) + (.2,-.1) $);
	\draw ($ (d) + (.2,-.1) $) .. controls ++(90:.1cm) and ++(0:.1cm) .. (d);
	\draw[fill=\betacolor] (a) circle (.05cm);
	\draw[fill=\betacolor] (b) circle (.05cm);
	\draw[fill=\alphacolor] (c) circle (.05cm);
	\draw[fill=\alphacolor] (d) circle (.05cm);
	\roundNbox{unshaded, dotted}{(f)}{.3}{0}{0}{$g$}
	\roundNbox{unshaded}{(e)}{.3}{0}{0}{$f$}
	\node at (.2,0) {\scriptsize{$x$}};
	\node at (1.4,-.3) {\scriptsize{$\overline{x}$}};
	\node at (.2,2.5+\addheight) {\scriptsize{$z$}};
	\node at (1.4,2.2+\addheight) {\scriptsize{$\overline{z}$}};
	\node at (.63,-.4) {\scriptsize{$\Lambda_1$}};
	\node at (1.05,2.6+\addheight) {\scriptsize{$\Lambda_2$}};
\node at (.1,1.4) {\scriptsize{$y$}};
\node at (1.5,1.3) {\scriptsize{$\overline{y}$}};
\end{tikzpicture}
\,\,\,\,\,=\,\,
\sum_{y\in\Irr(\cC)}\,\,\,\,
\begin{tikzpicture}[scale=1.1, baseline=1.5cm]
\node at (.1,1.3) {\scriptsize{$y$}};
\node at (1.5,1.45) {\scriptsize{$\overline{y}$}};
    \coordinate (a) at (1.4,.1+1.55);   % right red dot
    \coordinate (b) at (.2,1.25+1.45);   % left red dot
    \coordinate (c) at (.2,1.6-1.3);   % left blue dot
    \coordinate (d) at (1.4,2.7-1.4);   % right blue dot
    \coordinate (e) at (.6,.6+1.55);     % f box
    \coordinate (f) at (1,2.15-1.35);     % g box
	\draw[thick] (.8,2.2+\addheight) ellipse (.8 and .2);
	\draw[thick] (0,0) -- (0,2.2+\addheight);
	\draw[thick] (1.6,0) -- (1.6,2.2+\addheight);
	\halfDottedEllipse{(0,0)}{.8}{.2}
	\draw[thick, dotted] (.2,.13) -- (.2,2.12+\addheight);
	\draw[thick] (.2,2.07+\addheight) -- (.2,2.33+\addheight);
	\draw[very thick] (.6,-.2) -- (.6,2+\addheight);
	\draw[very thick, dotted] (1,.2) -- (1,2+\addheight);
	\draw[very thick] (1,2+\addheight) -- (1,2.4+\addheight);
	\draw[thick] (1.4, -.13) -- (1.4,2.07+\addheight);
	\draw (a) .. controls ++(160:.2cm) and ++(-45:.2cm) .. ($ (e) + (.3,-.3) $) -- ($ (e) + (-.3,.3) $) .. controls ++(135:.1cm) and ++(-20:.1cm) .. ($ (b) + (-.2,-.15) $);
	\draw[thick, dotted] ($ (b) + (-.2,-.15) $) .. controls ++(90:.1cm) and ++(225:.1cm) .. (b);
	\draw[thick, dotted] (c) .. controls ++(20:.2cm) and ++(-135:.2cm) .. ($ (f) + (-.3,-.3) $) -- ($ (f) + (.3,.25) $) .. controls ++(30:.1cm) and ++(-140:.1cm) .. ($ (d) + (.2,-.1) $);
	\draw ($ (d) + (.2,-.1) $) .. controls ++(90:.1cm) and ++(0:.1cm) .. (d);
	\draw[fill=\betacolor] (a) circle (.05cm);
	\draw[fill=\betacolor] (b) circle (.05cm);
	\draw[fill=\alphacolor] (c) circle (.05cm);
	\draw[fill=\alphacolor] (d) circle (.05cm);
	\roundNbox{unshaded, dotted}{(f)}{.3}{0}{0}{$g$}
	\roundNbox{unshaded}{(e)}{.3}{0}{0}{$f$}
	\node at (.2,0) {\scriptsize{$x$}};
	\node at (1.4,-.3) {\scriptsize{$\overline{x}$}};
	\node at (.2,2.5+\addheight) {\scriptsize{$z$}};
	\node at (1.4,2.2+\addheight) {\scriptsize{$\overline{z}$}};
	\node at (.63,-.4) {\scriptsize{$\Lambda_1$}};
	\node at (1.05,2.6+\addheight) {\scriptsize{$\Lambda_2$}};
\end{tikzpicture}
\,\,
$$
Here, we have used Theorem \ref{thm: endo} in order to write a generic element of $\End_{\cC'}(\underline\Delta(\Lambda_1))$
as a sum of operators of the form
$
\begin{tikzpicture}[baseline=-.1cm, scale=.6]
\useasboundingbox (-.8,-.9) rectangle (.8,.9);
	\draw[very thick] (0,-.6) -- (0,.6);
	\draw (-.6,.4) -- (.6,-.4);
	\draw (.6,-.6) -- (.6,.6);
	\draw (-.6,-.6) -- (-.6,.6);
	\draw[fill=\betacolor] (-.6,.4) circle (.07cm);
	\draw[fill=\betacolor] (.6,-.4) circle (.07cm);
	\draw[fill=black] (0,0) node[left, xshift=1, yshift=-4]{$\scriptstyle f$} +(.09,.09) rectangle +(-.09,-.09);
\node at (-.6,.8) {{$\scriptscriptstyle y$}};
\node at (.6,.8) {{$\scriptscriptstyle \overline{y}$}};
\node at (-.6,-.83) {{$\scriptscriptstyle x$}};
\node at (.6,-.8) {{$\scriptscriptstyle \overline{x}$}};
\node at (-.27,.42) {{$\scriptscriptstyle a$}};
\node at (.35,0) {{$\scriptscriptstyle a$}};
\node at (.05,.8) {{$\scriptscriptstyle \Lambda_1$}};
\node at (.05,-.8) {{$\scriptscriptstyle \Lambda_1$}};
\end{tikzpicture}$,
and similarly for $\End_{\cC'}(\underline\Delta(\Lambda_2))$.
We have then used the I=H relation to show that the resulting operators commute.
We have also secretly used the existence of a canonical isomorphism
\begin{equation}\label{eq: iso x bar x}
\bigoplus_{x\in\Irr(\cC)} \overline{x}\boxtimes \Lambda_2 \boxtimes x
%\Delta(\Lambda_2)=
\,\,\cong
\bigoplus_{x\in\Irr(\cC)} x\boxtimes \Lambda_2 \boxtimes \overline{x}.
\end{equation}
(At first sight, this looks like is might depend on the choice of isomorphisms between each $\overline x$ and the corresponding object of $\Irr(\cC)$. But as each $\overline x$ appears next to an $x$, the isomorphism (\ref{eq: iso x bar x}) is independent of those choices.)

\begin{lem}\label{lem: commutant on cyclic fusion}
Let $\Lambda_1$ and $\Lambda_2$ be coarse bimodules.
Then $N_1=\End_{\Bim(R)}(\Lambda_1)$
and $N_2=\End_{\Bim(R)}(\Lambda_2)$
are each other's commutants on
$\big[ \Lambda_1 \boxtimes_R \Lambda_2 \boxtimes_R - \big]_{\text{\rm cyclic}}$.
\end{lem}

\begin{proof}
The algebra $N_1$ is the commutant of $R\,\bar\otimes\,R^{\text{op}}$ on $\Lambda_1$.
By Lemma \ref{lem: invertible bimodules}, the latter is therefore invertible as an $N_1$-$(R^{\text{op}}\,\bar\otimes\,R)$-bimodule.
Similarly, $\Lambda_2$ is invertible as an $(R^{\text{op}}\,\bar\otimes\,R)$-$N_2$-bimodule.
It follows that 
$$
\big[ \Lambda_1 \boxtimes_R \Lambda_2 \boxtimes_R - \big]_{\text{cyclic}}=\,\Lambda_1 \boxtimes_{R^{\text{op}}\,\bar\otimes\,R} \Lambda_2
$$
is an invertible $N_1$-$N_2^{\text{op}}$-bimodule.
\end{proof}

\begin{prop}\label{prop: comm on cyc}
Let $\Lambda_1$ and $\Lambda_2$ be coarse bimodules.
Then
$M_1=\End_{\cC'}(\underline\Delta(\Lambda_1))$
and
$M_2=\End_{\cC'}(\underline\Delta(\Lambda_2))$
are each other's commutants on %the Hilbert space
$H=\bigoplus_{x\in\Irr(\cC)}
\big[ x \boxtimes \Lambda_1 \boxtimes \overline{x} \boxtimes \Lambda_2 \boxtimes - \big]_{\text{\rm cyclic}}
$.
\end{prop}
\begin{proof}
Let $f$ be in $M_2'$.
Since $f$ commutes with $\End_{\Bim(R)}(\Lambda_2)\subset M_2$,
it follows from Lemma~\ref{lem: commutant on cyclic fusion} that
$f\in \End_{\Bim(R)}(\Delta(\Lambda_1))$.
We therefore have the following situation:

\begin{equation}\label{eq: [g,f]=0 on cylinder}
\qquad\qquad
\sum_{y\in\Irr(\cC)}\,\,\,\,
\begin{tikzpicture}[scale=1.1, baseline=1cm]
    \coordinate (c) at (.2,1.05);   % left dot
    \coordinate (d) at (1.4,1.9);   % right dot
    \coordinate (f) at (1,1.45);    % g box
	\draw[thick] (.8,2.2) ellipse (.8 and .2);
	\draw[thick] (0,0) -- (0,2.2);
	\draw[thick] (1.6,0) -- (1.6,2.2);
	\halfDottedEllipse{(0,0)}{.8}{.2}
	\draw[thick, densely dotted] (.2,.13) -- (.2,2.12);
	\draw[thick] (.2,2.07) -- (.2,2.33);
	\draw[very thick] (.6,-.2) -- (.6,2);
	\draw[very thick, densely dotted] (1,.2) -- (1,2);
	\draw[very thick] (1,2) -- (1,2.4);
	\draw[thick] (1.4, -.13) -- (1.4,2.07);
	\draw[thick, densely dotted] (c) .. controls ++(10:.2cm) and ++(-145:.2cm) .. ($ (f) + (-.3,-.25) $) -- ($ (f) + (.3,.25) $) .. controls ++(30:.1cm) and ++(-140:.1cm) .. ($ (d) + (.2,-.1) $);
	\draw ($ (d) + (.2,-.1) $) .. controls ++(90:.1cm) and ++(0:.1cm) .. (d);
	\draw[fill=\betacolor] (c) circle (.05cm);
	\draw[fill=\betacolor] (d) circle (.05cm);
	\roundNbox{unshaded, densely dotted}{(f)}{.3}{0}{0}{$g$}
    \draw[very thick, densely dotted, rounded corners=5pt, unshaded] (0,.7).. controls ++(40:.2cm) and ++(220:.1cm) .. (.3,.9) -- (.3,.6);
    \draw[very thick, rounded corners=3pt, unshaded] (0,.3) .. controls ++(-40:.4cm) and ++(220:.3cm) .. (1.51,.3) --  (1.5,.7) .. controls ++(220:.3cm) and ++(-40:.4cm) .. (0,.7);
	\node at (.6,.32) {\scriptsize{$f$}};
	\node at (.2,0) {\scriptsize{$x$}};
	\node at (1.4,-.3) {\scriptsize{$\overline{x}$}};
	\node at (.2,2.5) {\scriptsize{$z$}};
	\node at (1.4,2.2) {\scriptsize{$\overline{z}$}};
	\node at (.63,-.4) {\scriptsize{$\Lambda_1$}};
	\node at (1.05,2.6) {\scriptsize{$\Lambda_2$}};
\node at (1.5,1.1) {\scriptsize{$\overline{y}$}};
\node at (.1,.89) {\scriptsize{$y$}};%------------
\end{tikzpicture}
\,\,\,\,\,=\,\,
\sum_{y\in\Irr(\cC)}\,\,\,\,
\begin{tikzpicture}[scale=1.1, baseline=1cm]
\node at (1.5,1.3) {\scriptsize{$\overline{y}$}};%------------
\node at (.1,1) {\scriptsize{$y$}};
    \coordinate (c) at (.2,1.05-.77); % left dot
    \coordinate (d) at (1.4,1.9-.77); % right dot
    \coordinate (f) at (1,1.45-.77);  % g box
	\draw[thick] (.8,2.2) ellipse (.8 and .2);
	\draw[thick] (0,0) -- (0,2.2);
	\draw[thick] (1.6,0) -- (1.6,2.2);
	\halfDottedEllipse{(0,0)}{.8}{.2}
	\draw[thick, densely dotted] (.2,.13) -- (.2,2.12);
	\draw[thick] (.2,2.07) -- (.2,2.33);
	\draw[very thick] (.6,-.2) -- (.6,2);
	\draw[very thick, densely dotted] (1,.2) -- (1,2);
	\draw[very thick] (1,2) -- (1,2.4);
	\draw[thick] (1.4, -.13) -- (1.4,2.07);
	\draw[thick, densely dotted] (c) .. controls ++(10:.2cm) and ++(-145:.2cm) .. ($ (f) + (-.3,-.25) $) -- ($ (f) + (.3,.25) $) .. controls ++(30:.1cm) and ++(-140:.1cm) .. ($ (d) + (.2,-.1) $);
	\draw ($ (d) + (.2,-.1) $) .. controls ++(90:.1cm) and ++(0:.1cm) .. (d);
	\draw[fill=\betacolor] (c) circle (.05cm);
	\draw[fill=\betacolor] (d) circle (.05cm);
	\roundNbox{unshaded, densely dotted}{(f)}{.3}{0}{0}{$g$}
    \draw[very thick, densely dotted, rounded corners=5pt, unshaded] (0,1.8).. controls ++(40:.2cm) and ++(220:.1cm) .. (.3,2) -- (.3,1.7);
    \draw[very thick, rounded corners=3pt, unshaded] (0,1.4) .. controls ++(-40:.4cm) and ++(220:.3cm) .. (1.51,1.4) --  (1.5,1.8) .. controls ++(220:.3cm) and ++(-40:.4cm) .. (0,1.8);
	\node at (.6,1.42) {\scriptsize{$f$}};
	\node at (.2,0) {\scriptsize{$x$}};
	\node at (1.4,-.3) {\scriptsize{$\overline{x}$}};
	\node at (.2,2.5) {\scriptsize{$z$}};
	\node at (1.4,2.2) {\scriptsize{$\overline{z}$}};
	\node at (.63,-.4) {\scriptsize{$\Lambda_1$}};
	\node at (1.05,2.6) {\scriptsize{$\Lambda_2$}};
\node[scale=.9] at (3.2,1.1) {$\forall g\!:\!\Lambda_2\boxtimes a$};
\node[scale=.9] at (3.75,0.7) {$\to a\boxtimes \Lambda_2.$};
\end{tikzpicture}
\end{equation}
It remains to show that $f$ commutes with the half-braiding. % $e_{\Delta(\Lambda_1)}$.
Write $\Lambda_2$ as ${}_R(H_2)\otimes_{\mathbb C} (H_1)_R$, for some right/left $R$-modules $H_1$ and $H_2$.
We then have a canonical isomorphism
\[
\big[ x \boxtimes \Lambda_1 \boxtimes \overline{x} \boxtimes \Lambda_2 \boxtimes - \big]_{\text{cyclic}}
=\,
H_1\boxtimes x \boxtimes \Lambda_1 \boxtimes \overline{x} \boxtimes H_2.
\]
Taking $g$ of the form
\[
\Lambda_2\boxtimes a=H_2\otimes_{\mathbb C} H_1\boxtimes a\stackrel{v\otimes u}{\relbar\joinrel\hspace{-.07mm}\relbar\joinrel\rightarrow} a\boxtimes H_2\otimes_{\mathbb C} H_1
=a\boxtimes \Lambda_2\]
for $R$-module maps $v:H_2\to a\boxtimes H_2$ and $u:H_1\boxtimes a\to H_1$, equation (\ref{eq: [g,f]=0 on cylinder}) becomes:
$$
\sum_{y\in\Irr(\cC)}\,\,
\begin{tikzpicture}[baseline=-.1cm, yscale=-1]
	\draw[very thick] (0,-1.6) -- (0,1.6);
	\draw (-.8,-1.6) -- (-.8,1.6);
	\draw (.8,-1.6) -- (.8,1.6);
	\draw (-.15,-1.6) -- (-.15,1.6);
	\draw (.15,-1.6) -- (.15,1.6);
    \draw (-.65,-.7) .. controls ++(90:.2cm) and ++(225:.2cm) .. (-.15,-.4);
    \draw (.15,-.2) .. controls ++(45:.6cm) and ++(270:.4cm) .. (.65,.7);
	\roundNbox{unshaded}{(0,.5)}{.3}{0}{0}{$f$}
	\roundNbox{unshaded}{(-.8,-1)}{.3}{0}{0}{$u$}
	\roundNbox{unshaded}{(.8,1)}{.3}{0}{0}{$v$}
	\node at (.05-.02,1.8) {\scriptsize{$\Lambda_1$}};
	\node at (.05-.02,-1.8) {\scriptsize{$\Lambda_1$}};
	\draw[fill=\betacolor] (-.15,-.4) circle (.05cm);
	\draw[fill=\betacolor] (.15,-.2) circle (.05cm);
	\node at (-.25-.02,1.82) {\scriptsize{$x$}};
	\node at (.3-.02,1.8) {\scriptsize{$\overline{x}$}};
	\node at (-.25-.02,-1.77) {\scriptsize{$z$}};
	\node at (.3-.02,-1.8) {\scriptsize{$\overline{z}$}};
	\node at (-.8,-1.8) {\scriptsize{$H_1$}};
	\node at (.8,-1.8) {\scriptsize{$H_2$}};
	\node at (-.8,1.8) {\scriptsize{$H_1$}};
	\node at (.8,1.8) {\scriptsize{$H_2$}};
	\node at (-.5,-.4) {\scriptsize{$a$}};
	\node at (.5,-.1) {\scriptsize{$a$}};
\node at (-.27,.05) {\scriptsize{$y$}};%--------
\node at (.27,.05) {\scriptsize{$\overline{y}$}};%--------
\end{tikzpicture}
\,\,\,=\,
\sum_{y\in\Irr(\cC)}\,\,
\begin{tikzpicture}[baseline=-.1cm, yscale=-1]
\node at (-.25,-.01) {\scriptsize{$y$}};%--------
\node at (.27,0) {\scriptsize{$\overline{y}$}};%--------
	\draw[very thick] (0,-1.6) -- (0,1.6);
	\draw (-.8,-1.6) -- (-.8,1.6);
	\draw (.8,-1.6) -- (.8,1.6);
	\draw (-.15,-1.6) -- (-.15,1.6);
	\draw (.15,-1.6) -- (.15,1.6);
    \draw (-.65,-.7) .. controls ++(90:.6cm) and ++(225:.4cm) .. (-.15,.2);
    \draw (.15,.4) .. controls ++(45:.2cm) and ++(270:.2cm) .. (.65,.7);
	\roundNbox{unshaded}{(0,-.5)}{.3}{0}{0}{$f$}
	\roundNbox{unshaded}{(-.8,-1)}{.3}{0}{0}{$u$}
	\roundNbox{unshaded}{(.8,1)}{.3}{0}{0}{$v$}
	\node at (.05-.02,1.8) {\scriptsize{$\Lambda_1$}};
	\node at (.05-.02,-1.8) {\scriptsize{$\Lambda_1$}};
	\draw[fill=\betacolor] (-.15,.2) circle (.05cm);
	\draw[fill=\betacolor] (.15,.4) circle (.05cm);
	\node at (-.5,.1) {\scriptsize{$a$}};
	\node at (.5,.4) {\scriptsize{$a$}};
	\node at (-.25-.02,1.82) {\scriptsize{$x$}};
	\node at (.3-.02,1.8) {\scriptsize{$\overline{x}$}};
	\node at (-.25-.02,-1.77) {\scriptsize{$z$}};
	\node at (.3-.02,-1.8) {\scriptsize{$\overline{z}$}};
	\node at (-.8,-1.8) {\scriptsize{$H_1$}};
	\node at (.8,-1.8) {\scriptsize{$H_2$}};
	\node at (-.8,1.8) {\scriptsize{$H_1$}};
	\node at (.8,1.8) {\scriptsize{$H_2$}};
\end{tikzpicture}
.
$$
This being true for any $u$ and $v$, it follows that
$$
\def\vshorten{.4}
\sum_{y\in\Irr(\cC)}\,\,
\begin{tikzpicture}[baseline=-.1cm, yscale=-1]
	\draw[very thick] (0,-1.6+\vshorten) -- (0,1.6-\vshorten);
	\draw (-.8,-1.6+\vshorten) -- (-.8,1.6-\vshorten);
	\draw (.8,-1.6+\vshorten) -- (.8,1.6-\vshorten);
	\draw (-.15,-1.6+\vshorten) -- (-.15,1.6-\vshorten);
	\draw (.15,-1.6+\vshorten) -- (.15,1.6-\vshorten);
    \draw (.65,1.6-\vshorten) -- (.65,.7) .. controls ++(90:-.6cm) and ++(225:-.4cm) .. (.15,-.2);
    \draw (-.15,-.4) .. controls ++(45:-.4cm) and ++(270:-.6cm) .. (-.65,-1.6+\vshorten);
	\roundNbox{unshaded}{(0,.5)}{.3}{0}{0}{$f$}
	\node at (.05-.02,1.8-\vshorten) {\scriptsize{$\Lambda_1$}};
	\node at (.05-.02,-1.8+\vshorten) {\scriptsize{$\Lambda_1$}};
	\draw[fill=\betacolor] (-.15,-.4) circle (.05cm);
	\draw[fill=\betacolor] (.15,-.2) circle (.05cm);
	\node at (-.25-.02,1.82-\vshorten) {\scriptsize{$x$}};
	\node at (.3-.02,1.8-\vshorten) {\scriptsize{$\overline{x}$}};
	\node at (-.25-.02,-1.77+\vshorten) {\scriptsize{$z$}};
	\node at (.3-.02,-1.8+\vshorten) {\scriptsize{$\overline{z}$}};
	\node at (-.87,-1.8+\vshorten) {\scriptsize{$H_1$}};
	\node at (-.58,-1.77+\vshorten) {\scriptsize{$a$}};
	\node at (.8,-1.8+\vshorten) {\scriptsize{$H_2$}};
	\node at (-.8,1.8-\vshorten) {\scriptsize{$H_1$}};
	\node at (.93,1.8-\vshorten) {\scriptsize{$H_2$}};
	\node at (.6,1.82-\vshorten) {\scriptsize{$a$}};
\node at (-.27,.05) {\scriptsize{$y$}};%--------
\node at (.27,.05) {\scriptsize{$\overline{y}$}};%--------
\end{tikzpicture}
\,=\,
\sum_{y\in\Irr(\cC)}\,\,
\begin{tikzpicture}[baseline=-.1cm, yscale=-1]
\node at (-.25,-.01) {\scriptsize{$y$}};%--------
\node at (.27,0) {\scriptsize{$\overline{y}$}};%--------
	\draw[very thick] (0,-1.6+\vshorten) -- (0,1.6-\vshorten);
	\draw (-.8,-1.6+\vshorten) -- (-.8,1.6-\vshorten);
	\draw (.8,-1.6+\vshorten) -- (.8,1.6-\vshorten);
	\draw (-.15,-1.6+\vshorten) -- (-.15,1.6-\vshorten);
	\draw (.15,-1.6+\vshorten) -- (.15,1.6-\vshorten);
    \draw (-.65,-1.6+\vshorten) -- (-.65,-.7) .. controls ++(90:.6cm) and ++(225:.4cm) .. (-.15,.2);
    \draw (.15,.4) .. controls ++(45:.4cm) and ++(270:.6cm) .. (.65,1.6-\vshorten);
	\roundNbox{unshaded}{(0,-.5)}{.3}{0}{0}{$f$}
	\node at (.05-.02,1.8-\vshorten) {\scriptsize{$\Lambda_1$}};
	\node at (.05-.02,-1.8+\vshorten) {\scriptsize{$\Lambda_1$}};
	\draw[fill=\betacolor] (-.15,.2) circle (.05cm);
	\draw[fill=\betacolor] (.15,.4) circle (.05cm);
	\node at (-.25-.02,1.82-\vshorten) {\scriptsize{$x$}};
	\node at (.3-.02,1.8-\vshorten) {\scriptsize{$\overline{x}$}};
	\node at (-.25-.02,-1.77+\vshorten) {\scriptsize{$z$}};
	\node at (.3-.02,-1.8+\vshorten) {\scriptsize{$\overline{z}$}};
	\node at (-.87,-1.8+\vshorten) {\scriptsize{$H_1$}};
	\node at (-.58,-1.77+\vshorten) {\scriptsize{$a$}};
	\node at (.8,-1.8+\vshorten) {\scriptsize{$H_2$}};
	\node at (-.8,1.8-\vshorten) {\scriptsize{$H_1$}};
	\node at (.93,1.8-\vshorten) {\scriptsize{$H_2$}};
	\node at (.6,1.82-\vshorten) {\scriptsize{$a$}};
\end{tikzpicture}
.
$$
Finally, fusing with $H_1$ and $H_2$ are faithful operations by Lemma \ref{lem: fusion is faithful}, and so the above equation implies
$e_{\Delta(\Lambda_1),a}\circ(f\boxtimes \id_a)=(\id_a\boxtimes f)\circ e_{\Delta(\Lambda_1),a}$, as desired.
\end{proof}

%%%%%%%%%%%%%%%%%%%%%%%%%%%%%%%%%%%%%%%%%%%%%%%%%
%%%%%%%%%%%%%%%%%%%%%%%%%%%%%%%%%%%%%%%%%%%%%%%%%
%%%%%%%%%%%%%%%%%%%%%%%%%%%%%%%%%%%%%%%%%%%%%%%%%
\section{Proof of the main theorem}

Let $\cC$ be a unitary fusion category, and let $\alpha:\cC\to\Bim(R)$ be a fully faithful representation. 
Then $\alpha$ extends to a functor
\begin{align*}
\alpha^{\Hilb}\,:\,\,\cC\otimes_\Vec \Hilb \,&\,\to\,\,\,\Bim(R)
\\
\textstyle
\bigoplus x_i\otimes H_i\,&\mapsto\,
\textstyle
\bigoplus \alpha(x_i)\otimes H_i.
\end{align*}
Here, the first `$-\otimes H_i$' is formal (as defined in Section \ref{sec: Unitary fusion categories}), whereas the second one is evaluated in $\Bim(R)$.

\begin{lem}\label{lem: (C x Hilb)' = C'}
The restriction functor $(\cC\otimes_\Vec \Hilb)'\to \cC'$ is an equivalence.
\end{lem}

\begin{proof}
Given an object $(X,e_X)\in\cC'$, we can extend the half-braiding $e_X=(e_{X,y}:X\otimes y\to y\otimes X)_{y\in \cC}$ 
to arbitrary objects $Y=\bigoplus y_i\otimes H_i$ of $\cC\otimes_\Vec \Hilb$ by
\begin{align*}
e_{X,Y}:X\boxtimes Y = X\boxtimes \Big({\textstyle \bigoplus}\, y_i\otimes H_i\Big)=&\,\bigoplus\, (X\boxtimes y_i)\otimes H_i
\\
\xrightarrow{\bigoplus e_{X,y_i}\otimes \id_{H_i}}
&\,\bigoplus\, (y_i\boxtimes X)\otimes H_i
= \Big({\textstyle \bigoplus}\, y_i\otimes H_i\Big) \boxtimes X
= Y\boxtimes X.
\end{align*}
The half-braiding $e_{X,Y}$ is completely determined from the $e_{X,y_i}$ by naturality, and so
the functor $(\cC\otimes_\Vec \Hilb)'\to \cC'$ is a bijection on objects.
To finish the argument, we note that again by naturality, given two objects $(X_1,e_{X_1})$ and $(X_2,e_{X_2})$ in $\cC'$,
a map $f:X_1\to X_2$ is a morphism $(X_1,e_{X_1})\to (X_2,e_{X_2})$ in $\cC'$ if and only if it is a morphism between the corresponding objects of $(\cC\otimes_\Vec \Hilb)'$.
\end{proof}

\begin{thm*}[Theorem \ref{thm:Main}]
Let $\cC$ be a unitary fusion category and let $\alpha:\cC\to\Bim(R)$ be a fully faithful representation.
Then $\alpha^{\Hilb}$ exhibits $\cC\otimes_\Vec \Hilb$ as a bicommutant category.
\end{thm*}

\begin{proof}
We will show that $\cC''$ is equivalent to $\cC\otimes_\Vec \Hilb$.
The result will then follow since $\cC''=(\cC\otimes_\Vec \Hilb)''$ by Lemma~\ref{lem: (C x Hilb)' = C'}.
We first note that the `inclusion' functor
$\iota:\cC\to \cC''$ (described in Section \ref{sec: Bicommutant categories})
%, given by $e'_{X,(Y,e_Y)}:=e_{Y,X}^{-1}:X\boxtimes_R Y\to Y\boxtimes_R X$ for $(Y,e_Y)\in\cC'$,
extends to a functor
\begin{align*}
\iota^{\Hilb}\,:\,\,\cC\otimes_\Vec \Hilb \,&\,\to\,\,\,\cC''
\\
\textstyle
\bigoplus x_i\otimes H_i\,&\mapsto\,
\textstyle
\bigoplus \iota(x_i)\otimes H_i
\end{align*}
where the first `$-\otimes H_i$' is formal, and the second is evaluated in $\cC''$.
\medskip

\noindent{\it $\bullet$ The functor $\iota^{\Hilb}$ is fully faithful:}\medskip\\\indent
The functor is fully faithful on simple objects, since their images remain simple in $\cC''$. 
Indeed, they remain simple in $\Bim(R)$, and therefore also in $\cC''$.
For finite sums of simple objects, fully faithfulness follows by additivity.
For the remaining objects, we have
\begin{align*}
\Hom_{\cC''} &\Big({\textstyle \bigoplus_i}\, \iota(x_i)\otimes H_i,
{\textstyle \bigoplus_j}\, \iota(y_j)\otimes K_j\Big)=
\bigoplus_{ij} \Hom_{\cC''}\!\big(\iota(x_i),\iota(y_j)\big)\otimes_\C \Hilb(H_i,K_j)
\\
&=
\bigoplus_{ij} \Hom_{\cC}(x_i,y_j)\otimes_\C \Hilb(H_i,K_j)
=\Hom_{\,\cC\otimes_\Vec \Hilb}
\Big({\textstyle \bigoplus_i}\, x_i\otimes H_i,
{\textstyle \bigoplus_j}\, y_j\otimes K_j\Big),
\vspace{-.2cm}
\end{align*}
where we have used the finite dimensionality of $\Hom_{\cC''}(\iota(x_i),\iota(y_j))$
in the first equality.
\bigskip

\noindent{\it $\bullet$ The functor $\iota^{\Hilb}$ is essentially surjective:}\medskip\\\indent
Let $\underline\Omega\in\cC'$ be an absorbing object.
The proof splits into three steps:
\begin{enumerate}[(1)]
\item
If $(X,e_X)$ is an object of $\cC''$, then
its underlying bimodule $X$ lies in $\cC\otimes_\Vec \Hilb$ (the essential image of $\iota^{\Hilb}$).
\item
Let $(X,e_X^{\scriptscriptstyle(1)})$ and  $(X,e_X^{\scriptscriptstyle(2)})\in\cC''$ be two objects with same underlying bimodule $X$.
Then
$e_{X,\underline\Omega}^{\scriptscriptstyle(1)}=e_{X,\underline\Omega}^{\scriptscriptstyle(2)}$.
\item
Given an object $(X,e_X)\in\cC''$, then $e_X=\big(e_{X,\underline Y}:X\boxtimes Y\to Y\boxtimes  X\big)_{\underline Y=(Y,e_Y)\in\cC'}$ is uniquely determined by  $e_{X,\underline\Omega}$.
\end{enumerate}
These are proven in Proposition \ref{prop:AllBimodulesInCBoxtimesHilb},
Proposition \ref{prop:UniqueHalfBraiding}, and
Proposition \ref{prop:DeltaDeterminesHalfBraidings}, respectively.
\end{proof}

\begin{prop}
\label{prop:AllBimodulesInCBoxtimesHilb}
The underlying bimodule of an object of $\cC''$ lies in $\cC\otimes_\Vec \Hilb$.
\end{prop}

\begin{proof}
Let $\Lambda_0:={}_RL^2(R)\otimes L^2(R)_R$ and let $(\Delta_0,e_{\Delta_0}):=\underline{\Delta}(\Lambda_0)$.
Given an object $(X,e_X)$ of $\cC''$, the half-braiding $e_X$ yields a bimodule map
\[
e\,:=\,e_{X,\underline{\Delta}(\Lambda_0)}: X\boxtimes_R \Delta_0\to \Delta_0\boxtimes_R X
\]
which, after rewriting
\begin{align*}
X\boxtimes_R \Delta_0
\,&=
X\boxtimes_R\Big(\bigoplus_{y\in \Irr(\cC)} y\boxtimes_R \Lambda_0\boxtimes_R \overline y
\Big)\\
&=
\bigoplus_{y\in \Irr(\cC)} X\boxtimes_R y\boxtimes_R L^2R\otimes L^2R\boxtimes_R \overline y
\\
&=
\bigoplus_{y\in \Irr(\cC)} \big(X\boxtimes_R y\big)\otimes \overline y\quad\,\,\, \text{and}
\\[.5em]
\Delta_0\boxtimes_R X\,&=
\bigoplus_{y\in \Irr(\cC)} y\otimes \big(\overline y\boxtimes_R X\big),
\end{align*}
becomes a map
\[
%e_{X,\Delta_0}:
e\,:\,\,
\bigoplus_{y\in \Irr(\cC)} \big(X\boxtimes_R y\big)\otimes \overline y \,\to \bigoplus_{y\in \Irr(\cC)} y\otimes \big(\overline y\boxtimes_R X\big).
\]
The Hilbert spaces $(X\boxtimes_R y)\otimes \overline y$ and $y\otimes (\overline y\boxtimes_R X)$ each have four actions of $R$, two left actions, and two right actions:
\[
\begin{tikzpicture}[baseline=-.25cm]
\node at (0,-.2) {$(X\boxtimes_R y)\hspace{.7cm}\otimes\hspace{.7cm} \overline y$};
\draw[->] (-2.2,-.6)node[below, scale=.85, inner sep=4, yshift=1] {\small 1\textsuperscript{st}$R$} to[bend left=13] (-2.2+.25,-.3);
\draw[->] (.1,-.6)node[below, scale=.85, inner sep=4, yshift=1] {\small 2\textsuperscript{nd}$R$} to[bend right=13] (.1-.25,-.3);
\draw[->] (1.28,-.6)node[below, scale=.85, inner sep=4, yshift=1] {\small 3\textsuperscript{rd}$R$} to[bend left=13] (1.28+.25,-.3);
\draw[->] (2.3,-.6)node[below, scale=.85, inner sep=4, yshift=1] {\small 4\textsuperscript{th}$R$} to[bend right=13] (2.3-.25,-.3);
\end{tikzpicture}
\,\,\,\qquad\text{and}\qquad\,\,\,
\begin{tikzpicture}[xscale=-1, baseline=-.25cm]
\node at (0,-.2) {$y \hspace{.7cm}\otimes\hspace{.7cm} (\overline y\boxtimes_R X)$};
\draw[->] (-2.2,-.6)node[below, scale=.85, inner sep=4, yshift=1] {\small 4\textsuperscript{th}$R$} to[bend left=13] (-2.2+.25,-.3);
\draw[->] (.1,-.6)node[below, scale=.85, inner sep=4, yshift=1] {\small 3\textsuperscript{rd}$R$} to[bend right=13] (.1-.25,-.3);
\draw[->] (1.3,-.6)node[below, scale=.85, inner sep=4, yshift=1] {\small 2\textsuperscript{nd}$R$} to[bend left=13] (1.3+.25,-.3);
\draw[->] (2.3,-.6)node[below, scale=.85, inner sep=4, yshift=1] {\small 1\textsuperscript{st}$R$} to[bend right=13] (2.3-.25,-.3);
\end{tikzpicture}
\]
In order to keep track of all these copies of $R$, we denote them $R_1$, $R_2$, $R_3$, $R_4$, respectively.

The map $e$
%$e_{X,\Delta_0}$
is a morphism in $\Bim(R)$, meaning that it is an $R_1$-$R_4$-bimodule map.
This map also has the property of being natural with respect to endomorphisms of $\underline{\Delta}(\Lambda_0)$.
Restricting attention to
\[
\End_{\Bim(R)}(\Lambda_0)= R^{\text{op}}\,\bar\otimes\,R\,\subset\, \End_{\cC'}(\underline\Delta(\Lambda_0)),
\]
this translates into the property of $e$ %$e_{X,\Delta_0}$ 
being an $R_3$-$R_2$-bimodule map
(or rather an $R_2^{\op}$-$R_3^{\op}$-bimodule map).
All in all, we learn that there is an isomorphism of \emph{quadri-modules}:
\[
\bigoplus_{y\in \Irr(\cC)} \big(X\boxtimes_R y\big)\otimes \overline y \,\,\,\cong \bigoplus_{y\in \Irr(\cC)} y\otimes \big(\overline y\boxtimes_R X\big).
\]

Now, applying $\Hom_{R_3,R_4}(L^2R, -)$ to the above isomorphism, we get an $R_1$-$R_2$-bimodule isomorphism:
\begin{align*}
X 
\,&\cong\,
\Hom_{R_3,R_4}\Big(L^2R, \bigoplus_{y\in \Irr(\cC)} (X\boxtimes y) \otimes \overline y\Big)
\\&\cong\,
\Hom_{R_3,R_4}\Big(L^2R, \bigoplus_{y\in \Irr(\cC)} y \otimes (\overline y\boxtimes X)\Big)
\cong
\bigoplus_{y\in \Irr(\cC)} y\otimes \Hom_{\Bim(R)}\big(L^2R, \overline y\boxtimes X\big).
\end{align*}
Finally, $\Hom_{\Bim(R)}(L^2R, \overline y\boxtimes X)$ is just some Hilbert space (because $L^2R$ is irreducible), and so
the above isomorphism exhibits $X$ as an element of $\cC\otimes_\Vec \Hilb$.
\end{proof}

Let $\cC'_{abs} \subset \cC'$ be the full subcategory of absorbing objects of $\cC'$. This is a non-unital tensor category, and it makes sense to talk about half-braidings with $\cC'_{abs}$ (the axioms of a half-braiding never mention unit objects).

\begin{lem}\label{lem: RL half relative commutant}
Let
$\underline\Omega=(\Omega,e_\Omega)\in \cC'$ be an absorbing object,
let $X$ be a right $R$-module, and let
$
u:X\boxtimes\Omega\to X\boxtimes\Omega
$
be a right module map that commutes with
$\id_X\otimes\End_{\cC'}(\underline{\Omega})$.
Then $u=v\boxtimes \id_{\Omega}$ for some right module map $v:X\to X$.
\end{lem}
\begin{proof}
By Theorem \ref{thm:Absorbing}, we can write $\underline \Omega$ as $\underline \Delta(\Lambda)$ for some absorbing bimodule $\Lambda$.
In particular, we then have $\Omega=\bigoplus_{x\in\Irr(\cC)} x\boxtimes \Lambda\boxtimes \overline x$.
Letting $\Lambda_2:={}_RL^2R\otimes_{\mathbb C} X_R$, we can then identify $X\boxtimes \Omega$ with
\[
\bigoplus_{x\in\Irr(\cC)}
\big[ x \boxtimes \Lambda \boxtimes \overline{x} \boxtimes \Lambda_2 \boxtimes - \big]_{\text{cyclic}}
\]
By Proposition \ref{prop: comm on cyc}, 
since $u$ commutes with $\End_{\cC'}(\underline\Delta(\Lambda))$, it lies in $\End_{\cC'}(\underline\Delta(\Lambda_2))$.
Now, we also know that $u$ commutes with
$R^{\op}=\End({}_RL^2R)$.
By Proposition \ref{prop: RelComm inside End_C'(Delta)}, it therefore comes from some element of $\End(X_R)$, which we may call $v$.
In other words, $u=v\boxtimes \id_{\Omega}$.
\end{proof}

\begin{prop}
\label{prop:UniqueHalfBraiding}
An object $X\in \Bim(R)$ admits at most one half-braiding with $\cC'_{abs}$.
\end{prop}
\begin{proof}
Let $e^{\scriptscriptstyle(1)}_X$ and $e^{\scriptscriptstyle(2)}_X$ be two half-braidings.
Given an object $\underline{\Omega}\in \cC'_{abs}$, with underlying bimodule $\Omega\in\Bim(R)$, we need to show that the two maps
$e_1:=e^{\scriptscriptstyle(1)}_{X,\underline \Omega}$
%:X\boxtimes\Omega\to\Omega\boxtimes X$
and $e_2:=e^{\scriptscriptstyle(2)}_{X,\underline \Omega}$
%:X\boxtimes\Omega\to\Omega\boxtimes X$
are equal.
Let $u:=e_2^{-1}\circ e_1$.
The maps $e_1$ and $e_2$ are natural with respect to endomorphisms of $\underline{\Omega}$, and so $u$ commutes with $\id_X\otimes\End_{\cC'}(\underline{\Omega})$.
By Lemma \ref{lem: RL half relative commutant},
we may therefore write it as $u=v\boxtimes \id_{\Omega}$ for some $v\in \End_{\Bim(R)}(X)$.
All in all, we get a commutative diagram
\[
\xymatrix @C=1.6cm{
X\boxtimes\Omega\ar@/^2pc/[rr]^{e_1}
\ar[r]^{v\;\!\boxtimes\;\! \id_{\Omega}}&X\boxtimes\Omega\ar[r]^{e_2}&\Omega\boxtimes X.
}
\]

Fix an isomorphism $\phi: \underline{\Omega}\boxtimes \underline{\Omega} \to \underline{\Omega}$ in $\cC'$,
and let us denote by the same letter the correponding isomorphism $\Omega \boxtimes \Omega\to \Omega$.
By combining the `hexagon' axiom with the statement that the half-braiding is natural with respect to $\phi$, we get the following commutative diagrams (as in the proof of Proposition \ref{prop:DeltaDeterminesHalfBraidings}):
\begin{equation}\label{eq: hex. axiom in proof (2)}
\begin{matrix}\xymatrix{
X \boxtimes \Omega \boxtimes \Omega
\ar[rr]^{e_1\boxtimes \id_\Omega}
\ar[d]^{\id_X\boxtimes\;\! \phi}
&&
\Omega\boxtimes X\boxtimes \Omega
\ar[rr]^{\id_\Omega\boxtimes e_1}
&&
\Omega \boxtimes \Omega\boxtimes X
\ar[d]^{\phi\boxtimes \id_X}
\\
X\boxtimes \Omega
\ar[rrrr]^{e_1}
&&&&
\Omega\boxtimes X
}
\end{matrix}
\end{equation}
and
\begin{equation}\label{eq: hex. axiom in proof (1)}
\begin{matrix}\xymatrix{
X \boxtimes \Omega \boxtimes \Omega
\ar[rr]^{e_2\boxtimes \id_\Omega}
\ar[d]^{\id_X\boxtimes\;\! \phi}
&&
\Omega\boxtimes X\boxtimes \Omega
\ar[rr]^{\id_\Omega\boxtimes e_2}
&&
\Omega \boxtimes \Omega\boxtimes X
\ar[d]^{\phi\boxtimes \id_X}
\\
X\boxtimes \Omega
\ar[rrrr]^{e_2}
&&&&
\Omega\boxtimes X.
}
\end{matrix}
\end{equation}
Horizontally precomposing (\ref{eq: hex. axiom in proof (1)}) with
\[
\xymatrix{
X \boxtimes \Omega \boxtimes \Omega
\ar[d]^{\id_X\boxtimes\;\! \phi}
\ar[rr]^{v\boxtimes \id_{\Omega \boxtimes \Omega}}
&&
X \boxtimes \Omega \boxtimes \Omega
\ar[d]^{\id_X\boxtimes\;\! \phi}
\\
X\boxtimes \Omega
\ar[rr]^{v\boxtimes \id_\Omega}
&&
X\boxtimes \Omega
}
\]
yields the following diagram
\[
\xymatrix{
X \boxtimes \Omega \boxtimes \Omega
\ar[rr]^{e_1\boxtimes \id_\Omega}
\ar[d]^{\id_X\boxtimes\;\! \phi}
&&
\Omega\boxtimes X\boxtimes \Omega
\ar[rr]^{\id_\Omega\boxtimes e_2}
&&
\Omega \boxtimes \Omega\boxtimes X
\ar[d]^{\phi\boxtimes \id_X}
\\
X\boxtimes \Omega
\ar[rrrr]^{e_1}
&&&&
\Omega\boxtimes X.
}
\]
The latter is almost identical to (\ref{eq: hex. axiom in proof (2)}), but for the top right arrow. All maps in sight being isomorphisms, it follows that $\id_\Omega\boxtimes e_1=\id_\Omega\boxtimes e_2$.
At last, by Lemma \ref{lem: fusion is faithful}, %since $\Bim(R)$ has no zero-divisors, 
we conclude that $e_1=e_2$.
\end{proof}

%\[
%\sum_{x,y\in\Irr(\cC)}
%\sqrt{d_xd_y}\,
%\begin{tikzpicture}[baseline=-.1cm, xscale=1.1]
%	\draw[very thick] (0,-.8) -- (0,.8);
%	\draw (.4,-.4) -- (-.4,.4);
%	\draw (.4,-.8) -- (.4,.8);
%	\draw (-.4,-.8) -- (-.4,.8);
%	\draw (-1,-.8) -- (-1,.8);
%	\draw[fill=\betacolor] (-.4,.4) circle (.05cm);
%	\draw[fill=\betacolor] (.4,-.4) circle (.05cm);
%	\draw[fill=black] (0,0)node[left, xshift=1.5, yshift=-2]%{$\scriptstyle_{\!\!\;f}$} +(.05,.05) rectangle +(-.05,-.05);
%	\node at (-.4,1) {\scriptsize{$y$}};
%	\node at (.4,1) {\scriptsize{$\overline{y}$}};
%	\node at (-.4,-1) {\scriptsize{$x$}};
%	\node at (-1,-1) {\scriptsize{$X$}};
%	\node at (-1,1) {\scriptsize{$X$}};
%	\node at (.4,-1) {\scriptsize{$\overline{x}$}};
%	\node at (-.17,.37) {\scriptsize{$a$}};
%	\node at (.23,-.03) {\scriptsize{$a$}};
%	\node at (0,1) {\scriptsize{$\Lambda$}};
%	\node at (0,-1) {\scriptsize{$\Lambda$}};
%\end{tikzpicture}:\,X\boxtimes\Omega\to X\boxtimes\Omega
%\]
%for every object $a\in \Irr(\cC)$ and every bimodule map $f:\Lambda\boxtimes a\to a\boxtimes \Lambda$.

%%%%%%%%%%%%%%%%%%%%%%%%%%%%%%%%%%%%%%%%%%%%%%%%%
\bibliographystyle{amsalpha}
{\footnotesize{
\bibliography{../../../../Documents/research/penneys/bibliography}
}}
\end{document}